\documentclass[a4paper]{amsart}

\usepackage{amsmath,amssymb}
\usepackage{longtable, color}

\usepackage{enumerate}
\usepackage{hyperref}

\newtheorem{theorem}{Theorem}[section]
\newtheorem{lemma}[theorem]{Lemma}
\newtheorem{corollary}[theorem]{Corollary}
\newtheorem{proposition}[theorem]{Proposition}

\newtheorem{construction}[theorem]{Construction}

\theoremstyle{definition}
\newtheorem{example}[theorem]{Example}

\newtheorem{remark}[theorem]{Remark}

\newcommand{\N}{\mathbb N}
\newcommand{\Z}{\mathbb Z}
\newcommand{\R}{\mathbb R}
\newcommand{\Q}{\mathbb Q}

\newcommand{\red}{\text{\rm red}}
\newcommand{\eq}{\text{\rm eq}}
\newcommand{\adj}{\text{\rm adj}}
\newcommand{\mon}{\text{\rm mon}}

\newcommand{\tensor}{\otimes}

 \newcommand{\id}[1]{\mathfrak{#1}}

\newcommand{\m}{\id{m}}
\newcommand{\n}{\id{n}}

\newcommand{\mc}[1]{\mathcal{#1}}

\newcommand{\mcA}{\mc{A}}
\newcommand{\mcB}{\mc{B}}
\newcommand{\mcD}{\mc{D}}

\newcommand{\mcC}{\mathfrak{C}}
\newcommand{\mcT}{\mathcal{T}}
\newcommand{\mcM}{\mathcal{M}}

\newcommand{\Cl}{\mathcal{C}}

 \DeclareMathOperator{\ord}{ord}

\DeclareMathOperator{\spec}{spec} \DeclareMathOperator{\supp}{supp}
 
  \DeclareMathOperator{\End}{End}
\DeclareMathOperator{\rank}{rank} \DeclareMathOperator{\spl}{spl}
\DeclareMathOperator{\add}{add}

\renewcommand{\t}{\, | \,}

\begin{document}

\title[Monoids of modules and arithmetic of direct-sum decompositions]{Monoids of modules and arithmetic of direct-sum decompositions}

\thanks{The first author was a Fulbright-NAWI Graz Visiting Professor in the Natural Sciences and supported by the Austrian-American Education Commission. The second author was supported by the Austrian Science Fund FWF, Project Number P21576-N18}

\author{Nicholas R. Baeth and Alfred Geroldinger}

\address{Department of Mathematics and Computer Science \\ University of Central Missouri \\ Warrensburg, MO 64093, USA; \normalfont{ baeth@ucmo.edu}}

\address{Institut f\"ur Mathematik und Wissenschaftliches Rechnen \\
Karl-Franzens-Universit\"at Graz \\
Heinrichstra\ss e 36\\
8010 Graz, Austria; \normalfont{alfred.geroldinger@uni-graz.at} }



\begin{abstract}
Let $R$ be a (possibly noncommutative) ring and let $\mathcal C$ be a class of finitely generated (right) $R$-modules which is closed under finite direct sums, direct summands, and isomorphisms. Then the set $\mathcal V (\mathcal C)$ of isomorphism classes of modules is a commutative semigroup with operation induced by the direct sum. This semigroup encodes all possible information about direct sum decompositions of modules in $\mathcal C$. If the endomorphism ring of each module in $\mathcal C$ is semilocal, then $\mathcal V (\mathcal C)$ is a Krull monoid. Although this fact was observed nearly a decade ago, the focus of study thus far has been on ring- and module-theoretic conditions enforcing that $\mathcal V(\mathcal C)$ is Krull. If $\mathcal V(\mathcal C)$ is Krull, its arithmetic depends only on the class group of $\mathcal V(\mathcal C)$ and the set of classes containing prime divisors. In this paper we provide the first systematic treatment to study the direct-sum decompositions of modules using  methods from Factorization Theory of Krull monoids. We do this when $\mathcal C$ is the class of finitely generated torsion-free modules over certain one- and two-dimensional commutative Noetherian local rings.
\end{abstract}

\maketitle


\bigskip
\section{Introduction} \label{1}
\bigskip

The study of direct-sum decompositions of finitely generated modules is a classical topic in Module Theory dating back over a century. In the early 1900s, Wedderburn, Remak, Krull, and Schmidt proved unique direct-sum decomposition results for various classes of groups (see \cite{Wed09}, \cite{Rem11}, \cite{Kru25}, and \cite{Sch29}). In the middle of the last century, Azumaya \cite{Azu50} proved uniqueness of (possibly infinite) direct-sum decomposition of modules provided that each indecomposable module has a local endomorphism ring. In the commutative setting, Evans \cite{Eva73} gave an example due to Swan illustrating a non-unique direct-sum decomposition of a finitely-generated module over a local ring. The past decade has seen a new semigroup-theoretical approach. This approach was first introduced by Facchini and Wiegand \cite{Fa-Wi04} and has been used by several authors (for example, see \cite{Ba07b}, \cite{Ba09a}, \cite{Ba-Lu11a}, \cite{Ba-Sa12a}, \cite{Di07a}, \cite{Fa02}, \cite{Fa06a}, \cite{Fa12a}, \cite{Fa-HK03}, \cite{F-H-K-W06}, \cite{Fa-Wi04}, \cite{H-K-K-W07}, \cite{He-Pr10a}, and \cite{Le-Od96a}). Let $R$ be a ring and let $\mathcal C$ be a class of right $R$-modules which is closed under finite direct sums, direct summands, and isomorphisms. For a module $M$ in $\mathcal C$, let $[M]$ denote the isomorphism class of $M$. Let $\mathcal V (C)$ denote the set of isomorphism classes of modules in $\mathcal C$. (We assume here that $\mathcal V(C)$ is indeed a set, and note that this hypothesis holds for all examples we study.) Then $\mathcal V (\mathcal C)$ is a commutative semigroup with operation defined by $[M] + [N] = [M \oplus N]$ and all information about direct-sum decomposition of modules in $\mathcal C$ can be studied in terms of factorization of elements in the semigroup $\mathcal V (\mathcal C)$. In particular, the direct-sum decompositions in $\mathcal C$ are (essentially) unique (in other words, the Krull-Remak-Schmidt-Azumaya Theorem --- briefly, KRSA --- holds) if and only if $\mathcal V (C)$ is a free abelian monoid. This semigroup-theoretical point of view was justified by Facchini \cite{Fa02} who showed that $\mathcal V (\mathcal C)$ is a reduced Krull monoid provided that the endomorphism ring $\End_R (M)$ is semilocal for all modules $M$ in $\mathcal C$. This result allows one to describe the direct-sum decomposition of modules in terms of factorization of elements in Krull monoids, a well-studied class of commutative monoids.

\smallskip
However, thus far much of the focus in this direction has been on the study of module-theoretic conditions which guarantee that all endomorphism rings are semilocal, as well as on trying to describe the monoid $\mathcal V(\mathcal C)$ in terms of various ring- and module-theoretic conditions. Although some factorization-theoretic computations have been done in various settings (eg. the study of elasticity in \cite{Ba09a}, \cite{Ba-Lu11a}, and \cite{Ba-Sa12a} and the study of the $\omega$-invariant in \cite{Di07a}), the general emphasis has not been on the arithmetic of the monoid $\mathcal V(\mathcal C)$. Our intent is to use known module-theoretic results along with factorization-theoretic techniques in order to give detailed descriptions of the arithmetic of direct-sum decompositions of finitely generated torsion-free modules over certain one- and two-dimensional local rings. We hope that this systematic approach will not only serve to inspire others to consider more detailed and abstract factorization-theoretic approaches to the study of direct-sum decompositions, but to provide new and interesting examples for zero-sum theory over torsion-free groups. We refer to \cite{Fa03a} and to the opening paragraph in the recent monograph \cite{Le-Wi12a} for broad information on the Krull-Remak-Schmidt-Azumaya Theorem, and to two recent surveys \cite{Fa12a} and \cite{Ba-Wi13a} promoting this semigroup-theoretical point of view. More details and references will be given in Section \ref{4}.

\smallskip
Krull monoids, both their ideal theory and their arithmetic, are well-studied; see \cite{Ge-HK06a} for a thorough treatment. A reduced Krull monoid is uniquely determined (up to isomorphism) by its class group $G$, the set of classes $G_{\mathcal P} \subset G$ containing prime divisors, and the number of prime divisors in each class. Let $\mathcal V (C)$ be a monoid of modules and suppose $\mathcal V(\mathcal C)$ is Krull with class group $G$ and with set of classes containing prime divisors $G_{\mathcal P}$. We are interested in determining what this information tells us about direct-sum decompositions of modules. Let $M$ be a module in $\mathcal C$ and let $M = M_1 \oplus \cdots \oplus M_{\ell}$ where $M_1, \ldots, M_{\ell}$ are indecomposable right $R$-modules. Then $\ell$ is called the length of this factorization (decomposition into indecomposables), and the set of lengths $\mathsf L (M) \subset \N$  is defined as the set of all possible factorization lengths. Then KRSA holds if and only if $|G|=1$. Moreover, it is easy to check that $|\mathsf L (M)|=1$ for all $M$ in $\mathcal C$ provided that $|G| \le 2$. Clearly, sets of lengths are a measure how badly KRSA fails. Assuming that $\mathcal V (\mathcal C)$ is Krull, $M$ has at least one direct-sum decomposition in terms of indecomposable right $R$-modules, and, up to isomorphism, only finitely many distinct decompositions. In particular, all sets of lengths are finite and nonempty. Without further information about the class group $G$ and the subset $G_{\mathcal P}\subset G$, this is all that can be said. Indeed, there is a standing conjecture that for every infinite abelian group $G$ there is a Krull monoid with class group $G$ and set $G_{\mathcal P}$ such that every set of lengths has cardinality one (see \cite{Ge-Go03}). On the other hand, if the class group of a Krull monoid is infinite and every class contains a prime divisor, then every finite subset of $\mathbb N_{\ge 2}$ occurs as a set of lengths (see Proposition \ref{7.2}).

\smallskip
Thus an indispensable prerequisite for the study of sets of lengths (and other arithmetical invariants) in Krull monoids is detailed information about not only the class group $G$, but also on the set $G_{\mathcal P} \subset G$ of classes containing prime divisors. For the monoid $\mathcal V (\mathcal C)$, this is of course a module-theoretic task which depends on both the ring $R$ and the class $\mathcal C$ of $R$-modules. Early results gave only extremal sets $G_{\mathcal P}$ and thus no further arithmetical investigations were needed.
In Sections \ref{5} and \ref{6} we determine, based on deep module-theoretic results,   the class group $G$ of $\mathcal V(\mathcal C)$. We then exhibit well-structured sets $G_{\mathcal P}$ providing a plethora of arithmetically interesting direct-sum decompositions.
In particular, we study the classes of finitely generated modules, finitely generated torsion-free modules, and maximal Cohen-Macaulay modules over one- and two-dimensional commutative Noetherian local rings. We restrict, if necessary, to specific families of rings in order to obtain explicit results for $G_{\mathcal P}$, since it is possible that even slightly different sets $G_{\mathcal P}$ can induce completely different behavior in terms of the sets of lengths. Given this information, we use transfer homomorphisms, a key tool in Factorization theory and introduced in Section \ref{4}, which make it possible to study sets of lengths and other arithmetical invariants of general Krull monoids instead in an associated monoid of zero-sum sequences (see Lemma \ref{4.4}).  These monoids can be studied using methods from Additive (Group and Number) Theory (see \cite{Ge09a}).

\smallskip
Factorization Theory describes the non-uniqueness of factorizations of elements in rings and semigroups into irreducible elements by arithmetical invariants such as sets of lengths, catenary, and tame degrees. We will define each of these invariants in Section \ref{3}. The goal is to relate the arithmetical invariants with algebraic parameters (such as class groups) of the objects under consideration. The study of sets of lengths in Krull monoids is a central topic in Factorization Theory. However, since much of this theory was motivated by examples in number theory (such as holomorphy rings in global fields), most of the focus so far has been on Krull monoids with finite class group and with each class containing a prime divisor. This is in contrast to Krull monoids stemming from Module Theory which often have infinite class group (see Section \ref{5}). A key result in Section \ref{7} shows that the arithmetic of these two types of Krull monoids can have drastically different arithmetic.

\smallskip
In combination with the study of various arithmetical invariants of a given Krull monoid, the following dual question has been asked since the beginning of Factorization Theory: Are arithmetical phenomena characteristic for a given Krull monoid (inside a given class of Krull monoids)? Affirmative answers have been given for the class of Krull monoids with finitely generated class groups where every class contains a prime divisor. Since sets of lengths are the most investigated invariants in Factorization Theory, the emphasis in the last decade has been on the following question:
Within the class of Krull monoids having finite class group and such that every class contains a prime divisor, does the system of sets of lengths of a monoid $H$ characterize the class group of $H$?
A survey of these problems can be found in \cite[Sections 7.1 and 7.2]{Ge-HK06a}. For recent progress, see \cite{Sc09b}, \cite{Sc09c}, and \cite{B-G-G-P13a}. In Corollary \ref{7.9} we exhibit that for many Krull monoids stemming from the Module Theory of Sections \ref{5} and \ref{6},  the system of sets of lengths and the behavior of absolutely irreducible elements  characterizes the class group of these monoids.

\smallskip
In Section \ref{3} we introduce some of the main arithmetical invariants studied in Factorization Theory as well as their relevance to the study of direct-sum decompositions. Our focus is on sets of lengths and on parameters controlling their structure, but we will also need other invariants such as catenary and tame degrees. Section \ref{4} gives a brief introduction to Krull monoids, monoids of modules, and transfer homomorphisms. Sections \ref{5} and \ref{6} provide explicit constructions stemming from Module Theory of class groups and distribution of prime divisors in the classes. Finally, in Section \ref{7}, we present our results on the arithmetic of direct-sum decomposition in the Krull monoids discussed in Sections \ref{5} and \ref{6}.

\smallskip
We use standard notation from Commutative Algebra and Module Theory (see \cite{Le-Wi12a}) and we follow the notation of \cite{Ge-HK06a} for Factorization Theory. All monoids of modules $\mathcal V (\mathcal C)$ are written additively, while all abstract Krull monoids are written multiplicatively. This follows the tradition in Factorization Theory, and makes sense also here because our crucial tool, the monoid of zero-sum sequences, is written multiplicatively. In particular, our arithmetical results in Section \ref{7} are written in a multiplicative setting but they are derived for the additive monoids of modules discussed in Sections \ref{5} and \ref{6}. Since we hope that this article is readable both for those working in Ring and Module Theory as well as those working in Additive Theory and Factorization Theory, we often recall concepts of both areas which are well known to the specialists in the respective fields.

\bigskip
\section{Arithmetical Preliminaries} \label{3}
\bigskip

In this section we gather together the concepts central to describing the arithmetic of non-factorial monoids. In particular, we exhibit the arithmetical invariants which will be studied in Section \ref{7} and which will give a measure of non-unique direct sum decompositions of classes of modules studied in Sections \ref{5} and \ref{6}. When possible, we recall previous work in the area of  direct-sum decompositions for which certain invariants have been studied. For more details on non-unique factorization, see \cite{Ge-HK06a}. First we record some preliminary terminology.

\medskip
\noindent
{\bf Notation.}\label{2} We denote by $\N$ the set of
positive integers and set $\N_0 = \N \cup \{0\}$. For every $n \in \mathbb N$, $C_n$ denotes a cyclic group of order $n$. For real
numbers $a, b \in \R$ we set $[a, b] = \{ x \in \Z \, \colon a \le x \le
b \}$. We use the convention that $\sup \emptyset = \max \emptyset = \min \emptyset = 0$.

\medskip
\noindent
{\bf Subsets of the integers.}
Let $L, L' \subset \Z$.  We denote by $L+L' = \{a+b \, \colon
a \in L,\, b \in L' \}$ the {\it sumset} of $L$ and $L'$. If $\emptyset \ne L
\subset \N$, we call
\[
\rho (L) = \sup \Bigl\{ \frac{m}{n} \; \colon \; m, n \in L
\Bigr\} = \frac{\sup L}{\min L} \, \in \Q_{\ge 1} \cup \{ \infty
\}
\]
the \ {\it elasticity} \ of $L$. In addition,  we set
$\rho (\{0\}) = 1$.
Distinct elements $k, l \in L$ are called \emph{adjacent} if $L \cap [\min\{k,l\},\max\{k,l\}]=\{k,l\}$.
A positive integer $d \in \N$ is called a \ {\it distance} \ of $L$ \ if there exist adjacent elements $k,l \in L$ with $d=|k-l|$.  We denote by \ $\Delta (L)$ \ the {\it set of
      distances} of $L$. Note that $\Delta (L) = \emptyset$ if and only if $|L| \le 1$,
and that $L$  is an arithmetical progression with difference $d \in \N$ if and only if $\Delta (L)
\subset \{d\}$.

\medskip
\noindent
{\bf Monoids and Rings.}
By a {\it monoid} $H$ we always mean a commutative semigroup with
identity $1$ which satisfies the cancellation law; that is, if $a$, $b$, and $c$
are elements of the $H$ with $ab = ac$, then $b = c$.

Let $H$ be a monoid. We denote by $\mathcal A (H)$ the set of atoms (irreducible elements) of $H$, by $\mathsf q (H)$ a quotient group of $H$ with $H \subset \mathsf q (H) = \{ a^{-1}b  \, \colon \, a, b \in H \}$, and by  $H^{\times}$  the set of
invertible elements of $H$. We say that $H$ is  {\it reduced}
if $H^{\times} = \{1\}$, and we denote by  $H_{\red} = H/H^{\times} = \{
aH^{\times} \, \colon a \in H \}$  the associated reduced monoid.
 Let $H' \subset
H$ be a subset. We say that $H'$ is {\it divisor-closed} if $a \in H'$ and $b \in H$ with $b \t a$ implies that $b \in H'$. Denote by $[H'] \subset H$ the submonoid generated by $H'$.

A monoid $F$ is called \ {\it free abelian with basis $\mathcal P \subset
F$} \ if every $a \in F$ has a unique representation of the form
\[
a = \prod_{p \in \mathcal P} p^{\mathsf v_p(a) } \quad \text{with} \quad
\mathsf v_p(a) \in \N_0 \ \text{ and } \ \mathsf v_p(a) = 0 \ \text{
for almost all } \ p \in \mathcal P \,.
\]
If $F$ is free abelian with basis $\mathcal P$, we set $F = \mathcal F(\mathcal P)$ and call
\[
 |a| = \sum_{p \in \mathcal P} \mathsf v_p (a) \quad \text{the \ {\it
length}} \ \text{of} \ a \  \quad \text{and} \quad \supp (a) = \{p \in \mathcal P \, \colon \, \mathsf v_p (a) > 0 \} \quad \text{the {\it support} of} \ a \,.
\]
The multiplicative monoid $\mathcal F (\mathcal P)$ is, of course,  isomorphic to the additive monoid $(\N_0^{(\mathcal P)}, +)$.

Throughout this manuscript, all rings have a unit element and, apart from a a few motivating remarks in Section \ref{4}, all rings are commutative. Let $R$ be a ring.  Then we let $R^{\bullet} = R \setminus \{0\}$ denote the nonzero elements of $R$ and let $R^{\times}$ denote its group of units.
 Note that if $R$ is a domain, then
$R^{\bullet}$ is a monoid as defined above. By the dimension of a ring we always mean its Krull dimension.

\medskip
\noindent
{\bf Abelian groups.} Let $G$ be an additive abelian group and let $G_0 \subset G$ a subset.
Then $-G_0 = \{-g \, \colon g \in G_0\}$, $G_0^{\bullet} = G_0 \setminus \{0\}$,  and
$\langle G_0 \rangle \subset G$ denotes the subgroup generated by
$G_0$. A family $(e_i)_{i \in I}$ of elements of $G$ is said to be \
{\it independent} \ if $e_i \ne 0$ for all $i \in I$ and, for every
family $(m_i)_{i \in I} \in \Z^{(I)}$,
\[
\sum_{i \in I} m_ie_i =0 \qquad \text{implies} \qquad m_i e_i =0 \quad \text{for all} \quad i \in I\,.
\]
The family $(e_i)_{i \in I}$ is called a {\it basis} for $G$ if
 $G = \bigoplus_{i \in I} \langle e_i \rangle$. The {\it total rank} $\mathsf r^* (G)$ is the supremum of the cardinalities of independent subsets of $G$. Thus $\mathsf r^* (G) = \mathsf r_0 (G) + \sum_{p \in \mathbb P} \mathsf r_p (G)$, where $\mathsf r_0 (G)$ is the torsion-free rank of $G$ and $\mathsf r_p (G)$ is the $p$-rank of $G$ for every prime $p \in \mathbb P$.

\medskip
\noindent
{\bf Factorizations.}
Let $H$ be a monoid. The free abelian monoid \ $\mathsf Z (H) = \mathcal F \bigl( \mathcal
A(H_\red)\bigr)$ \ is called the \ {\it factorization monoid} \ of
$H$, and the unique homomorphism
\[
\pi \colon \mathsf Z (H) \to H_{\red} \quad \text{satisfying} \quad
\pi (u) = u \quad \text{for each} \quad u \in \mathcal A(H_\red)
\]
is called the \ {\it factorization homomorphism} \ of $H$. For $a
\in H$ and $k \in \N$,

\begin{itemize}
\item $\mathsf Z_H (a) = \mathsf Z (a)  = \pi^{-1} (aH^\times) \subset
\mathsf Z (H) \quad
\text{is the \ {\it set of factorizations} \ of \ $a$}$,
\item $\mathsf Z_k (a)  = \{ z \in \mathsf Z (a) \, \colon |z| = k \} \quad
\text{is the \ {\it set of factorizations} \ of \ $a$ of length \
$k$}$,
\item $\mathsf L_H (a) = \mathsf L (a)  = \bigl\{ |z| \, \colon \, z \in
\mathsf Z (a) \bigr\} \subset \N_0 \quad \text{is the \ {\it set of
lengths} \ of $a$},  \quad \text{and}$
\item $\mathcal L (H)  = \{ \mathsf L (b) \, \colon b \in H \} \quad \text{is the \ {\it system of sets of lengths} of $H$}$.
\end{itemize}

By definition, we have \ $\mathsf Z(a) = \{ 1 \}$ and $\mathsf L (a) =\{0\}$ for all $a \in H^\times$. If $H$ is assumed to be Krull, as is the case in the monoids of modules we study, and $a \in H$, then the set of factorizations $\mathsf Z(a)$ is finite and nonempty and hence $\mathsf L(a)$ is finite and nonempty. Suppose that there is $a \in H$ with $|\mathsf L(a)|>1$ with distinct $k,l \in \mathsf L(a)$. Then for all $N \in \N$, $\mathsf L \big(a^N \big) \supset \{(N-i)k+il\, \colon\, i \in [0,N]\}$ and hence $|\mathsf L \big(a^N \big)|>N$. Thus, whenever there is an element $a\in H$ that has at least two  factorizations of distinct lengths, there exist elements of $H$ having arbitrarily many factorizations of  distinct lengths. This motivates the need for more refined measures of nonunique factorization.

Several invariants such as elasticity and the $\Delta$-set measure non-uniqueness in terms of sets of lengths. Other invariants such as the catenary degree provide an even more subtle measurement in terms of the distinct factorizations of elements. However, these two approaches cannot easily be separated and it is often the case that a factorization-theoretical invariant is closely related to an invariant of the set of lengths. Thus the exposition that follows will introduce invariants as they are needed and so that the relations between these invariants can be made as clear as possible.

The monoid $H$ is called
\begin{itemize}
\item  {\it atomic} \  if  \ $\mathsf Z(a) \ne \emptyset$ \
      for all \ $a \in H$,

\item {\it factorial} \ if \ $|\mathsf Z (a)| = 1$ \ for all $a \in
H$ (equivalently, $H_\red$ is free abelian), and

\item {\it half-factorial} \ if \ $|\mathsf L (a)| = 1$ \ for all $a \in H$.
\end{itemize}
Let $z,\, z' \in \mathsf Z (H)$. Then we can write
\[
z = u_1 \cdot \ldots \cdot u_lv_1 \cdot \ldots \cdot v_m \quad
\text{and} \quad z' = u_1 \cdot \ldots \cdot u_lw_1 \cdot \ldots
\cdot w_n\,,
\]
where  $l,\,m,\, n\in \N_0$ and $u_1, \ldots, u_l,\,v_1, \ldots,v_m,\,
w_1, \ldots, w_n \in \mathcal A(H_\red)$ are such that
\[
\{v_1 ,\ldots, v_m \} \cap \{w_1, \ldots, w_n \} = \emptyset\,.
\]
Then $\gcd(z,z')=u_1\cdot\ldots\cdot u_l$, and we call
\[
\mathsf d (z, z') = \max \{m,\, n\} = \max \{ |z \gcd (z, z')^{-1}|,
|z' \gcd (z, z')^{-1}| \} \in \N_0
\]
the {\it distance} between $z$ and $z'$. If $\pi (z) = \pi (z')$ and $z
\ne z'$, then clearly
\begin{equation*}\label{E:Dist}
2 +
      \bigl| |z |-|z'| \bigr| \le \mathsf d (z, z') \,.
\end{equation*}
For subsets $X, Y \subset \mathsf Z (H)$,
we set
\[
\mathsf d (X, Y) = \min \{ \mathsf d (x, y ) \, \colon x \in X, \, y \in
Y \} \,,
\]
and thus $\mathsf d (X, Y) =
0$ if and only if ($X \cap Y \ne \emptyset$, $X = \emptyset$, or $Y = \emptyset$).

\bigskip
\noindent
From this point on, we will assume all monoids to be atomic. Since the the monoids described in Sections \ref{5} and \ref{6} are of the form $\mathcal V(\mathcal C)$ for $\mathcal C$ a subclass of finitely generated modules over a commutative Noetherian ring, they are  Krull and hence atomic.
\smallskip

\medskip
\noindent
{\bf The set of distances and chains of factorizations.} We now recall the $\Delta$-set of a monoid $H$, an invariant which describes the sets of lengths of elements in $H$,  and illustrate its relationship with distances between factorizations of elements in $H$. We denote by
\[
\Delta (H) \ = \ \bigcup_{L \in \mathcal L (H) } \Delta (L) \ \subset \N
\]
the {\it set of distances} of $H$. By definition, $\Delta (H) = \emptyset$ if and only if $H$ is half-factorial. For a more thorough investigation of factorizations in $H$, we will need a distinguished subset of the set of distances. Let $\Delta^* (H)$ denote the set of all $d = \min \Delta (S)$ for some divisor-closed submonoid $S \subset H$ with $\Delta (S) \ne \emptyset$. By definition, we have $\Delta^* (H) \subset \Delta (H)$.

Suppose that $H$ is not factorial. Then there exists an element $a \in H$ with $|\mathsf Z(a)|>1$, and so there exist distinct $z,z'\in \mathsf Z(a)$. Then for $N\in \N$, $\mathsf Z \big(a^N \big) \supset \{z^{N-i}(z')^i\, \colon\, i \in [0,N]\}$. Although $\mathsf d \big(z^N, (z')^N \big)=Nd(z,z')>N$ suggests that the factorizations $z^N$ and $(z')^N$ of $a^N$ are very different,  $$\mathsf d \big(z^{N-i}(z')^i,z^{N-i+1}(z' )^{i-1} \big)= \mathsf d(z,z')$$ for each $i \in [1,N]$. This illustrates that the distance alone is too coarse of an invariant, and motivates the study of the catenary degree as a way of measuring how distinct two factorizations are. As will be described below, there is a structure theorem for the set of lengths of a Krull monoid. However, except in very simple situations, there is no known structure theorem for the set of factorizations of an element in a Krull monoid. Thus we use the catenary degree, its many variations, the tame degree, and other invariants help to measure the subtle distinctions between factorizations.

Let $a \in H$ and
$N \in \mathbb N_0 \cup \{\infty\}$. A finite sequence $z_0, \ldots,
z_k \in \mathsf Z (a)$ is called a {\it $($monotone$)$ $N$-chain of
factorizations} if $\mathsf d (z_{i-1}, z_i) \le N$ for all $i \in
[1, k]$ and ($|z_0| \le \cdots \le |z_k|$ or $|z_0| \ge \cdots \ge
|z_k|$ resp.). We denote by  $\mathsf c (a)$ (or by $\mathsf c_{\mon} (a)$ resp.)  the smallest $N \in \N _0 \cup \{\infty\}$ such
      that any two factorizations $z,\, z' \in \mathsf Z (a)$ can be
      concatenated by an $N$-chain (or by a monotone $N$-chain resp.).
Then
\[
\mathsf c(H) = \sup \{ \mathsf c(b) \, \colon b \in H\} \in \N_0 \cup
\{\infty\} \quad \text{and} \quad \mathsf c_{\mon} (H) = \sup \{
\mathsf c_{\mon} (b) \, \colon b \in H\} \in \N_0 \cup \{\infty\} \quad
\,
\]
denote  the \ {\it catenary degree} \ and the \ {\it monotone
catenary degree} of $H$. The monotone catenary degree is studied by
using the two auxiliary notions of the equal and the adjacent
catenary degrees. Let $\mathsf c_{\eq} (a)$ denote the smallest $N
\in \mathbb N_0 \cup \{\infty\}$ such that any two factorizations $z, z' \in \mathsf Z (a)$ with $|z| = |z'|$ can be concatenated by a monotone $N$-chain.
We call
      \[
      \mathsf c_{\eq} (H) = \sup \{ \mathsf c_{\eq} (b) \, \colon b \in H \} \in \mathbb N_0 \cup \{\infty\}
      \]
      the {\it equal catenary degree} of $H$.
We set
      \[
      \mathsf c_{\adj}(a)  = \sup \{ \mathsf d \big( \mathsf Z_k (a),
      \mathsf Z_l (a) \big) \, \colon k, l \in \mathsf L (a) \ \text{are
      adjacent} \} \,,
      \]
      and the {\it adjacent catenary degree} of $H$ is defined as
      \[
      \mathsf c_{\adj} (H) = \sup \{ \mathsf c_{\adj} (b) \, \colon b \in H \} \in \mathbb N_0 \cup
      \{\infty\} \,.
      \]
Obviously,  we have
\[
\mathsf c (a) \le \mathsf c_{\mon} (a) = \sup \{ \mathsf c_{\eq}
(a), \mathsf c_{\adj} (a) \} \le \sup \mathsf L (a) \quad \text{for
all} \quad a \in H \,,
\]
and hence
\begin{equation*}
\mathsf c (H) \le \mathsf c_{\mon} (H) = \sup \{ \mathsf c_{\eq}
(H), \mathsf c_{\adj} (H) \} \,. \label{basic2}
\end{equation*}
Note that $\mathsf c_{\adj} (H) = 0$ if and only if $H$ is half-factorial, and if $H$ is not half-factorial, then $2 + \sup \Delta (H) \le \mathsf c (H)$. Moreover,  $\mathsf c_{\eq} (H) = 0$ if and only if for all $a \in H$ and all $k \in \mathsf L (a)$ we have $|\mathsf Z_k (a)| = 1$. A
recent result of Coykendall and Smith implies that
  if $D$ is a domain, we have that $\mathsf c_{\eq} (D^{\bullet}) = 0$ if and only if $D^{\bullet}$ is factorial \cite[Corollary 2.12]{Co-Sm11a}.

We call
\[
\sim_{H, \eq} = \big\{ (x,y) \in \mathsf Z (H) \times \mathsf Z (H)
\, \colon \pi (x) = \pi (y) \ \text{and} \ |x| = |y| \big\}
\]
the {\it monoid of equal-length relations} of $H$.
Let $Z \subset \mathsf Z (H)$ be a subset. We say that an element $x \in
Z$ is {\it minimal} in $Z$ if for all
elements $y \in Z$ with $y \t x$ it follows that $x =
y$. We denote by $\text{\rm Min} \bigl( Z \bigr)$ the
{\it set of minimal elements} in $Z$. Let $x \in
Z$. Since the number of elements $y \in Z$
with $y \t x$ is finite, there exists an $x^* \in \text{\rm Min}
\bigl( Z \bigr)$ with $x^* \t x$. We will need the following lemma.

\smallskip
\begin{lemma} \label{3.1}
Let $H$ be an atomic monoid.
\begin{enumerate}[(1)]
\item $\mathsf c_{\eq} (H) \le \sup \big\{ |x| \, \colon (x,y) \in \mathcal A (\sim_{H, \eq}) \ \text{for some} \ y \in \mathsf Z (H)\setminus \{x\} \big\}$.

\smallskip
\item For $d \in \Delta (H)$ let $A_d = \big\{ x \in \mathsf Z (H) \, \colon |x|-d \in \mathsf L ( \pi (x) \big) \big\}$.  Then \newline
      $\mathsf c_{\adj} (H) \le \sup \{ |x| \, \colon x \in \text{\rm Min} (A_d), \,  d \in \Delta (H) \}$.
\end{enumerate}
\end{lemma}

\begin{proof}
See \cite[Proposition 4.4]{Bl-Ga-Ge11a}.
\end{proof}

\medskip
\noindent
{\bf Unions of sets of lengths and the refined elasticities.} We now return to studying sets of lengths. We note that the elasticity of certain monoids of modules were studied in \cite{Ba-Lu11a} and \cite{Ba-Sa12a}, but that in Section \ref{7} we will provide results which generalize these results to larger classes of Krull monoids. In addition, we will fine tune these results by also computing the refined elasticities. Let $k, l  \in \mathbb N$. If $H \ne H^{\times}$, then
\[
      \mathcal U_k (H) \ = \ \bigcup_{k \in L,  L \in \mathcal L (H)} L
\]
is the union of all sets of lengths containing $k$.
In other words,
$\mathcal U_k (H)$ is set of all  $m \in \mathbb N$  for which there exist  $u_1,
\ldots, u_k, v_1, \ldots, v_m \in \mathcal A (H)$ with
$u_1 \cdot \ldots \cdot u_k = v_1 \cdot \ldots \cdot v_m$.
When $
H^\times=H$, we set $\mathcal U_k (H) =\{k\}$. In both cases, we define
$\rho_k (H) = \sup \mathcal U_k (H) \in \N \cup \{\infty\}$  and $\lambda_k (H)= \min
\mathcal U_k (H) \in [1,k]$. Clearly, we have $\mathcal U_1 (H) =\{1\}$, $k
\in \mathcal U_k (H)$, and since $\mathcal U_k (H) + \mathcal
      U_l (H) \ \subset \ \mathcal U_{k+l} (H)$, it follows that
      \[
      \lambda_{k+l} (H) \le \lambda_k (H) + \lambda_l (H) \le k+l \le \rho_k
      (H) + \rho_l (H) \le \rho_{k+l} (H) \,.
      \]
The {\it elasticity} $\rho (H)$ of $H$ is defined as
\[
\rho (H)  = \sup \{ \rho (L) \, \colon
L \in \mathcal L(H) \} \in \R_{\ge 1} \cup \{ \infty \} \,,
\]
and it is not difficult to verify that
\[
\rho (H) = \sup \Bigl\{ \frac{\rho_{k} (H)}{k} \, \colon k \in \N
\Bigr\} = \lim_{k \to \infty}\frac{\rho_k(H)}{k}\,.
\]

\medskip
\noindent
{\bf The structure of sets of lengths.} To describe the structure of sets of lengths and of their unions, we need the concept of
arithmetical progressions as well as  various generalizations.
Let $l, M \in \mathbb N_0$, $d \in \N$,  and  $\{0,d\} \subset \mathcal D
\subset [0,d]$. We set
\[
P_l (d) = d \Z \cap [0,\, ld] = \{ 0,\,d,\, 2d,\, \ldots , \,ld\}\,.
\]
Thus a  subset $L \subset \mathbb Z$ is an  arithmetical progression
       (with \ {\it  difference} \ $d \in \N$  and  {\it length}  $l \in \N_0$)  if  $L = \min L + P_l (d)$.
A subset  $L \subset \Z$  is called an {\it almost
arithmetical multiprogression} \ ({\rm AAMP} \ for
      short) \ with \ {\it difference} \ $d$, \ {\it period} \ $\mathcal D$,
      \  and \ {\it bound} \ $M$, \ if
\[
L = y + (L' \cup L^* \cup L'') \, \subset \, y + \mathcal D + d \Z \qquad \text{where}
\]
\begin{itemize}
\item  $L^*$ is finite and nonempty with $\min L^* = 0$ and $L^* =
       (\mathcal D + d \Z) \cap [0, \max L^*]$,

\item  $L' \subset [-M, -1]$ \ and \ $L'' \subset \max L^* + [1,M]$, and
\item      $y \in \Z$.
\end{itemize}
Note that an AAMP is finite and nonempty, and that an AAMP with period $\{0,d\}$ and bound $M=0$ is a (usual) arithmetical progression with difference $d$.

\medskip
\noindent
{\bf The $\omega$-invariant and the tame degrees.} We now study the $\omega$-invariant as well as local and global tame degrees. We note that these notions have been studied in specific noncommutative module-theoretic situations in terms of the so-called {\it semi-exchange property} (see \cite{Di07a}). Moreover,  when describing the sets of lengths of elements within a Krull monoid $H$ in terms of AAMPs (see Proposition \ref{7.2}), the bound $M$ (described above) is a tame degree related to the monoid $H$. We begin with the definition.
For  an atom $u \in H$,  let  $\omega (H,u)$ \ denote the
      smallest  $N \in \N \cup \{\infty\}$  having the following property{\rm \,:}
      \begin{enumerate}
      \smallskip
      \item[] For any multiple $a$ of $u$ and any factorization $a = v_1 \cdot \ldots \cdot v_n$ of $a$,  there exists a
              subset \ $\Omega \subset [1,n] $ \ such that \ $|\Omega | \le N $ \ and
              \[
              u \Bigm| \, \prod_{\nu \in \Omega} v_\nu \,.
              \]
      \end{enumerate}
Furthermore, we set
\[
\omega (H) = \sup \{ \omega (H, u) \, \colon u \in \mathcal A (H) \} \in \mathbb N \cup \{\infty\} \,.
\]
An atom $u \in H$ is prime if and only if $\omega (H, u) = 1$, and thus $H$ is factorial if and only if $\omega (H) = 1$.  If $H$ satisfies the ACC on divisorial ideals  (in particular, $H$ is a Krull monoid or a Noetherian domain), then
$\omega (H, u) < \infty$ for all $u \in \mathcal A (H)$ \cite[Theorem 4.2]{Ge-Ha08a}. Roughly speaking, the tame degree $\mathsf t (H,u)$ is  the maximum of  $\omega (H,u)$ and  a factorization length of $u^{-1} \prod_{\nu \in \Omega} v_{\nu}$ in $H$. More precisely, for an atom $u \in H$, the local tame degree $\mathsf t (H, u)$ is the smallest $N \in \N_0 \cup \{\infty\}$ having the following property:
\begin{enumerate}
\smallskip
\item[] For any multiple $a$ of $u$ and any factorization $a = v_1 \cdot \ldots \cdot v_n$ of $a$ which does not contain $u$, there is a short subproduct which is a multiple of $u$, say $v_1 \cdot \ldots \cdot v_m$, and a refactorization of this subproduct which contains $u$, say $v_1 \cdot \ldots \cdot v_m = u u_2 \cdot \ldots \cdot u_{\ell}$, such that $\max \{\ell, m \} \le N$.
 \smallskip
      \end{enumerate}
Thus the local tame degree $\mathsf t (H, u)$ measures the distance between any factorization of a multiple $a$ of $u$ and a factorization of $a$ which contains $u$. As before, we set
\[
\mathsf t (H) = \sup \{ \mathsf t (H, u) \, \colon u \in \mathcal A (H) \} \in \mathbb N_0 \cup \{\infty\} \,.
\]

We conclude this section with the following lemma (see \cite[Chapter 1]{Ge-HK06a} and \cite{Ge-Ka10a}) which illustrates how the primary invariants measure the non-uniqueness of factorizations and show that all of these invariants are trivial if the monoid is factorial.

\begin{lemma} \label{2.2}
Let $H$ be an atomic monoid.
\begin{enumerate}[(1)]
\item $H$ is half-factorial if and only if $\rho(H)=1$ if and only if $\rho_k(H)=k$ for every $k \in \mathbb N$.
\item $H$ is factorial if and only if $  \mathsf c(H) = \mathsf t(H)=0$ if and only if $\omega (H) = 1$.
\item $\mathsf c(H)=0$ or $\mathsf c(H) \geq 2$, and if $\mathsf c(H)\leq 2$, then $H$ is half-factorial.
\item $\mathsf c(H) \leq \omega (H) \le \mathsf t(H) \le \omega (H)^2$, and if $H$ is not factorial, then $\max\{2,\rho(H)\} \leq \omega (H)$.
\item If $\mathsf c(H)=3$, then every $L \in \mathcal L(H)$ is an arithmetical progression with difference $1$.
\end{enumerate}
\end{lemma}

\bigskip
\section{Krull monoids, monoids of modules, and transfer homomorphisms} \label{4}
\bigskip

The theory of Krull monoids is presented in detail in the monographs
\cite{HK98} and \cite{Ge-HK06a}. Here we gather the terminology required for our treatment. We then present an introduction to monoids of modules --- the key objects of our study. Finally, we recall important terminology and results about monoids of zero-sum sequences and transfer homomorphisms --- the key tools in our arithmetical investigations.

\medskip
\noindent {\bf Krull monoids.}
Let  $H$  and  $D$  be monoids. A
monoid homomorphism $\varphi \colon H \to D$   is called
\begin{itemize}
\smallskip
\item a  {\it divisor homomorphism} if $\varphi(a)\mid\varphi(b)$ implies that $a \t b$  for all $a,b \in H$.

\smallskip
\item  {\it cofinal} \ if for every $a \in D$ there exists some $u
      \in H$ such that $a \t \varphi(u)$.

\smallskip
\item  a {\it divisor theory} (for $H$) if $D = \mathcal F (\mathcal P)$
for some set $\mathcal P$, $\varphi$ is a divisor homomorphism, and for every
$a  \in \mathcal{F}(\mathcal P)$, there exists a finite nonempty
subset $X \subset H$ satisfying $a = \gcd \bigl(
\varphi(X) \bigr)$.
\end{itemize}
We call
$\mathcal{C}(\varphi)=\mathsf q (D)/ \mathsf q (\varphi(H))$ the
class group of $\varphi $, use additive notation for this group, and for
$a \in \mathsf q(D)$, we denote by \ $[a]  = a
\,\mathsf q(\varphi(H)) \in \mathsf q (D)/ \mathsf q (\varphi(H))$ \
the class containing \ $a$.
Clearly $D/H = \{[a]\, \colon a \in D \} \subset \mathcal C (\varphi)$ is a submonoid with quotient group $\mathcal C (\varphi)$.
The homomorphism  $\varphi$ is cofinal if
and only if $\mathcal{C}(\varphi) = D/H$ and, by definition, every divisor theory is cofinal.
Let $\varphi \colon H \to D = \mathcal F (\mathcal P)$
be a divisor homomorphism. Then $\varphi(H)= \{a \in D
\, \colon [a]=[1]\}$ and $G_{\mathcal P} = \{[p] = p \mathsf q (\varphi(H)) \, \colon p \in \mathcal P \} \subset
\mathcal{C}(\varphi)$
is called the \ {\it  set of classes containing prime divisors}. Moreover,
we have $\langle G_{\mathcal P} \rangle = \mathcal C (\varphi)$ and $[G_{\mathcal P}] =  \{[a]\, \colon a \in D \}$.

The monoid $H$ is called a {\it Krull monoid} if it satisfies one of
the following equivalent conditions{\rm \,:}

\begin{itemize}
\item[(a)]  $H$ is  completely integrally closed and satisfies the ACC on divisorial ideals.

\smallskip
\item[(b)] $H$ has a divisor theory.

\smallskip
\item[(c)] $H$ has a divisor homomorphism into a free abelian monoid.
\end{itemize}
If $H$ is a Krull monoid, then a divisor theory
is unique up to unique isomorphism, and the  class group associated to a divisor theory depends only on $H$. It is called the class group of $H$ and will be denoted by $\mathcal C (H)$. Moreover, a reduced Krull monoid $H$ with divisor theory $H \hookrightarrow \mathcal F(\mathcal P)$ is uniquely determined up to isomorphism by its \emph{characteristic} $(G, (m_g)_{g \in G})$ where $G$ is an abelian group together with an isomorphism $\Phi: G \rightarrow \mathcal C(H)$ and with family $(m_g)_{g\in G}$ of cardinal numbers $m_g=|\mathcal P \cap \Phi(g)|$ (see \cite[Theorem 2.5.4]{Ge-HK06a}, and the forthcoming Lemma \ref{4.4}).

 It is well-known that a domain $R$ is a Krull domain if and only if its multiplicative monoid $R^{\bullet}$ is a Krull
monoid, and we set the class group of $R$ to be $\mathcal C (R) = \mathcal C (R^{\bullet})$. Property (a) shows that a Noetherian domain is
Krull if and only if it is integrally closed. In addition, many well-studied classes of commutative monoids such as regular congruence monoids in Krull domains and Diophantine monoids  are Krull. The focus of the present paper is on Krull monoids stemming from Module Theory.

\medskip
\noindent
{\bf Monoids of modules.} Let $R$ be a (not necessarily commutative) ring and $\mathcal C$ a class of (right) $R$-modules. We say that $\mathcal C$ is {\it closed under finite direct sums, direct summands, and isomorphisms} provided the following holds: Whenever $M, M_1,$ and $M_2$ are $R$-modules with $M \cong M_1 \oplus
M_2$, we have $M \in \mathcal C$ if and only if $M_1, M_2 \in \mathcal C$. We say that $\mathcal C$ satisfies the Krull-Remak-Schmidt-Azumaya Theorem (KRSA for short) if the following holds:
\begin{enumerate}
\item[] If $k, l \in \mathbb N$ and $M_1, \ldots, M_k, N_1, \ldots, N_l$ are indecomposable modules in $\mathcal C$ with $M_1 \oplus \cdots \oplus M_k \cong N_1 \oplus \cdots \oplus N_l$, then $l=k$ and, after a possible reordering of terms, $M_i \cong N_i$ for all $i \in [1,k]$.
\end{enumerate}

Suppose that $\mathcal C$ is closed under finite direct sums, direct summands, and isomorphisms. For a module $M \in \mathcal C$, we denote by $[M]$ its isomorphism class, and by $\mathcal V (\mathcal C)$ the set of  isomorphism classes. (For our purposes here, we tacitly  assume that this is actually a set. For the classes of modules studied in Sections \ref{5} and \ref{6} this is indeed the case. For the involved set-theoretical problems in a more general context, see \cite[Section 2]{Fa12a}.) Then $\mathcal V (\mathcal C)$ is a commutative semigroup with operation $[M] + [N] = [M \oplus N]$ and
all information about direct-sum decompositions of  modules in $\mathcal C$ can be studied in terms of factorizations in the semigroup $\mathcal V (\mathcal C)$. By definition, $\mathcal C$ satisfies KRSA if and only if $\mathcal V (\mathcal C)$ is a free abelian monoid, which holds  if $\End_R(M)$ is local for each indecomposable $M$ in $\mathcal C$ (see \cite[Theorem 1.3]{Le-Wi12a}).

If the endomorphism ring $\End_R (M)$ is semilocal for all modules $M$  in $\mathcal C$, then $\mathcal V (\mathcal C)$ is a  Krull monoid (\cite[Theorem 3.4]{Fa02}).  There is an abundance of recent work which provides examples of rings and classes of modules over these rings for which all endomorphism rings are semilocal (see \cite{Fa04a}, \cite{Fa06a}, and \cite{Fa12a}.  For monoids of modules, a characterization of when the class group is a torsion group is given in \cite{Fa-HK03}).

\smallskip
Suppose  that $\mathcal V (\mathcal C)$ is a Krull monoid. Then to understand the structure of direct-sum decompositions of modules in $\mathcal C$ is to understand the arithmetic of the reduced Krull monoid $\mathcal V (\mathcal C)$. Since any reduced Krull monoid $H$ is uniquely determined by its class group and by the distribution of prime divisors (that is, the characteristic of $H$), one must study these parameters.

\smallskip
In the present paper we will focus on the following classes of modules over a commutative Noetherian local ring $S$, each closed under finite direct sums, direct summands, and isomorphisms. For a commutative Noetherian local ring $S$, we denote by
\begin{itemize}
\item $\mathcal M (S)$ the semigroup of isomorphism classes of finitely generated $S$-modules, by

\item $\mathcal T (S)$ the semigroup of isomorphism classes of finitely generated torsion-free $S$-modules, and by

\item $\mathfrak C (S)$ the semigroup of isomorphism classes of maximal Cohen-Macaulay (MCM) $S$-modules.
\end{itemize}
Note that in order to make $\mcC(S)$ a semigroup, we insist that $[0_S] \in \mcC(S)$, even though the zero module is not MCM. We say that a commutative Noetherian local ring $S$ has finite representation type if there are, up to isomorphism, only finitely many indecomposable MCM $S$-modules. Otherwise we say that $S$ has infinite representation type.

Throughout, let $(R, \mathfrak m)$ be a commutative Noetherian local ring with maximal ideal $\mathfrak m$, and let $(\widehat R, \widehat{\mathfrak m})$ denote its $\mathfrak m$-adic completion. Let $\mathcal V (R)$ and $\mathcal V(\widehat R)$ be any of the above three semigroups. If $M$ is an $R$-module such that $[M] \in \mathcal V (R)$, then $\widehat M \cong M \tensor_R \widehat R$ is an $\widehat R$-module with $[\widehat M] \in \mathcal V (\widehat R)$, and every such $\widehat R$-module is called {\it extended}.  Note that $R$ has finite representation type if and only if $\widehat R$ has finite representation type (see \cite[Chapter 10]{Le-Wi12a}), and that the dimension of $R$ is equal to the  dimension of $\widehat R$. The following crucial result shows that the monoid $\mathcal V(R)$ is Krull.

\medskip
\begin{lemma} \label{module_divisorhom}
Let $(R, \mathfrak m)$ be a commutative Noetherian local ring with maximal ideal $\mathfrak m$, and let $(\widehat R, \widehat{\mathfrak m})$ denote its $\mathfrak m$-adic completion.
\begin{enumerate}[(1)]
\item For each indecomposable finitely generated $\widehat R$-module $M$, $\End_{\widehat R} (M)$ is local, and therefore $\mcM(\widehat R)$, $\mcT(\widehat R)$, and $\mcC(\widehat R)$ are free abelian monoids.

\item  The embedding $\mcM(R) \hookrightarrow \mcM(\widehat R)$ is a divisor homomorphism.  It is cofinal if and only if every finitely generated $\widehat R$-module is a direct summand of an extended module.

\item The embeddings    $\mcT( R) \hookrightarrow \mcT(\widehat R)$ and $\mcC( R) \hookrightarrow \mcC(\widehat R)$ are divisor homomorphisms.
\end{enumerate}
In particular,  $\mcM(R), \mcT( R)$, and $\mcC( R)$ are reduced Krull monoids. Moreover,  the embeddings in (2) and (3)  are injective and map $R$-modules onto the submonoid of extended $\widehat R$-modules.
\end{lemma}

\begin{proof}

\smallskip
Property (1) holds by the Theorem of Krull-Remak-Schmidt-Azumaya (see \cite[Chapter 1]{Le-Wi12a}).

\smallskip
In \cite{Wi01},  Wiegand proved that the given embedding is a divisor homomorphism (see also \cite[Theorem 3.6]{Ba-Wi13a}). The characterization of cofinality follows from the definition and thus (2) holds.

Let $M, N$ be $R$-modules such that either $[M], [N] \in \mathcal V (R)$ where $\mathcal V(R)$ denotes either $\mathcal T(R)$ or $\mathfrak C(R)$ and suppose that $[\widehat M] $ divides $[\widehat N]$ in $\mathcal V (\widehat R)$. Then we have divisibility in $\mathcal M(\widehat R)$, and hence in $\mathcal M (R)$ by (2). Since $\mathcal V (R) \subset \mathcal M (R)$ is divisor-closed, it follows that $[M]$ divides $[N]$ in $\mathcal V (R)$, proving (3).

Together, (2) and (3) show that $\mcM(R), \mcT( R)$, and $\mcC( R)$ satisfy Property (c) in the definition of Krull monoids. Since each of these monoids is reduced,  the maps induced by $[M] \mapsto [\widehat M]$ are injective.
\end{proof}

\smallskip
Note that the embedding $\mcM(R) \hookrightarrow \mcM(\widehat R)$ is not necessarily cofinal, as is shown in \cite{Ha-Wi09} and \cite{Fr-SW-Wi08a}.
In Sections \ref{5} and \ref{6} we will study in detail the class group and the distribution of prime divisors of these Krull monoids, in the case of  one-dimensional and two-dimensional commutative Noetherian local rings.

\medskip
\noindent
{\bf Monoids of zero-sum sequences.} We now introduce  Krull monoids having a combinatorial flavor which are used  to model  arbitrary Krull monoids. Let $G$ be an additive abelian group and let $G_0 \subset G$ be a subset.
Following the tradition in Additive Group and Number Theory, we call the elements of $\mathcal F (G_0)$ {\it sequences} over $G_0$. Thus a sequence $S \in \mathcal F (G_0)$ will be written in the form
\[
S = g_1 \cdot \ldots \cdot g_l = \prod_{g \in G_0} g^{\mathsf v_g (S)}.
\]
We will use all notions (such as the length) as in general free abelian monoids. We set $-S = (-g_1) \cdot \ldots \cdot (-g_l)$,  and call $\sigma (S) = g_1 + \cdots + g_l \in G$ the {\it sum} of $S$. The monoid
\[
\mathcal B (G_0) = \{ S \in \mathcal F (G_0) \, \colon \sigma (S) = 0 \}
\]
is called the {\it monoid of zero-sum sequences} over $G_0$, and its elements are called {\it zero-sum sequences} over $G_0$.  Obviously, the inclusion $\mathcal B (G_0) \hookrightarrow \mathcal F (G_0)$ is a divisor homomorphism, and hence $\mathcal B (G_0)$ is a reduced Krull monoid by Property (c) in the definition of Krull monoids. By definition, the inclusion $\mathcal B (G_0) \hookrightarrow \mathcal F (G_0)$ is cofinal if and only if for every $g \in G_0$ there is an $S \in \mathcal B (G_0)$ with $g \t S$;  equivalently, there is no proper subset $G_0' \subsetneq G_0$ such that $\mathcal B (G_0') = \mathcal B (G_0)$. If $|G|\ne 2$, then $\mathcal C ( \mathcal B (G)) \cong G$, and every class contains precisely one prime divisor.

For every arithmetical invariant  $\ast(H)$, as  defined for a monoid
$H$ in Section \ref{3}, it is usual to  write $\ast(G_0)$ instead of $\ast(\mathcal B(G_0))$
(whenever the meaning is clear from the context). In particular, we set  $\mathcal A (G_0) = \mathcal
A (\mathcal B (G_0))$,  $\mathcal L (G_0) = \mathcal L (\mathcal B (G_0))$, $\mathsf c_{\mon} (G_0) = \mathsf c_{\mon} (\mathcal B (G_0))$, etc.

The study of sequences, subsequence sums, and zero-sums is a flourishing subfield of Additive Group and Number Theory (see, for example, \cite{Ga-Ge06b}, \cite{Ge-Ru09}, and \cite{Gr13a}). The {\it Davenport constant} $\mathsf D (G_0)$, defined as
\[
\mathsf D (G_0) = \sup \{ |U| \, \colon U \in \mathcal A (G_0) \} \in \mathbb N_0 \cup \{\infty\} \,,
\]
is among the most studied invariants in Additive Theory and will  play a crucial role in the computations of arithmetical invariants (see the discussion after Lemma \ref{4.4}). We will need the following two simple lemmas which we present here so as to not clutter the exposition of Section \ref{7}.

\smallskip
\begin{lemma} \label{4.2}
Suppose that the inclusion $\mathcal B (G_0) \hookrightarrow \mathcal F (G_0)$ is cofinal.
The following  are  equivalent.
\begin{enumerate}
\item[(a)] There exist nontrivial submonoids $H_1, H_2 \subset \mathcal B (G_0)$ such that $\mathcal B (G_0) = H_1 \times H_2$.

\item[(b)] There exist nonempty subsets $G_1, G_2 \subset G_0$  such that $G_0=G_1 \uplus G_2$ and $\mcB(G_0)=\mcB(G_1)\times \mcB(G_2)$.

\item[(c)] There exist nonempty subsets $G_1, G_2 \subset G_0$ such that $G_0=G_1 \uplus G_2$ and $\mcA(G_0) = \mcA(G_1) \uplus \mcA(G_2)$.
\end{enumerate}
\end{lemma}

\begin{proof}
Clearly (b) implies (a). The converse follows from \cite[Proposition 2.5.6]{Ge-HK06a}. The implication (b) implies (c) is obvious. We now show that (c) implies (b). Let $B \in \mathcal B (G_0)$. Since $\mathcal B (G_0)$ is a Krull monoid, it is atomic and hence $B = U_1 \cdot \ldots \cdot U_l$ with $U_1, \ldots, U_l \in \mathcal A (G_0)$. After renumbering (if necessary), we can find $k \in [0,l]$ such that $U_1, \ldots, U_k \in \mathcal A (G_1)$ and $U_{k+1}, \ldots, U_l \in \mathcal A (G_2)$. Thus $\mathcal B (G_0) = \mathcal B (G_1) \mathcal B (G_2)$. If $B \in \mathcal B (G_1) \cap \mathcal B (G_2)$, then $B$ is a product of atoms from $\mathcal A (G_1)$ and a product of atoms from $\mathcal A (G_2)$. Since their intersection is empty, both products are empty. Therefore $B = 1$ and hence $\mcB(G_0)=\mcB(G_1)\times \mcB(G_2)$.
\end{proof}

\smallskip
Lemma \ref{4.2}.(c)  shows that $\mathcal B (G^{\bullet})$ is not a direct product of submonoids.
 Suppose that $0 \in G_0$. Then $0 \in \mathcal B (G_0)$ is a prime element and $\mathcal B (G_0) = \mathcal B ( \{0\}) \times \mathcal B (G_0^{\bullet})$. But $\mathcal B ( \{0\}) = \mathcal F ( \{0\}) \cong (\mathbb N_0, +)$, and thus all the arithmetical invariants measuring the non-uniqueness of factorizations of $\mathcal B (G_0)$ and of $\mathcal B (G_0^{\bullet})$ coincide. Therefore we can assume that $0 \notin G_0$ whenever it is convenient.

\smallskip
\begin{lemma} \label{4.3}
Let $G$ be an abelian group and let $G_0 \subset G$ be a subset such that $1 < \mathsf D (G_0) < \infty$.
\begin{enumerate}[(1)]
\item We have $\rho (G_0) \le \mathsf D (G_0)/2$, $k \le \rho_k (G_0) \le k \rho (G_0)$, and $\rho (G_0)^{-1} k \le \lambda_k (G_0) \le k$ for all $k \in \N$.

\smallskip
\item Suppose that $\rho_2 (G_0) = \mathsf D (G_0)$. Then $\rho (G_0) = \mathsf D (G_0)/2$, and for all $k \in \mathbb N$, we have
      \[
      \rho_{2k} (G_0) = k \mathsf D (G_0) \quad \text{and} \quad k \mathsf D (G_0) + 1 \le \rho_{2k+1} (G_0) \le k \mathsf D (G_0) + \frac{\mathsf D (G_0)}{2} \,.
      \]
      Moreover, if $j, l \in \mathbb N_0$ such that $l \mathsf D (G_0) + j \ge 1$, then
      \[
      2l + \frac{2j}{\mathsf D (G_0)} \le \lambda_{l \mathsf D (G_0)+j} (G_0) \le 2l + j \,.
      \]
\end{enumerate}
\end{lemma}

\begin{proof}
By definition, $\lambda_k (G_0) \le k \le \rho_k (G_0)$. Since $\rho (G_0) = \sup \{ \rho_k (G_0)/k \, \colon k \in \mathbb N \}$, it follows that $\rho_k (G_0) \le k \rho (G_0)$ and $k \le \rho (G_0) \lambda_k (G_0)$. Furthermore, $2 \rho_k (G_0) \le k \mathsf D (G_0)$ for all $k \in \mathbb N$ implies that $\rho (G_0) \le \mathsf D (G_0)/2$. This gives (1).

\smallskip
We now prove (2). Since $\rho_k (G_0) + \rho_l (G_0) \le \rho_{k+l} (G_0)$ for every $k, l \in \mathbb N$, (1) implies  that
\[
k \mathsf D (G_0) = k \rho_2 (G_0) \le \rho_{2k} (G_0) \le (2k) \frac{\mathsf D (G_0)}{2} = k \mathsf D (G_0) \,,
\]
and hence
\[
k \mathsf D (G_0) + 1 = \rho_{2k}(G_0) + \rho_1 (G_0) \le \rho_{2k+1} (G_0) \le (2k+1) \rho (G_0)  \le  k \mathsf D (G_0) + \frac{\mathsf D (G_0)}{2} \,.
      \]
Let $j, l \in \mathbb N_0$ such that $l \mathsf D (G_0) + j \ge 1$. For convenience,   set $\rho_0 (G_0) = \lambda_0 (G_0) = 0$. Since
\[
2l = \frac{2}{\mathsf D (G_0)} l \mathsf D (G_0) \le \lambda_{l \mathsf D (G_0)} (G_0) \quad \text{and} \quad \rho_{2l} (G_0) = l \mathsf D (G_0) \,,
\]
it follows that $\lambda_{l \mathsf D (G_0)}(G_0) = 2l$, and hence
\[
\begin{aligned}
2l + \frac{2j}{\mathsf D (G_0)} & = \frac{2}{\mathsf D (G_0)} \big( l \mathsf D (G_0) + j \big) = \rho (G_0)^{-1} \big( l \mathsf D (G_0) + j \big) \\
 & \le \lambda_{l \mathsf D (G_0)+j} (G_0) \le \lambda_{l \mathsf D (G_0)}(G_0) + \lambda_j (G_0) \le 2l + j \,. \qedhere
\end{aligned}
\]
\end{proof}

\medskip
\noindent
{\bf Transfer homomorphisms.}
Transfer homomorphisms are a central tool in Factorization Theory. In order to study a given monoid $H$, one constructs a transfer homomorphism $\theta \colon H \to B$ to a simpler monoid $B$, studies factorizations in $B$, and then lifts arithmetical results from $B$ to $H$. In the case of Krull monoids, transfer homomorphisms allow one to study nearly all of the arithmetical invariants introduced in Section \ref{3} in an associated monoid of zero-sum sequences. We now gather the basic tools necessary for this approach.

\newpage
A monoid homomorphism \ $\theta \colon H \to B$ is called a \ {\it
transfer homomorphism} \ if it has the following properties:
\begin{enumerate}
\item[]
\begin{enumerate}
\item[{\bf (T\,1)\,}] $B = \theta(H) B^\times$ \ and \ $\theta
^{-1} (B^\times) = H^\times$.

\smallskip

\item[{\bf (T\,2)\,}] If $u \in H$, \ $b,\,c \in B$ \ and \ $\theta
(u) = bc$, then there exist \ $v,\,w \in H$ \ such that \ $u = vw$,
\ $\theta (v) \simeq b$ \ and \ $\theta (w) \simeq c$.
\end{enumerate}\end{enumerate}

The next result provides the link between the arithmetic of Krull monoids and Additive Group and Number Theory. This interplay is highlighted in the survey \cite{Ge09a}.

\medskip
\begin{lemma} \label{4.4}
Let $H$ be a Krull monoid, $\varphi \colon H \to D = \mathcal F (\mathcal P)$
a cofinal divisor homomorphism, $G = \mathcal C (\varphi)$ its class
group, and $G_{\mathcal P} \subset G$ the set of classes containing prime
divisors. Let $\widetilde{\boldsymbol \beta} \colon D \to \mathcal F
(G_{\mathcal P})$ denote the unique homomorphism defined by
$\widetilde{\boldsymbol \beta} (p) = [p]$ for all $p \in \mathcal P$.
\begin{enumerate}[(1)]
\item The inclusion $\mathcal B (G_{\mathcal P}) \hookrightarrow \mathcal F (G_{\mathcal P})$ is cofinal, and the homomorphism $\boldsymbol \beta = \widetilde{\boldsymbol \beta} \circ \varphi \colon H \to \mathcal B
      (G_{\mathcal P})$ is a transfer homomorphism.

\item For all  $a \in H$, $\mathsf L_H (a) = \mathsf L_{\mathcal B (G_{\mathcal P})} ( \boldsymbol \beta (a) )$. In particular,  $\mathcal L (H) = \mathcal L (G_{\mathcal P})$, $\Delta (H) = \Delta (G_{\mathcal P})$, $\mathcal U_k (H) = \mathcal U_k (G_{\mathcal P})$, $\rho_k (H) = \rho_k (G_{\mathcal P})$, and $\lambda_k (H) = \lambda_k (G_{\mathcal P})$ for each $k \in \mathbb N$.

\item Suppose that $H$ is not factorial. Then $\mathsf c (H) = \mathsf c (G_{\mathcal P})$, $\mathsf c_{\adj} (H) = \mathsf c_{\adj} (G_{\mathcal P})$, $\mathsf c_{\mon} (H) = \mathsf c_{\mon} (G_{\mathcal P})$, $\Delta^* (H) = \Delta^* (G_{\mathcal P})$, and $\omega (H) \le \mathsf D (G_{\mathcal P})$.
\end{enumerate}
\end{lemma}

\begin{proof}
See \cite[Section 3.4]{Ge-HK06a} for details pertaining to most of the invariants. For the statements on the monotone catenary degree see \cite{Ge-Gr-Sc-Sc10}. Roughly speaking, all of the statements in (2) are straightforward, but the statements in (3) are more subtle. Note that a  statement corresponding to (3) does not hold true for the tame degree (see \cite{Ga-Ge-Sc14a}).
\end{proof}

\medskip
In summary, if the monoid of modules $\mathcal V (R)$  is Krull with class group $G$ and set $G_{\mathcal P}$ of classes containing prime divisors, then the arithmetic of direct-sum decompositions can be studied in the monoid $\mathcal B (G_{\mathcal P})$ of zero-sum sequences over $G_{\mathcal P}$.
In particular, if $H = \mathcal M (R)$ and $D = \mathcal M (\widehat R)$ as in Lemma \ref{module_divisorhom} and all notation is as in Lemma \ref{4.4}, then
\[
\mathsf D (G_{\mathcal P}) = \sup \{ l \, \colon \, \widehat M \cong N_1 \oplus \cdots \oplus N_l \ \text{such that} \ [M] \in \mathcal  A (H) \  \text{ and  $[N_i] \in \mathcal A (D)$ for each $i \in [1,l]$} \} \,.
\]

\smallskip
\section{Monoids of modules: Class groups and distribution of prime divisors I} \label{5}

\noindent
Throughout this section we use the following setup:
\begin{itemize}
\item[\bf{(S)}] {\it $(R,\mathfrak m)$  denotes a one-dimensional analytically unramified commutative Noetherian local ring with unique maximal ideal $\mathfrak m$, $k = R/\mathfrak m$ its residue  field,   $\widehat{R}$ its $\m$-adic completion, and $\spl(R)=|\spec(\widehat R)|-|\spec(R)|$ the splitting number of $R$.}
\end{itemize}

In this section we investigate the characteristic of the Krull monoids $\mathcal M (R)$ and $\mathcal T (R)$ for certain one-dimensional local rings. This study is based on deep module-theoretic work achieved over the past several decades. We gather together module-theoretic information and proceed using a recent construction  (see \ref{diophantinemonoid}) to obtain results on the class group and on the set $G_{\mathcal P}$ of classes containing prime divisors. The literature does not yet contain  a systematic treatment along these lines. Indeed, early results (see Theorem \ref{5.2} below) indicated only the existence of  extremal  sets  $G_{\mathcal P}$ which imply either trivial direct-sum decompositions or  that all arithmetical invariants describing the direct-sum decompositions are  infinite. In either case there was no need for further arithmetical study. Here we reveal that finite and well-structured sets $G_{\mathcal P}$ occur in abundance. Thus, as we will see in Section \ref{7},  the arithmetical behavior of direct-sum decompositions is well-structured.

\smallskip
 We first gather  basic ring and module-theoretic properties. By definition,  $\widehat R$ and  $R$ are both reduced and the integral closure of $R$ is a finitely generated $R$-module. Moreover, we have $\mathfrak C (R) = \mathcal T (R)$. Let  $M$ be a finitely generated $R$-module. If ${\mathfrak p}$ is a minimal prime ideal of $R$,  then $R_{\mathfrak p}$ is a field,  $M_{\mathfrak p}$ is a finite-dimensional  $R_{\mathfrak p}$-vector space, and we set $\rank_{\mathfrak p}(M)= \dim_{R_{\mathfrak p}}(M_{\mathfrak p})$. If $\mathfrak p_1, \ldots, \mathfrak p_s$ are the minimal prime ideals of $R$, then $\rank (M) = (r_1, \ldots, r_s)$ where $r_i = \rank_{\mathfrak p_i} (M)$ for all $i \in [1,s]$. The module $M$ is said to have constant rank if $r_1= \cdots = r_s$.

We start with a  beautiful result of Levy and Odenthall, which gives us a tool to determine which finitely generated $\widehat R$-modules are extended from $R$-modules.

\begin{proposition}\cite[Theorem 6.2]{Le-Od96a}\label{LevyOdenthall}
  Let $M$ be a finitely generated torsion-free $\widehat R$-module. Then $M$ is extended  if and only if $\rank_{\mathfrak p}(M)=\rank_{\mathfrak q}(M)$ whenever $\mathfrak p$ and $\mathfrak q$ are minimal prime ideals of $\widehat{R}$  with $\mathfrak p \cap R= \mathfrak q \cap R$. In particular, if $R$ is a domain, then $M$ is extended  if and only if its rank is constant.
\end{proposition}

\smallskip
We start our discussion with a result which completely determines the characteristic of the Krull monoid $\mathcal M (R)$. The arithmetic of this monoid is studied in Proposition \ref{7.2}.2.

\begin{theorem} \cite[Theorem 6.3]{H-K-K-W07} \label{5.2}
Let $G$ denote the class group of $\mathcal M (R)$ and let $G_{\mathcal P} \subset G$ denote the set of classes containing prime divisors.
\begin{enumerate}[(1)]
\item If $R$ is not Dedekind-like, then $G$ is free abelian of rank $\spl (R)$ and  each  class contains $|k|\aleph_0$ prime divisors.
\item If $R$ is a DVR, then $G=0$.
\item If $R$ is Dedekind-like but not a DVR, then either
\begin{enumerate}
\item $\spl(R) =0$ and $G=0$, or
\item $\spl(R) =1$, $G$ is infinite cyclic with $G = \langle e \rangle$ and $G_{\mathcal P}=\{-e,0,e\}$. Each of the classes $e$ and $-e$ contain  $\aleph_0$ prime divisors and the class $0$ contains  $|k|\aleph_0$ prime divisors.
\end{enumerate}
\end{enumerate}
\end{theorem}

\smallskip
Thus, for the rest of this section,  we focus our attention on  $\mcT(R)$. To determine if the divisor homomorphism $\mcT(R) \hookrightarrow \mcT(\widehat R)$ is a divisor theory, we will require additional information. For now, we easily show that it is always cofinal.

\begin{proposition}\label{cofinal}
The embedding $\mcT(R)\hookrightarrow \mcT(\widehat R)$ is a cofinal divisor homomorphism.
\end{proposition}

\begin{proof}
By Lemma \ref{module_divisorhom} the embedding is a divisor homomorphism. If $M$ is a finitely generated torsion-free $\widehat R$-module, we can consider its rank, $\rank(M)=(r_1,  \ldots , r_t)$, where $t$ is the number of minimal primes of $\widehat R$. If $r_1=\cdots=r_t$, then by Proposition \ref{LevyOdenthall} $M$ is extended, say $M=\widehat N$ for some finitely generated torsion-free $R$-module $N$ and the result is trivial. If the rank of $M$ is not constant, set $r=\max\{r_1, \ldots, r_t \}$ and consider the $\widehat R$-module $L=\left(\widehat R/\mathfrak q_1\right)^{r-r_1}\oplus \cdots \oplus \left(\widehat R/\mathfrak q_t\right)^{r-r_t}$, where $\mathfrak q_1,  \ldots , \mathfrak q_t$ denote the minimal primes of $\widehat R$. Then $\rank(N \oplus L)=(r, \ldots ,r)$ is constant and hence $N \oplus L$ is extended, say $N \oplus L\cong \widehat P$ for some finitely generated torsion-free $R$-module $P$. Clearly $M$ is isomorphic to a direct summand of $\widehat P$ and the result follows.
\end{proof}

  Since $\mcT(\widehat R)$ is free abelian, we can identify it with the free abelian monoid $\N_0^{(\mathcal P)}$ where $\mathcal P$ is an index set for the isomorphism classes of indecomposable finitely generated torsion-free $\widehat R$-modules.  We then use Proposition \ref{LevyOdenthall} to give a detailed description of $\mcT(R)$. The following construction has been used numerous times (see, for example,  \cite{Ba-Lu11a}, \cite{Ba-Sa12a}, and \cite{F-H-K-W06}).

\begin{construction}\label{diophantinemonoid}~

\begin{itemize}
\item Let $\mathfrak p_1,\ldots, \mathfrak p_s$ be the distinct minimal prime ideals of $R$.
      For each $i \in [1,  s]$, let $\mathfrak q_{i,1},\ldots,\mathfrak q_{i,t_i}$ be the minimal primes of $\widehat{R}$ lying over $\mathfrak p_i$. Note that $\spl(R)=\sum_{i=1}^s(t_i-1)$.

\item Let $\mathcal P$ be  the set of isomorphism classes of indecomposable finitely generated torsion-free $\widehat R$-modules.

\item Let $\mathsf A(R)$ be the $\spl(R) \times |\mathcal P|$ matrix whose column indexed by the isomorphism class $[M]\in \mathcal P$ is
\[
\begin{bmatrix} r_{1,1}-r_{1,2} & \cdots & r_{1,1}-r_{1,t_1} & \cdots &
r_{s,1}-r_{s,2} & \cdots & r_{s,1}-r_{s,t_s}
\end{bmatrix}^T
\]
where $r_{i,j}=\rank_{\mathfrak q_{i,j}}(M)$.
\end{itemize}

Then $\mcT(R) \cong \ker(\mathsf A(R))\cap \mathbb N^{(\mathcal P)}\subset \N_0^{(\mathcal P)}$ is a Diophantine monoid.

\end{construction}

If one has a complete description of how the minimal prime ideals of $\widehat R$ lie over the minimal prime ideals of $R$ together with the ranks of all indecomposable finitely generated torsion-free $\widehat R$-modules, then Construction \ref{diophantinemonoid} completely describes the monoid $\mcT(R)$. In certain cases (eg. Section \ref{DimensionOne_FRT}) we are able to obtain all of this information. Other times we know only some of the ranks that occur for indecomposable $\widehat R$-modules and thus have only a partial description for $\mcT(R)$. However, as was shown in \cite{Ba-Sa12a}, the ranks of indecomposable cyclic $\widehat R$-modules gives enough information about the columns of $\mathsf A (R)$ to prove that $\mcT(R) \hookrightarrow \mcT(\widehat R)$ is nearly always a divisor theory. First we recall that if $\mathfrak q_1, \ldots , \mathfrak q_t$ are the minimal primes of $\widehat R$, and $E \subset [1, t]$. Then
\[
\rank\left(\frac{\widehat R}{\bigcap_{i \in E} \mathfrak q_i}\right)= (r_1,\dots,r_t),
\text{ where } r_i = \begin{cases} 1, \quad i \in E \\ 0, \quad i
\not\in E.\end{cases}
\]
Thus \emph{every} nontrivial $t$-tuple of zeros and ones can be
realized as the rank of a nonzero (necessarily indecomposable) cyclic $\widehat R$-module.
Thus we have the following:

\begin{construction} \label{C:constructionT}
Let all notation be as in Construction \ref{diophantinemonoid}. After renumbering if necessary, there is $p \in [0,s]$ such that  $t_1, \ldots , t_p \geq 2$  and  such that $t_i=1$ for each $i \in [p+1,s]$. Then $\spl(R)=\sum_{j=1}^p t_j - p$. For each $i \in [1, p]$, let $A_i$ be the set of $(t_i-1)\times 1$ column
vectors all of whose entries are either $0$ or $1$, and let $B_i$ be
the set of $(t_i-1)\times 1$ column vectors all of whose entries are
either $0$ or $-1$.

We now define  $\mcT$ to be
the $\spl(R)\times \prod_{i=1}^p(2^{t_i}-1)$ matrix, each of whose columns has the form
\begin{equation*}\label{E:columnT}
\left[\begin{array}{c} T_1\\ \hline \vdots\\ \hline
T_p\end{array}\right], \qquad \text{where $T_i\in A_i\cup B_i$ for each $i\in [1, p]$.}
\end{equation*}
\end{construction}

With the notation as in Constructions \ref{diophantinemonoid} and \ref{C:constructionT}, we give a realization result which shows that the matrix $\mathcal T$ occurs as a submatrix of $\mathsf A (R)$.

\begin{proposition} \cite[Proposition 3.7]{Ba-Sa12a}\label{trivialranks}
For each column $\boldsymbol{\alpha}$ of
$\mcT$, there exist nonnegative integers $r_{i,j}$ and an
indecomposable torsion-free
$\widehat{R}$-module~$M_{\boldsymbol{\alpha}}$ of rank
\[
(r_{1,1}, \ldots, r_{1,t_1}, \ldots, r_{p,1}, \ldots, r_{p,t_p},
r_{p+1,1}, \ldots, r_{s,1})
\]
such that
\begin{equation*}\label{E:alpha}
\boldsymbol{\alpha}=\begin{bmatrix} r_{1,1}-r_{1,2} & \cdots &
r_{1,1}-r_{1,t_1} & \cdots & r_{p,1}-r_{p,2} & \cdots &
r_{p,1}-r_{p,t_p}
\end{bmatrix}^T.
\end{equation*}
\end{proposition}

In particular,  the matrix $\mathsf A(R)$ nearly always satisfies the hypotheses of the following lemma.

\begin{lemma}\cite[Lemma 4.1]{Ba-Sa12a}\label{DivisorTheory}
Fix an integer $q\ge 1$, and let $I_q$ denote the $q\times q$
identity matrix. Let $\mathcal P$ be an index set, and let $\mcD$ be a
$q\times |\mathcal P|$ integer matrix whose columns are indexed by
$\mathcal P$. Assume $\mcD=\begin{bmatrix} D_1 \,|\, D_2
\end{bmatrix}$, where $D_1$ is the $q\times (2q+2)$ integer matrix
\[
\left[\begin{array}{c|c|cc} & & 1 & -1\\ I_q & -I_q & \vdots
&\phantom{-}\vdots\\ & & 1 & -1 \end{array}\right],
\]
and $D_2$ is an arbitrary integer matrix with $q$ rows (and possibly
infinitely many columns). Let $H=\ker(\mcD)\cap \N_0^{(\mathcal P)}$.
\begin{enumerate}[(1)]
\item \label{surjection} The map $\mcD\colon \Z^{(\mathcal P)} \to \Z^{(q)}$ is surjective.
\item \label{divtheory} The natural inclusion $H \hookrightarrow \N_0^{(\mathcal P)}$ is a divisor theory.
\item \label{differences} $\ker(\mcD)=\mathsf q (H)$.
\item \label{classgroup} $\Cl(H) \cong \Z^{(q)}$, and this isomorphism maps the set of classes containing prime divisors onto the set of distinct columns of  $\mcD$.
\end{enumerate}
In particular, we observe the following: Given a fixed column $\alpha$ of $\mcD$, the cardinality of $\{ \beta \, \colon \, \beta \ \text{is a column of}$  $\mcD \ \text{ and} \  \beta = \alpha \}$ is equal to the cardinality of prime divisors in the class corresponding to $\alpha$. Therefore, the characteristic of the Krull monoid $H$ is completely given by the matrix $\mcD$.
\end{lemma}

 Based on the previous results, one easily obtains the following theorem which provides the framework for our study of the characteristic of $\mcT(R)$.

\begin{theorem}\label{1dimensional-summary}~

\begin{enumerate}[(1)]
\item If $\spl (R) =0$, then $\mcT(R)\cong \mcT(\widehat R)$ is free abelian.

\item If $\spl (R) \geq 2$ then the embedding $\mcT(R) \hookrightarrow \mcT(\widehat R)$  is a divisor theory. Moreover,
\begin{enumerate}

\item $\mcT(R)\cong \ker(\mathsf A (R)) \cap \N_0^{(\mathcal P)}$,
\item $\Cl(\mcT(R)) \cong \Z^{(\spl (R))}$, and this isomorphism maps the set of classes containing prime divisors onto the set of distinct columns of $\mathsf A (R)$.
\end{enumerate}
\end{enumerate}
\end{theorem}

Suppose that $\spl (R)=1$. The embedding $\mcT(R) \hookrightarrow \mcT(\widehat R)$ is a divisor theory if and only if
the defining matrix $\mathsf A (R)$ contains at least two positive and at least two negative entries (see  Proposition \ref{7.1}.2).

In many cases, computing the ranks of indecomposable $\widehat R$-modules and hence the columns of the defining matrix $\mathsf A (R)$ is difficult. However, an additional hypotheses on $R$ implies that the set of classes containing prime divisors satisfies $G_{\mathcal P} = -G_{\mathcal P}$, a crucial property for all arithmetical investigations (see Proposition \ref{7.2} and the subsequent remarks).

\begin{corollary}\label{oppositeranks}
Suppose in addition that  $\widehat R\cong S/(f)$ where $(S,\n)$ is a hypersurface i.e., a regular Noetherian local ring of dimension two and where $0\not=f \in \n$. If $G$ is the class group of  $\mcT(R) \hookrightarrow \mcT(\widehat R)$  and $G_{\mathcal P}$ is the set of  classes containing prime divisors, then $G_{\mathcal P}=-G_{\mathcal P}$.
\end{corollary}

\begin{proof}
With the hypotheses given, we can apply \cite[Proposition 6.2]{Ba-Sa12a} to see that if $M$ is any indecomposable $\widehat R$-module with rank $(r_1, \ldots , r_t)$, then there is an indecomposable $\widehat R$-module $N$ with rank $(m-r_1, m-r_2, \ldots , m-r_t)$ for some $m\geq \max\{r_1, \ldots , r_t\}$. Using Construction \ref{diophantinemonoid} we see that if $\alpha=\begin{bmatrix} a_1 & \cdots & a_q \end{bmatrix}$ is the column of $\mathsf A (R)$ indexed by $M$, then $-\alpha$ is the column indexed by $N$. Therefore, since $G_{\mathcal P}$ corresponds to the distinct columns of $\mathsf A (R)$, $G_{\mathcal P}=-G_{\mathcal P}$.
\end{proof}

\begin{remark}\label{R:isomorphicG_0}
Although the system of equations developed in Construction \ref{diophantinemonoid} is somehow natural, it is not the only system of equations which can be used to define $\mcT(R)$. Indeed, the matrix $\mathsf A (R)$ can be adjusted by performing any set of elementary row operations. If $J$ is an elementary matrix corresponding to such a set of row operations, then $\mcT(R) \cong \ker(\mathsf A (R)) \cap \N_0^{(\mathcal P)} \cong \ker(J \mathsf A (R)) \cap \N_0^{(\mathcal P)}$. Moreover, this isomorphism gives rise to an automorphism of $\Cl(\mcT(R))$ mapping the set of  classes containing prime divisors to another set of classes containing prime divisors. Example \ref{E:isomorphicG_0} illustrates the usefulness of considering an alternate defining matrix for $\mcT(R)$.
\end{remark}

\quad

\subsection{Finite Representation Type}\label{DimensionOne_FRT} \label{4a}
{\it Throughout this subsection, let $R$ be as in Setup {\bf (S)}, and  suppose in addition that $R$ has finite representation type.}

Decades with of work going back to a paper of Green and Reiner \cite{Gr-Re78a} and including papers by R.~and S.~Wiegand, \c{C}imen, Arnuvut, Luckas, and Baeth (see \cite{Wi-Wi94a}, \cite{Ci98a}, \cite{Ar-Lu-Wi07}, and \cite{Ba07b}) culminated in a paper by Baeth and Luckas \cite{Ba-Lu09a} which classified precisely the tuples that can occur as the ranks of indecomposable torsion-free $R$-modules. We note that since $R$ has finite representation type, both $R$ and $\widehat R$ have at most three minimal primes (see \cite[Theorem 0.5]{Ci-Wi-Wi92}).

\begin{proposition}\cite[Main Theorem 1.2]{Ba-Lu09a}~

\begin{enumerate}[(1)]
\item If $\widehat R$ is a domain, then every indecomposable finitely generated torsion-free $\widehat R$-module has rank $1$, $2$ or $3$.

\item If $\widehat R$ has exactly two minimal prime ideals, then every indecomposable finitely generated torsion-free $\widehat R$-module has rank $(0,1),(1,0),(1,1),(1,2),(2,1) \textrm{ or } (2,2).$

\item If $\widehat R$ has exactly two minimal prime ideals, then every indecomposable finitely generated torsion-free $\widehat R$-module has rank $(0,0,1),(0,1,0),(1,0,0),(0,1,1),(1,0,1),(1,1,0)$, $(1,1,1)$ or $(2,1,1).$
\end{enumerate}
\end{proposition}

Note the lack of symmetry in case (3): With a  predetermined order on the minimal prime ideals of $\widehat R$, there is an indecomposable module of rank $(2,1,1)$, but not of rank $(1,2,1)$ or $(1,1,2)$. As is stated in \cite[Remark 5.2]{Ba-Lu09a}, even for a fixed number of minimal primes, not each of these tuples will occur as the rank of an indecomposable module for each ring. However, since when applying Construction \ref{diophantinemonoid} we cannot distinguish between an indecomposable of rank $(2,1)$ and one of rank $(1,0)$, and since all nontrivial tuples of zeros and ones occur as ranks of indecomposable cyclic modules, we have \cite[Proposition 3.3]{Ba-Lu11a}:

\bigskip
\begin{enumerate}[(1)]
\item If $\spl(R)=1$, then $\mathsf A (R)=\begin{bmatrix}1 & \cdots & 1 & -1 & \cdots & -1 & 0 & \cdots & 0\end{bmatrix}$.
\item If $\spl(R)=2$, then $\mathsf A (R) =\begin{bmatrix}0 & -1 & 1 & -1 & 1 & 0 & 0 & 1 & \cdots \\ -1 & 0 & 1 & -1 & 0 & 1 & 0 & 1 & \cdots\end{bmatrix}$.
\end{enumerate}
\bigskip

When $\spl(R)=1$, we are guaranteed at least one entry for each of $1$, $-1$, and $0$, coming from the ranks of indecomposable cyclic $\widehat R$-modules. If we have at most one $1$ or at most one $-1$ in the defining matrix $\mathsf A (R)$, then it must be the case that $R$ is a domain, $\widehat R$ has exactly two minimal primes $\mathfrak p$ and $\mathfrak q$, and up to isomorphism either $\widehat R/\mathfrak p$ is the only indecomposable torsion-free $\widehat R$-module of rank $(r,s)$ with $r-s=1$ or the $\widehat R/\mathfrak q$ is the only indecomposable torsion-free $\widehat R$-module of rank $(r,s)$ with $r-s=-1$. If this is the situation, we say that $R$ satisfies condition ($\dagger$). In case $\spl(R)=2$, we are guaranteed that each column listed appears at least once as a column of $\mathsf A (R)$.

 We then have the following  refinement of Theorem \ref{1dimensional-summary} when $R$ has finite representation type. The arithmetic of this monoid is studied in Proposition \ref{7.2}.2, Theorem \ref{7.4}, and Corollary \ref{7.10}.

\begin{theorem}\cite[Proposition 3.3]{Ba-Lu11a}\label{DivisorTheory2}~

\begin{enumerate}[(1)]

\item If $\spl(R)=1$ and $R$ satisfies condition ($\dagger$) then  $\mcT(R) \hookrightarrow \mcT(\widehat R)$ is not a divisor theory but $\mcT(R)$ is free abelian.
\item If $\spl(R)=1$ and $R$ does not satisfy condition ($\dagger$), then $\mcT(R) \hookrightarrow \mcT(\widehat R)$ is a divisor theory with infinite cyclic class group  $G = \langle e \rangle$, and $G_{\mathcal P}=\{-e,0,e\}$.

\item If $\spl(R)=2$, then $\mcT(R) \hookrightarrow \mcT(\widehat R)$ is a divisor theory and $\Cl(\mcT(R))\cong \mathbb Z^{(2)}$. Moreover, this isomorphism maps the set of classes containing prime divisors onto $\left\{\begin{bmatrix}1 \\ 1\end{bmatrix}, \begin{bmatrix}-1 \\ -1\end{bmatrix}, \begin{bmatrix}1 \\ 0\end{bmatrix}, \begin{bmatrix}-1 \\ 0\end{bmatrix}, \begin{bmatrix}0 \\ 1\end{bmatrix}, \begin{bmatrix}0 \\ -1\end{bmatrix} \right\}$.
\end{enumerate}
\end{theorem}

\quad

\subsection{Infinite Representation Type}\label{DimensionOne_IRT} \label{4b}
{\it Throughout this subsection, let $R$ be as in Setup {\bf (S)}, and  suppose in addition that $R$ has infinite representation type.}

Unfortunately, in this case,  there is no known complete list of the tuples that can occur as ranks of indecomposable finitely generated torsion-free $R$-modules. Thus we cannot give a full description of $\mcT(R)$ using Construction \ref{diophantinemonoid}. However, with the additional assumption that $\widehat R/{\mathfrak q}$ has infinite representation type for some minimal prime ideal ${\mathfrak q}$ of $\widehat R$, we can produce a wide variety of interesting ranks and can provide a partial description of $\mcT(R)$. This information is enough to show that, very much unlike the finite representation type case of Section \ref{DimensionOne_FRT}, all of the arithmetical invariants we study are infinite.

\begin{proposition}\cite[Theorem 3.4.1]{Sa10a}\label{lotsofranks}
Let $S$ be a one-dimensional analytically unramified commutative Noetherian local ring with residue field $K$, and with $t$ minimal prime ideals  $\mathfrak q_1$, \ldots, $\mathfrak q_t$ such that
$S/\mathfrak q_1$ has infinite representation type. Let $(r_1,\ldots, r_t)$ be a nonzero $t$-tuple
of nonnegative integers with $r_i\le 2r_1$ for all $i\in
[2, t]$.
\begin{enumerate}[(1)]
\item There exists an indecomposable torsion-free $S$-module of rank $(r_1,\ldots, r_t)$.
\item If the residue field $K$ is infinite, then the set of isomorphism classes of indecomposable torsion-free
$S$-modules of rank $(r_1,\ldots, r_t)$ has cardinality $|K|$.
\end{enumerate}
\end{proposition}

By Proposition \ref{lotsofranks} the conditions of Lemma \ref{DivisorTheory} are satisfied. Therefore the map $\mcT(R) \hookrightarrow \mcT(\widehat R)$  is a divisor theory and the class group  $\Cl(\mcT(R))$ is free abelian of rank $\spl (R)$.
Our main result of this subsection is a refinement of Theorem \ref{1dimensional-summary}. Its arithmetical consequences are given in Proposition \ref{7.2}.1, strongly improving the arithmetical characterizations given in \cite{Ba-Sa12a}.

\begin{theorem}\label{integers}
Suppose that  $\spl (R)\geq 1$ and that there is at least one minimal prime ideal $\mathfrak q$ of $\widehat{R}$
such that $\widehat{R}/\mathfrak q$ has infinite representation type. Then $\Cl(\mcT(R))$ is  free abelian  of rank $\spl (R)$ and the set  of classes containing prime divisors contains an infinite cyclic subgroup.
\end{theorem}

\begin{proof}
Let $\mathfrak p_1,  \ldots , \mathfrak p_s$ denote the minimal primes of $R$ and, for each $i \in [1, s]$, let $\mathfrak q_{i,1}, \ldots , \mathfrak q_{i,t_i}$ denote the set of minimal primes of $\widehat R$ lying over $\mathfrak p_i$. Without loss of generality, assume that $\widehat R/\mathfrak q_{1,1}$ has infinite representation type. If $t_1=1$, then without loss of generality, $t_2>1$. From Proposition \ref{lotsofranks} there is, for each pair $(r,s)$ of nonnegative integers (not both zero), an indecomposable $\widehat R$-module $M$ with $\rank_{\mathfrak q_{2,1}}(M)=r$, $\rank_{\mathfrak q_{2,2}}(M)=s$, and $\rank_{\mathfrak q_{i,j}}(M)=0$ for all $(i,j)\not\in\{(1,1),(2,1),(2,2)\}$. Now suppose that $t_1>1$. Then we have, from Proposition \ref{lotsofranks}, for each pair $(r,s)$ of nonnegative integers (not both zero) satisfying $r-s \geq -s$, an indecomposable $\widehat R$-module $M$ with $\rank_{\mathfrak q_{1,1}}(M)=r$, $\rank_{\mathfrak q_{1,2}}(M)=s$, and $\rank_{\mathfrak q_{i,j}}(M)=0$ for all $(i,j)\not\in\{(1,1),(1,2)\}$. In either case, using Construction \ref{diophantinemonoid} we see that the set $\left\{\begin{bmatrix}x & 0 & \cdots & 0\end{bmatrix}^T\,:\, x \in \mathbb Z\right\}$ occurs as a set of columns for $\mathsf A (R)$ and hence occurs as a subset of the set of classes containing prime divisors.
\end{proof}

\subsection{Divisor-closed submonoids of $\mcT(R)$}\label{DimensionOne_unknown} \label{4c}

Suppose that $R$ has infinite representation type but, in contrast to  Theorem \ref{integers},  suppose  that $\widehat R/{\mathfrak q}$ has finite representation type for each minimal prime ${\mathfrak q}$ of $\widehat R$. Then there is no known classification of all ranks of indecomposable finitely generated torsion-free $R$-modules. Specific rings have been studied in the literature, but even in these settings, a complete solution has been unattainable. We now give such an example which we will return to in Section \ref{5.4}.

\begin{example} \label{unknownranks}
Let $K$ be an algebraically closed field of characteristic zero. Consider the  ring $S=K[[x,y]]/(x^4-xy^7)$ which has exactly two minimal primes $xS$ and $(x^3-y^7)S$. Detailed constructions in \cite{Ka-Wi11a} and \cite{Sa10a} show that $S$ has indecomposable modules of ranks $(m,m)$, $(m+1,m)$, and $(m+2,m)$ for each positive integer $m$. Moreover, \cite[Proposition 6.2]{Ba-Sa12a} guarantees indecomposable modules of ranks $(s-(m+1),s-m)$ and $(t-(m+2),t-m)$ where $s\geq m+1$ and $t\geq m+2$ are positive integers. Determining what other tuples occur as ranks of indecomposable torsion-free $S$-modules appears to be quite difficult.
\end{example}

Thus, since studying $\mathcal T (R)$ as a whole is out of reach at the present state of knowledge, we pick finitely many $R$-modules $M_1,  \ldots , M_n$, and study the direct-sum relations among them.
In more technical terms,
instead of studying the full Krull monoid $\mathcal T (R)$, we focus on divisor-closed submonoids. Suppose that $H$ is a Krull monoid and
$H \hookrightarrow \mathcal F ({\mathcal P})$ a cofinal divisor homomorphism. If $H' \subset H$ is a divisor-closed submonoid, then $H' \hookrightarrow H \hookrightarrow \mathcal F ({\mathcal P})$ is a divisor homomorphism. For each of  the arithmetical invariants $\ast( \cdot)$ introduced in Section \ref{3}, we have $\ast(H') \le \ast(H')$ resp. $\ast(H') \subset \ast(H)$; for example we have $\mathsf c (H') \le \mathsf c (H)$, $\mathcal L (H') \subset \mathcal L (H)$, and so on.
Moreover, if $H'$ is the smallest divisor-closed submonoid containing finitely many elements $a_1, \ldots, a_k \in H$, it is also the smallest divisor-closed submonoid containing $a_1 \cdot \ldots \cdot a_k$.

For the rest of Section \ref{5}, we study  divisor-closed submonoids of $\mathcal T (R)$ generated by a single $R$-module $M$, regardless of whether $R$ has  finite or infinite representation type. We denote this monoid by $\add (M)$. Before discussing specific examples in Section \ref{5.4}, we carefully recall the consequences of our main Construction \ref{diophantinemonoid} for such submonoids.

\begin{construction}\label{add(M)}
Let $R$ and $\widehat R$ be as in Construction \ref{diophantinemonoid}. Let $M$ be a finitely generated torsion-free $R$-module. Then $\add(M)$ consists of all isomorphism classes $[N] \in \mcT(R)$ such that $N$ is isomorphic to a direct summand of $M^{(n)}$ for some finite positive integer $n$.

Write $\widehat M=L_1^{(n_1)}\oplus \cdots \oplus L_k^{(n_k)}$ where the $L_i$ are pairwise nonisomorphic indecomposable finitely generated torsion-free $\widehat R$-modules and the $n_i$ are positive integers. If $[N] \in \add(M)$, then $[\widehat N] \in \add (\widehat M)$ and thus, since direct-sum decomposition is essentially unique over $\widehat R$, $\widehat N\cong L_1^{(a_1)}\oplus \cdots \oplus L_k^{(a_k)}$ with each $a_i$ a nonnegative integer at most $n_i$. Thus there is a divisor homomorphism $\Psi: \add(M) \rightarrow \N_0^{(k)}$ given by $[N] \mapsto (a_1, \ldots , a_k)$. We  identify $\add(M)$ with the saturated submonoid $\Gamma(M)=\Psi(\add(M))$ of $\N_0^{(k)}$.

Moreover, if $\mathsf A (M)$ is the $\spl(R) \times k$ integer-valued matrix for which the $l$th column is the transpose of the row-vector $\begin{bmatrix} r_{1,1}-r_{1,2} & \cdots & r_{1,1}-r_{1,t_1} & \cdots & r_{s,1}-r_{s,2} & \cdots & r_{s,1}-r_{s,t_s} \end{bmatrix}$ where $r_{i,j}=\rank_{\mathfrak q_{i,j}}(V_l)$, then $\add(M) \cong \Gamma(M)=\ker(\mathsf A (M))\cap \N_0^{(k)}$.
\end{construction}

We now state a corollary of  Theorem \ref{1dimensional-summary} for $\add(M)$.

\begin{corollary}\label{add(M)-summary}
Let $M$ be a finitely generated torsion-free $R$-module as in Construction \ref{add(M)}.
\begin{enumerate}[(1)]
\item If $\spl (R)=0$, then $\add(M)\cong \add(\widehat M)$ is free abelian.
\item If $\spl (R) \geq 1$ and $\mathsf A (M)$ satisfies the conditions of Lemma \ref{DivisorTheory}, then the inclusion $\Gamma(M)\subset \N_0^{(k)}$ is a divisor theory. Moreover,
\begin{enumerate}
\item $\add(M)\cong \ker(\mathsf A (M)) \cap \N_0^{(k)}$,
\item $\Cl(\add(M)) \cong \Z^{(\spl (R))}$, and this isomorphism maps the set of classes containing prime divisors onto the set of distinct columns of $\mathsf A (M)$.
\end{enumerate}
\end{enumerate}
\end{corollary}

Before considering explicit examples, we give a realization result (see also \cite[Chapter 1]{Le-Wi12a}).

\medskip
\begin{proposition} \label{realization}
Let $H$ be a reduced Krull monoid with free abelian class group $G$ of rank $q$ and let $G_{\mathcal P} \subset G$ denote the set of classes containing prime divisors. Suppose that $G_{\mathcal P}$ is finite and that $G$ has a basis $(e_1, \ldots, e_q)$ such that $G_0 = \{e_0=e_1+ \cdots +e_q, e_1, \ldots, e_q, -e_0, \ldots, -e_q\} \subset G_{\mathcal P}$.  Then there exists an analytically unramified commutative Noetherian local domain $S$ and a finitely generated torsion-free $S$-module $M$ such that $\add(M) \cong H$.
\end{proposition}

\begin{proof}
Let $\Phi: G \rightarrow \mathbb Z^{(q)}$ denote the isomorphism   which maps $(e_1, \ldots, e_q)$ onto the standard basis of $\mathbb Z^{(q)}$. Let $S$ be an analytically unramified Noetherian local domain with completion $\widehat S$ having $q+1$ minimal primes $Q_0, \ldots , Q_q$ such that $\widehat S/Q_0$ has infinite representation type. For $\mathbf{s}=\begin{bmatrix} s_1 & \cdots & s_q\end{bmatrix} \in \Phi (G_{\mathcal P})$,  set $r_0=s_1+\cdots +s_q$ and $r_i=\sum_{j\not=i}s_j$ for each $i \in [1, q]$. By Theorem \ref{lotsofranks} there exists an indecomposable finitely generated torsion-free $\widehat S$-module $N_{\mathbf{s}}$ such that $\rank(N_{\mathbf{s}})=(r_0, \ldots , r_q)$. Set $N=\bigoplus_{\mathbf{s} \in \Phi (G_{\mathcal P})}N_{\mathbf{s}}$ and write $\rank(N)=(a_0, \ldots , a_q)$. Set $a=\max\{a_0, \ldots , a_q\}$ and $L=\bigoplus_{i=0}^q\left(\widehat S/Q_i\right)^{(a-a_i)}$. Then $N\oplus L$ is a finitely generated torsion-free $\widehat S$-module with constant rank and is thus extended from a finitely generated torsion-free $S$-module $M$. By Construction \ref{add(M)} and Corollary \ref{add(M)-summary} we see that $\add(M)$ has class group isomorphic to $\mathbb Z^{(q)}$ and this isomorphism maps the set of prime divisors onto the elements of the set $\Phi (G_{\mathcal P})$.
\end{proof}

\quad

\subsection{Examples} \label{5.4} \label{4d}

In this section we provide the constructions of naturally occurring monoids $\add(M)$ where $M$ is a finitely generated torsion-free $R$-module. In particular, we construct specific modules $M$, whose completion $\widehat M$ is often a direct sum of indecomposable cyclic $\widehat R$-modules and we determine the class group $G$ of $\add (M)$ and the set of classes $G_{\mathcal P} \subset G$ containing prime divisors. Note that the Krull monoids $\mathcal M (R)$ of \emph{all} finitely generated $R$-modules and $\mathcal T (R)$ of \emph{all} finitely generated torsion-free $R$-modules have class groups $G' \supset G$ and a set $G_{\mathcal P}'$ of classes containing prime divisors such that $G_{\mathcal P}' \supset G_{\mathcal P}$. Since $\add(M)$ is a divisor-closed submonoid of both $\mathcal M(R)$ and of $\mathcal T(R)$, a study of the arithmetic of $\add(M)$ provides a partial description of $\mathcal M(R)$ and $\mathcal T(R)$. Moreover, the values of arithmetical invariants of $\add(M)$ give lower bounds on the same arithmetical invariants of $\mathcal M(R)$ and $\mathcal T(R)$.

In each of the following examples we construct an $\widehat R$-module $L = L_1^{n_1} \oplus \cdots \oplus L_k^{n_k}$ of constant rank, where $L_1, \ldots, L_k$ are pairwise non-isomorphic indecomposable $\widehat R$-modules. Then, by Corollary \ref{add(M)-summary} with $\widehat M  \cong L$ for some $R$-module $M$,
\[
\add (M) \cong \ker \big( \mathsf A (M) \big) \cap \N_0^{(k)} \subset  \N_0^{(k)} \cong \add (L) \,.
\]
In particular, we do so in such a way that the natural map $\add (M) \hookrightarrow \add (L)$ is a divisor theory with class group isomorphic to $\Z^{(\spl (R))}$ and where the set of classes containing prime divisors maps onto the distinct columns of $\mathsf A (M)$.

\begin{example} \label{4.19}
We now construct a monoid of modules whose arithmetic will be studied in Proposition \ref{7.12}. Let $S$ be as in Example \ref{unknownranks}. Then there are indecomposable torsion-free $ S$-modules $M_1$, $M_{-1}$, $M_2$, $M_{-2}$, $N_1$, $N_{-1}$, $N_2$, and $N_{-2}$ with ranks (respectively) $(2,1)$, $(1,2)$, $(3,1)$, $(1,3)$, $(3,2)$, $(3,2)$, $(2,3)$, $(4,2)$, and $(2,4)$. Set $L$ to be the direct sum of these eight indecomposable $S$-modules. By Lech's Theorem \cite{Le86a}, there exists a Noetherian local  domain $(R,\m)$ with $\m$-adic completion $\widehat R\cong S$. Since $L$ has constant rank, $L$ is extended from some $R$-module $M$, and   $\add(M)\hookrightarrow \add(L)\cong \N_0^{(8)}$ is a divisor theory with infinite cyclic class group $G$ and with  $G_{\mathcal P}=\{-2e,-e,e,2e\}$ where $G = \langle e \rangle$.
\end{example}

\medskip
\begin{example}\label{E:isomorphicG_0}
We now provide an example that illustrates the convenience of choosing an alternate defining matrix for $\add(M)$, as is described in Remark \ref{R:isomorphicG_0}.
Its arithmetic is given in Theorem \ref{7.4}.
Suppose that $R$ has two minimal prime ideals $\mathfrak p_1, \mathfrak p_2$ and that $\widehat R$  has five minimal prime ideals $\mathfrak q_{(1,1)}$, $\mathfrak q_{(1,2)}$, $\mathfrak q_{(1,3)}$, $\mathfrak q_{(2,1)}$, and $\mathfrak q_{(2,2)}$,  with $\mathfrak q_{(i,j)}$ lying over $\mathfrak p_i$ for each $i \in [1,2]$ and for each $j$. Set
\begin{eqnarray*} L & = & \frac{\widehat R}{\mathfrak q_{(1,1)} \cap \mathfrak q_{(1,3)} \cap \mathfrak q_{(2,2)}} \oplus \frac{\widehat R}{\mathfrak q_{(1,1)} \cap \mathfrak q_{(1,2)} \cap \mathfrak q_{(2,1)}} \oplus  \frac{\widehat R}{\mathfrak q_{(1,2)} \cap \mathfrak q_{(2,1)}}  \oplus \\ & & \frac{\widehat R}{\mathfrak q_{(1,1)}} \oplus \frac{\widehat R}{\mathfrak q_{(1,3)} \cap \mathfrak q_{(2,2)}}\oplus \frac{\widehat R}{\mathfrak q_{(1,2)} \cap \mathfrak q_{(1,3)} \cap \mathfrak q_{(2,1)}\cap \mathfrak q_{(2,2)}}.
\end{eqnarray*}
 Since $L$ has constant rank $3$, there is an $R$-module $M$ such that $\widehat M \cong L$. Then $\add(M) \cong \ker(\mathsf A (M))\cap \N_0^{(6)}$ where $$\mathsf A (M)=\begin{bmatrix}1 & 0 & -1 & 1 & 0 & -1\\ 0 & 1 & 0 & 1 & -1 & -1\\ -1 & 1 & 1 & 0 & -1& 0\end{bmatrix}\sim \begin{bmatrix} 1 & 0 & -1 & 1 & 0 & -1\\ -1 & 1 & 1 & 0 & -1 & 0\\ 0 & 0 & 0 & 0 & 0 & 0\end{bmatrix}\sim\begin{bmatrix} 1 & 0 & -1 & 1 & 0 & -1\\ 0 & 1 & 0 & 1 & -1 & -1\\ 0 & 0 & 0 & 0 & 0 & 0\end{bmatrix}=J\mathsf A (M).$$ Thus $\add(M) \cong \ker(J\mathsf A (M))\cap \N_0^{(6)}\cong \ker \begin{bmatrix} 1 & 0 & -1 & 1 & 0 & -1\\ 0 & 1 & 0 & 1 & -1 & -1\end{bmatrix} \cap \N_0^{(6)}$. Since the matrix $\mathsf A (M)$ has rank two, the representation of $\add(M)$ as a Diophantine matrix defined by two equations more clearly describes this monoid. Moreover, since the map from $\mathbb Z^{(6)}$ to $\mathbb Z^{(2)}$ is surjective (the map $\mathsf A (M): \mathbb Z^{(6)} \rightarrow \mathbb Z^{(3)}$ is not surjective), we immediately see that $\Cl(\add(M))\cong \mathbb Z^{(2)}$,  and this isomorphism maps the set of  classes  containing prime divisors onto $$\left\{\begin{bmatrix}1 \\ 0\end{bmatrix}, \begin{bmatrix}0 \\ 1\end{bmatrix}, \begin{bmatrix}-1 \\ 0\end{bmatrix}, \begin{bmatrix}0 \\ -1\end{bmatrix}, \begin{bmatrix}1 \\ 1\end{bmatrix}, \begin{bmatrix}-1 \\ -1\end{bmatrix}\right\}.$$
\end{example}

\medskip
\begin{example} \label{cube-and-its-negative}
 We now consider a monoid $\add(M)$ which  generalizes  the monoid $\mcT(R)$ when $R$ has finite representation type,
and its arithmetic is studied in Theorem \ref{7.8} and Corollary \ref{7.10}. Suppose that  $\widehat R$ has $q+1$ minimal primes $\mathfrak q_1, \ldots , \mathfrak q_{q+1}$, and set
 \[
 L=\bigoplus_{\emptyset\not=I \subset [1,  q+1]} \frac{\widehat R}{\cap _{i \in I}\mathfrak q_i} \,.
 \] From the symmetry of the set of ranks of the indecomposable cyclic $\widehat R$-modules $\frac{\widehat R}{\cap _{i \in I}\mathfrak q_i}$ we immediately see that $L$ has constant rank $(2^q, \ldots , 2^q)$ and is therefore extended from some $R$-module $M$. Then $\add(M)\cong \ker(\mathsf A (M))\cap \N_0^{(q)}$ where $\mathsf A (M)$ is an $q \times 2^{q+1}-1$ integer-valued matrix with columns $\begin{bmatrix} \epsilon_1  & \cdots & \epsilon_q\end{bmatrix}^T$ where either $\epsilon_i \in \{0,1\}$ for all $i \in [1, q]$ or $\epsilon_i \in \{0,-1\}$ for all $i \in [1,q]$.

Since the columns of $\mathsf A (M)$ contain a basis for $\mathbb Z^{(q)}$, $\add(M) \hookrightarrow \add(L)\cong \N_0^{(q)}$ is a divisor theory with class group $\Cl(\add(M))\cong \mathbb Z^{(q)}$, and this isomorphism maps the set of  classes containing prime divisors onto \begin{equation*}\left(\left\{\begin{bmatrix} \epsilon_1  & \cdots & \epsilon_q\end{bmatrix}^T \,:\, \epsilon_i \in \{0,1\}\right\}\cup \left\{\begin{bmatrix} \epsilon_1  & \cdots & \epsilon_q\end{bmatrix}^T \,:\, \epsilon_i \in \{0,-1\}\right\}\right)\backslash \left\{\begin{bmatrix} 0 &  \cdots & 0\end{bmatrix}\right\}.\end{equation*}
\end{example}

\medskip
\begin{example} \label{bigcube}
In this example we construct a monoid $\add(M)$ which generalizes the monoid of Example \ref{cube-and-its-negative} by including all vectors having entries in $\{-1,0,1\}$ in the set $G_{\mathcal P}$. This larger set of classes containing prime divisors adds much complexity to the arithmetic. Suppose that $R$ has $q$ minimal primes and that $\widehat R$  has $2q$ minimal primes $$\mathfrak q_{(1,1)}, \mathfrak q_{(1,2)}, \mathfrak q_{(2,1)}, \ldots , \mathfrak q_{(q,2)}$$ where $\mathfrak q_{(i,j)}\cap R=\mathfrak q_{(i',j')}\cap R$ if and only if $i=i'$. As in the previous example, let $$L=\bigoplus_{\emptyset\not=I \subset \{(1,1), \ldots , (q,2)\}} \frac{\widehat R}{\cap _{(i,j) \in I}\mathfrak q_{(i,j)}}.$$ From the symmetry of the set of ranks of the indecomposable cyclic $\widehat R$-modules $\frac{\widehat R}{\cap _{\{i,j\} \in I}\mathfrak q_{i,j}}$ we immediately see that $L$ has constant rank $(2^{2q-1}, \ldots , 2^{2q-1})$ and is therefore extended from some $R$-module $M$. Then $\add(M)\cong \ker(\mathsf A (M))\cap \N_0^{(q)}$ where $\mathsf A (M)$ is an $q \times 2^{2q}-1$ integer-valued matrix with columns of the form $\begin{bmatrix}r_{(1,1)}-r_{(1,2)} & r_{(2,1)}-r_{(2,2)} & \cdots & r_{(q,1)}-r_{(q,2)}\end{bmatrix}^T$ where $(r_{(1,1)}, r_{(1,2)}, \ldots , r_{(q,2)})$ is the rank of one of the $2^{2q}-1$ indecomposable cyclic $\widehat R$-modules --- that is, any one of the $q$-tuples of $1$s and $0$s (not all $0$). In other words, the columns of $\mathsf A (M)$ are exactly the $3^{q}$ columns $\begin{bmatrix} \epsilon_1 &  \cdots & \epsilon_q\end{bmatrix}^T$ where $\epsilon_i \in \{-1,0,1\}$ for all $i \in [1,q]$, repeated with some multiplicity. For example, the column of all zeros occurs for each of the indecomposable cyclic $\widehat R$-modules $\frac{\widehat R}{\cap _{(i,j) \in I}\mathfrak q_{(i,j)}}$ where $(i,1) \in I$ if and only if $(i,2)\in I$.

Since the columns of $\mathsf A (M)$ contain a basis for $\mathbb Z^{(q)}$, $\add(M) \hookrightarrow \add(L)\cong \N_0^{(2^{2q}-1)}$ is a divisor theory whose class group $G \cong \Z^{(q)}$, and this isomorphism maps the set of  classes containing prime divisors onto  \begin{equation*}\label{2q}\left\{\begin{bmatrix} \epsilon_1  & \cdots & \epsilon_q\end{bmatrix}^T \,:\, \epsilon_i \in \{-1,0,1\}\right\}.\end{equation*}
\end{example}

\medskip

\begin{example} \label{splitexample-1}
In this example we consider $\add(M)$ when the completion of $M$ is isomorphic to a direct sum of some (but not all) of the indecomposable cyclic $\widehat R$-modules. In this case, the example is constructed in such a way that $\mcB(G_{\mathcal P})$ is a direct product of non-trivial submonoids (see Lemma \ref{4.2}). Suppose that $R$ has $q$ minimal primes and that  $\widehat R$  has $3q$ minimal primes $$\left\{\mathfrak q_{(i,j)} \,:\, i \in [1,  q], j\in [1,3]\right\}$$ where $\mathfrak q_{(i,j)}\cap R=\mathfrak q_{(i',j')}\cap R$ if and only if $i=i'$. Let $L$ be the $\widehat R$-module $$\bigoplus_{i=1}^q\left(\widehat R/\mathfrak q_{(i,1)}\oplus \widehat R/\mathfrak q_{(i,2)} \oplus \widehat R/\mathfrak q_{(i,3)}\oplus \widehat R/\left(\mathfrak q_{(i,1)}\cap \mathfrak q_{(i,2)}\right)\oplus \widehat R/\left(\mathfrak q_{(i,1)}\cap \mathfrak q_{(i,3)}\right) \oplus \widehat R/\left(\mathfrak q_{(i,2)}\cap \mathfrak q_{(i,3)}\right) \right).$$ We see immediately that $L$ has constant rank $(3,  \ldots , 3)$ and thus $L$ is extended from some $R$-module $M$. Then $\add(M)\cong \ker(\mathsf A (M))\cap \N_0^{(2q)}$ where $\mathsf A (M)$ is an $2q \times 6q$ integer-valued matrix with columns
 \[
 \left\{e_{2k-1}, e_{2k}, e_{2k-1}+ e_{2k}, -e_{2k-1}, -e_{2k},  -e_{2k-1}- e_{2k} \,:\,  k \in [1, q] \right\}
 \]
  where $(e_1, \ldots, e_{2q})$ denotes the  canonical basis  of $\mathbb Z^{(2q)}$.

For $k \in [1,q]$, we set $G_k = \{e_{2k-1}, e_{2k}, e_{2k-1}+e_{2k}, -e_{2k-1}, -e_{2k},- e_{2k-1}-e_{2k} \}$. Then $G_{\mathcal P} = \uplus_{k \in [1,q]} G_k$ is the set of classes containing prime divisors and $\mathcal B (G_{\mathcal P}) = \mathcal B (G_1) \times \cdots \times \mathcal B (G_q)$. From Proposition \ref{7.1} we will see that $\mathcal B (G_k) \hookrightarrow \mathcal F (G_k)$ is a divisor theory, whence $\mathcal B (G_{\mathcal P}) \hookrightarrow \mathcal F (G_{\mathcal P})$ and $\add (M) \hookrightarrow \add (L)$ are divisor theories. The arithmetic of this monoid is studied in Proposition \ref{7.12} and Corollary \ref{7.15}.

\end{example}

\begin{example} \label{splitexample-2}
As in Example \ref{splitexample-1}, suppose that  $R$ has $q$ minimal primes and suppose that the completion $\widehat R$ of $R$ has $3q$ minimal primes $$\left\{\mathfrak q_{(i,j)} \,:\, i \in [1,  q], j\in [1,3]\right\}$$ where $\mathfrak q_{(i,j)}\cap R=\mathfrak q_{(i',j')}\cap R$ if and only if $i=i'$. Further suppose that $\widehat R=S/(f)$ where $(S,\mathfrak n)$ is a regular Noetherian local ring of dimension two and where $0 \ne f\in \n$ and that $\widehat R/\mathfrak q_{(i,j)}$ has infinite representation type for all pairs $(i,j)$. By Proposition \ref{lotsofranks}, for each $k\in [1,q]$ there are indecomposable finitely generated torsion-free $\widehat R$-modules $M_k$ and $N_k$ of ranks $(r_{1,1}, \ldots , r_{q,3})$ and $(s_{1,1}, \ldots , s_{q,3})$ where $$r_{i,j}=\begin{cases}0, \quad i\not=k \\ 2, \quad i=k, j\in [1,2] \\ 0, \quad i=k, j=3  \end{cases}\text{ and } s_{i,j}=\begin{cases}0, \quad i\not=k \\ 3, \quad i=k, j=1 \\ 2, \quad i=k, j=2 \\ 0, \quad i=k, j=3  \end{cases}.$$ Moreover, by Corollary \ref{oppositeranks}, for each $k\in [1,q]$ there are constant $t_k\geq 2$ and $t_k'\geq 3$ and indecomposable finitely generated torsion-free $\widehat R$-modules $M_k'$ and $N_k'$ having ranks $(r'_{1,1}, \ldots , r'_{q,3})$ and $(s'_{1,1}, \ldots , s'_{q,3})$ where $$r'_{i,j}=\begin{cases}t_k, \quad i\not=k \\ t_k-2, \quad i=k, j\in [1,2] \\ t_k, \quad i=k, j=3  \end{cases}\text{ and } s'_{i,j}=\begin{cases}t_k', \quad i\not=k \\ t_k'-3, \quad i=k, j=1 \\ t_k'-2, \quad i=k, j=2 \\ t_k', \quad i=k, j=3  \end{cases}.$$ Let $$L=\widehat R\oplus \left(\bigoplus_{k=1}^q \left(M_k\oplus N_k\oplus M_k'\oplus N_k'\right)\oplus\widehat R/(\mathfrak q_{(i,1)}\cap \mathfrak q_{(i,3)}) \oplus \widehat R/(\mathfrak q_{(i,1)}\cap \mathfrak q_{(i,2)})\oplus \widehat R/\mathfrak q_{(i,2)}\oplus \widehat R/\mathfrak q_{(i,3)}\right).$$ Since $L$ has constant rank, $L$ is extended  from an $R$-module $M$.
Then $\add(M)\cong \ker(\mathsf A (M))\cap \N_0^{(2q)}$ where $\mathsf A (M)$ is an $2q \times (8q+1)$ integer-valued matrix with columns $0$ and
 \[
 \left\{e_{2k-1}, e_{2k}, 2e_{2k}, e_{2k-1}+2e_{2k}, -e_{2k-1}, -e_{2k}, -2e_{2k}, -e_{2k-1}-2e_{2k} \,:\,  k \in [1, q] \right\}
 \]
  where $(e_1, \ldots, e_{2q})$ denotes the  canonical basis  of $\mathbb Z^{(2q)}$.

For $k \in [1,q]$, we set $G_k = \{e_{2k-1}, e_{2k}, 2e_{2k}, e_{2k-1}+2e_{2k}, -e_{2k-1}, -e_{2k}, -2e_{2k}, -e_{2k-1}-2e_{2k} \}$. Then $G_{\mathcal P} = \uplus_{k \in [1,q]} G_k$ is the set of classes containing prime divisors and $\mathcal B (G_{\mathcal P}) = \mathcal B (G_1) \times \cdots \times \mathcal B (G_q)$. From Proposition \ref{7.1} we will see that $\mathcal B (G_k) \hookrightarrow \mathcal F (G_k)$ is a divisor theory, whence $\mathcal B (G_{\mathcal P}) \hookrightarrow \mathcal F (G_{\mathcal P})$ and $\add (M) \hookrightarrow \add (N)$ are divisor theories. The arithmetic of this monoid is studied in Proposition \ref{7.13} and Corollary \ref{7.15}.
\end{example}

\medskip
\begin{example}\label{FRTmonoid_generalization}
In our final example we  construct a tuple $(G, G_{\mathcal P})$ which generalizes the monoid $\mcT(R)$ when $R$ has finite representation type (see Theorem \ref{DivisorTheory2}). The arithmetic of such Krull monoids is studied in Theorem \ref{7.4} and Corollary \ref{7.10}. Suppose that $\widehat R$ has $q+1$ minimal primes $\mathfrak q_1,  \ldots , \mathfrak q_{q+1}$, and set
 \[
 L=\bigoplus_{j=1}^{q+1}\left(\left(\widehat R/\cap_{i\not= j}\mathfrak q_i\right) \oplus \left(\widehat R/\mathfrak q_j\right)\right) \,.
 \]
 Note that $L$ has constant rank $(q,  \ldots , q)$ and is hence extended from some $R$-module $M$. Then $\add(M)\cong \ker(\mathsf A (M))\cap \N_0^{(q)}$ where $\mathsf A (M)$ is an $q \times 2q$ integer-valued matrix with columns
 \[e_1,  \ldots, e_q, e_0=e_1+e_2+\cdots +e_q, -e_1,  \ldots , -e_q,-e_0 \,.
 \]
By Proposition \ref{7.1}, $\add(M) \hookrightarrow \add(L)\cong \N_0^{(2q)}$ is a divisor theory with class group $\Cl(\add(M))\cong \mathbb Z^{(q)}$, and this isomorphism maps the set of  classes containing prime divisors onto \begin{equation*}\{e_0=e_1+\cdots +e_q,e_1, \ldots , e_q,  -e_0, \ldots , -e_q\}.\end{equation*}
\end{example}

\bigskip
\section{Monoids of modules: Class groups and distribution of prime divisors II} \label{6}
\bigskip

In this section we investigate the characteristic of the Krull monoids $\mathcal T (R)$ and $\mathfrak C (R)$ for two-dimensional Noetherian local Krull domains (see Theorem \ref{dimension2classgroup}).  We will show that, apart from a well-described exceptional case, their class groups are both isomorphic to the factor group $\mathcal C (\widehat R)/\iota (\mathcal C (R))$, where $\iota \colon \mathcal C (R) \to \mathcal C (\widehat R)$ is the natural homomorphism between the class groups of $R$ and $\widehat R$. In a well-studied special case where  $R$ is factorial and $\widehat R$ is a hypersurface with finite representation type, this factor group is a finite cyclic group (see Theorem \ref{2dimensional-classgroups}). This is in strong contrast to the results on one-dimensional rings in the previous section where all class groups are torsion-free.

Let $S$ be a Krull domain and let $\mathcal I_v^* (S)$ denote the monoid of nonzero divisorial ideals. Then $\varphi \colon S \to \mathcal I_v^* (S)$, defined by $a \mapsto aS$, is a divisor theory. In this section we view $\mathcal C (S)$ as the class group of this specific divisor theory.
First we give a classical result (see \cite[Chapter VII, Section 4.7]{Bo88}).

\medskip
\begin{lemma}\label{basicclassgroupfacts} Let $S$ be a Noetherian Krull domain. One can associate to each finitely generated $S$-module $M$ a class $\mathbf{c}  (M) \in \mathcal C (S)$ in such a way that
\begin{enumerate}[(1)]
\item  If $0 \to M' \to M \to M'' \to 0$ is an exact sequence of finitely generated $S$-modules, then $\mathbf{c} (M) = \mathbf{c} (M')+\mathbf{c} (M'')$.

\item  If $I$ is a fractional ideal of $S$ and $I_v$ the divisorial ideal generated by $I$, then $\mathbf{c} (I) = \mathbf{c} (I_v)$.
\end{enumerate}
\end{lemma}

Note that if $S$ is any Noetherian domain, every ideal of $S$ is obviously an indecomposable finitely generated torsion-free $S$-module. If, in addition, the ring has  dimension two, then we have the following stronger result.

\begin{lemma}\label{2dimensional-extended}
Let $S$ be a Noetherian local Krull domain of  dimension two.
\begin{enumerate}[(1)]
\item \cite[Lemma 1.1 and Theorem 3.6]{Ev-Gr85a} Every divisorial ideal of $S$ is an indecomposable MCM $S$-module.
\item \cite[Proposition 3]{Ro-We-Wi99} In addition, assume that the $\m$-adic completion $\widehat S$ of $S$ is a Krull domain. Then a finitely generated torsion-free $\widehat S$-module $N$ is extended from an $S$-module if and only if $\mathbf{c}( N)$ is in the image of the natural homomorphism $\iota: \Cl(S) \rightarrow \Cl(\widehat S)$.
\end{enumerate}
\end{lemma}

We now give a result on abstract Krull monoids which encapsulates the structure of the monoids of modules described in Theorem \ref{dimension2classgroup}.

\begin{lemma}\label{2.5.1}
Let $D = \mathcal F (\mathcal P)$ be a free abelian monoid, $G$ an additive abelian group, $\psi: D \rightarrow G$ a homomorphism, and $H=\psi^{-1}(0) \subset D$.
\begin{enumerate}[(1)]
\item If $H \subset D$ is cofinal, then the inclusion $H \hookrightarrow D$ is a divisor homomorphism and $\overline{\psi}: D/H \rightarrow \psi (D) \subset G$ given by $\overline{\psi}\left([a] \right)=\psi(a)$  is an isomorphism.

\item  The inclusion $H \hookrightarrow D$ is a divisor theory if and only if $\langle \psi(\mathcal P)\rangle=[\psi(\mathcal P\backslash\{q\})]$ for every $q\in \mathcal P$. If this is the case, then $\overline{\psi}:D/H\rightarrow \psi(D)$ is an isomorphism and, for every $g \in \psi(D)$, the set $\mathcal P\cap \psi^{-1}(g)$ is the set of prime divisors in the class $\overline{\psi}^{-1}(g)$.

\item If the restriction $\psi \mid_{\mathcal P} \colon \mathcal P \to G$ of $\psi$ to $\mathcal P$ is an epimorphism, then $H \hookrightarrow D$ is cofinal. Moreover, it is a divisor theory apart from the following exception:
    \[
    G = \{0, g \} \quad \text{and} \quad |\mathcal P \cap \psi^{-1} (g)| = 1 \,.
    \]
    If $H \hookrightarrow D$ is not a divisor theory, then $H$ is factorial.
\end{enumerate}
\end{lemma}

\begin{proof}
For the proofs of (1) and (2), see \cite[Proposition 2.5.1]{Ge-HK06a}.
We now consider the proof of (3).

Let $a \in D$. Since $\psi \mid_{\mathcal P} \colon \mathcal P \to G$ is an epimorphism, there exists $p \in \mathcal P \subset D$ such that $\psi (p) = - \psi (a)$. Therefore $a p \in H$ and the inclusion $H\hookrightarrow D$ is cofinal. In order to show that $H \hookrightarrow D$ is a divisor theory we distinguish three cases. First suppose that $|G| = 1$. Then $|D/H| = 1$, and hence $H = D$. Next suppose that $|G| > 2$. By (2) we must verify that
\[
\psi (q) \in  [\psi(\mathcal P\backslash\{q\})] \quad \text{for every} \ q \in \mathcal P \,.
\]
Let $q \in \mathcal P$. Since $|G|> 2$, there exist $g_1, g_2 \in G \setminus \{0, \psi (q)\}$ with $\psi (q) = g_1+g_2$. Since the restriction $\psi \mid_{\mathcal P} \colon \mathcal P \to G$ is an epimorphism, there exist $p_1, p_2 \in \mathcal P \setminus \{q\}$ with $\psi (p_i) = g_i$ for $i \in [1,2]$. Therefore
\[
\psi (q) = g = g_1+g_2 = \psi (p_1)+\psi (p_2) \in [\psi(\mathcal P\backslash\{q\})] \,.
\]

Finally, suppose that $|G| = 2$. Then $H \hookrightarrow D$ is a divisor theory if and only if $ \langle \psi(\mathcal P)\rangle=[\psi(\mathcal P\backslash\{q\})]$ for every $q\in \mathcal P$ if and only if there exist distinct $q_1,q_2 \in \mathcal P$ such that $\psi (q_i) \ne 0$ for $i\in [1,2]$.
Clearly, if $q \in \mathcal P$ is the unique element of $\mathcal P$ with $\psi (q) \ne 0$, then $H$ is free abelian with basis $\mathcal P \setminus \{q\} \cup \{q^2\}$.
\end{proof}

 We  are now able to determine both the class group and the set of classes containing prime divisors for the monoids $\mcT(R)$ and $\mcC(R)$. This generalizes and refines the results of \cite{Ba09a}. Since each divisorial ideal over a two-dimensional local ring is MCM and thus finitely generated and torsion-free,  Theorem \ref{dimension2classgroup} can be stated in parallel both for $\mcT(R)$ and $\mcC(R)$.

\medskip
\begin{theorem}\label{dimension2classgroup}
Let $(R,\m)$  be a Noetherian local Krull domain of dimension two whose  $\m$-adic completion $\widehat R$ is also a Krull domain.
Let $\mathcal V(R)$ (respectively $\mathcal V(\widehat R)$) denote either $\mcT(R)$ (resp. $\mcT(\widehat R)$) or $\mcC(R)$ (resp. $\mcC(\widehat R)$), and let $\iota \colon \mathcal C (R) \to \mathcal C ( \widehat R)$ be the natural map.
\begin{enumerate}[(1)]
\item The embedding $\mathcal V(R)\hookrightarrow \mathcal V(\widehat R)$  is a cofinal divisor homomorphism. The class group of this divisor homomorphism is isomorphic to $G = \mathcal C (\widehat R)/\iota (\mathcal C (R))$ and every class contains a prime divisor. Moreover the embedding is a divisor theory except if $\widehat R$ satisfies the following condition:
\begin{itemize}
\item[{\bf (E)}] $|G| = 2$ and, up to isomorphism, there is precisely one non-extended indecomposable $\widehat R$-module $M$ with $[M] \in \mathcal V (\widehat R)$.
\end{itemize}
In particular,  $\mathcal V (R)$ satisfies KRSA if and only if either $|G|=1$ or $\widehat R$ satisfies {\bf (E)}.

\item
Suppose that the embeddings $\mathcal T (R)\hookrightarrow \mathcal T (\widehat R)$ and $\mathfrak C (R)\hookrightarrow \mathfrak C (\widehat R)$ are both divisor theories. Then their class groups are isomorphic. If $(G, (m_g)_{g\in G})$ is the characteristic of $\mcT(R)$ and $(G, (n_g)_{g\in G})$ is the characteristic of $\mcC(R)$, then $m_g\geq n_g$ for all $g \in G$. Moreover, $\sum_{g \in G} m_g$ infinite.
\end{enumerate}
\end{theorem}

\begin{proof}
We set $D = \mathcal F (\mathcal P) =\mathcal V(\widehat R)$.
By (1) of Lemma \ref{basicclassgroupfacts} there is a homomorphism
\[
\psi \colon \mathcal V(\widehat R)  \rightarrow  \mathcal C(\widehat R)/\iota (\mathcal C(R))=G \quad \text{ given by} \quad [M]  \mapsto  \mathbf{c} (M) + \iota(\mathcal C(R)) \,.
\]
 By (1) of Lemma \ref{2dimensional-extended}, every divisorial ideal of $\widehat R$ is an indecomposable MCM $\widehat R$-module. That is, the  class of each divisorial ideal of $\widehat R$ in $\mathcal C ( \widehat R)$ is the image of some $[M] \in \mathcal V (\widehat R)$, where $M$ is an indecomposable MCM $\widehat R$-module. In other words, $\psi$ restricted to $\mathcal P = \mathcal A (D)$ is an epimorphism and $\psi(\mathcal A(D))=\mathcal C(\widehat R)/ \iota (\mathcal C(R))$.

By Lemma \ref{2.5.1}, the inclusion  $H=\psi^{-1}(0) \subset D$ is a cofinal divisor homomorphism. By (2) of Lemma \ref{2dimensional-extended}, $H$ is the image of the embedding $\mathcal V(R) \hookrightarrow \mathcal V(\widehat R)$, and thus the embedding $\mathcal V(R) \hookrightarrow \mathcal V(\widehat R)$ is a divisor theory if and only if the inclusion $H \hookrightarrow D$ is a divisor theory. By Lemma \ref{2.5.1} this always holds apart from the described Exception {\bf (E)}. A Krull monoid  is factorial if and only if its class group is trivial. Thus, if $\mathcal V(R) \hookrightarrow \mathcal V(\widehat R)$ is a divisor theory then KRSA holds for $\mathcal V(R)$ if and only if $|G|=0$. If $\mathcal V(R) \hookrightarrow \mathcal V(\widehat R)$ is not a divisor theory, then the inclusion $H\hookrightarrow D$ is not a divisor theory. By Lemma \ref{2.5.1}, $H$ is factorial, whence $\mathcal V(R)$ is factorial.

Since each MCM $R$-module is finitely generated and torsion-free, it is clear that $m_g \ge n_g$ for each $g \in G$. From \cite{Ba62a} we know that there infinitely many nonisomorphic indecomposable finitely generated torsion-free $\widehat R$-modules, and therefore $\sum_{g \in G} m_g$ infinite.
\end{proof}

\smallskip
Let $R$ be as in the above theorem and assume in addition that
\begin{itemize}
\item $R$  contains a field and $k = R/\m$ is algebraically closed with characteristic zero.
\item $\widehat R$ is a hypersurface i.e., $\widehat R$ is isomorphic to a three-dimensional regular Noetherian local ring modulo a regular element.
\item   $R$ has finite representation type.

\end{itemize}
Such rings were classified in  \cite{Bu-Gr-Sc87} and \cite{Kn87a} and are given, up to isomorphism, in Table \ref{hypers}. Note that since $\widehat R$ has finite representation type and each divisorial ideal of $\widehat R$ is and indecomposable MCM $\widehat R$-module, $\mathcal C ( \widehat R)$ and hence $\mathfrak C (R)$ is  finite.

\medskip
\begin{center}
\begin{longtable}{lclr}\caption{Two-dimensional Rings with Finite Representation Type}\label{hypers}\\
 \endhead \endfoot ($A_n$) & \hspace{.25in} &
$k[[x,y,z]]/(x^2+y^2+z^{n+1})$ & $n \geq 1$
\\ ($D_n$) & & $k[[x,y,z]]/(x^2z+y^2+z^{n-1})$ & $n \geq 4$ \\
($E_6$) & & $k[[x,y,z]]/(x^3+y^2+z^4)$ & \\ ($E_7$) & &
$k[[x,y,z]]/(x^3+xz^3+y^2)$ & \\ ($E_8$) & &
$k[[x,y,z]]/(x^2+y^3+z^5)$ &
\end{longtable}
\end{center}

An amazing theorem of Heitmann \cite{He93a} gives the existence of a local factorial domain whose completion is a ring as in Table 1.  In this situation we can determine the characteristic of $\mathfrak C (R)$.

\smallskip
\begin{theorem}\label{2dimensional-classgroups}
Let $(R,\m)$ be a Noetherian local  factorial domain with $\m$-adic completion $\widehat R$ isomorphic to a ring in Table \ref{hypers}.
\begin{enumerate}[(1)]
\item If $\widehat R$ is a ring of type ($A_n$), then $\mathcal C(\mcC(R))$ is cyclic of order $n+1$ and each class contains exactly one prime divisor.

\smallskip
\item Suppose $\widehat R$ is a ring of type ($D_n$).
\begin{enumerate}
\item If $n$ is even, then $\mathcal C(\mcC(R)) \cong C_2 \oplus C_2$. The trivial class  contains $\frac{n}{2}$ prime divisors. Two non-trivial classes  each contain a single prime divisor and their sum contains $\frac{n-2}{2}$ prime divisors.
\item If $n$ is odd, then $\mathcal C(\mcC(R))$ is cyclic of order four. The classes of order four each contain a single prime. The remaining classes  each contain $\frac{n-1}{2}$ prime divisors.
\end{enumerate}

\smallskip
\item If $\widehat R$ is a ring of type ($E_6$), then $\mathcal C(\mcC(R))$ is cyclic of order three. The trivial class  contains three prime divisors, while each remaining class contains two prime divisors.

\smallskip
\item If $\widehat R$ is a ring of type ($E_7$), then $\mathcal C(\mcC(R))$ is cyclic of order two. The trivial class  contains five prime divisors and the non-trivial class  contains three prime divisors.

\smallskip
\item If $\widehat R$ is a ring of type ($E_8$), then $\mathcal C(\mcC(R))$ is trivial, with the trivial class  containing all nine prime divisors.
\end{enumerate}
\end{theorem}

\begin{proof}
The class groups $\mathcal C(\widehat R)$ for $\widehat R$ a ring listed in Table \ref{hypers} were given in \cite{Br68a}. Since $R$ is factorial, $\mathcal C(R)=0$ and by Theorem \ref{dimension2classgroup}, $\mathcal C(\mcC(R))\cong \mathcal C(\widehat R)$. Following the proof of \cite[Theorem 4.3]{Ba09a} one can compute the class of each indecomposable $\widehat R$-module in $\mathcal C(\widehat R)$ by using the Auslander-Reiten sequence for $\widehat R$. The result follows by considering the map $\overline{\psi}$ defined in the proof of Theorem \ref{dimension2classgroup}.
\end{proof}

The above theorem completely determines the characteristic of the monoid $\mcC(R)$. With this information, in addition to being able to completely describe the arithmetic of $\mcC(R)$ as we do in Theorem \ref{7.9}, we can easily enumerate the atoms of $\mcC(R)$ (the nonisomorphic indecomposable MCM modules). We now illustrate this ability with an example. If $\boldsymbol \beta \colon \mcC(R) \rightarrow \mathcal B(G_{\mathcal P})$ is the transfer homomorphism of Lemma \ref{4.4}, then $\mathcal A (\mcC(R) ) = \boldsymbol \beta^{-1}(\mathcal A(\mathcal B(G_{\mathcal P}))$. Suppose that  $\widehat R$ is a ring of type ($D_n$) with $n$ even. Then $\widehat R$ has exactly $n+1$ nonisomorphic indecomposable MCM modules. If $C_2\oplus C_2=\{0,e_1, e_2, e_1+e_2\}$, then
\[
\mathcal A(C_2\oplus C_2) = \{0, e_1^2, e_2^2, (e_1+e_2)^2, e_1e_2(e_1+e_2) \}\,,
\]
and hence $R$ has exactly
\[
|\mathcal A (\mcC(R))| = \frac{n}{2}+1\cdot 1+1\cdot 1+\frac{n-2}{2}\cdot \frac{n-2}{2}+1\cdot1\cdot \frac{n-2}{2}=\frac{n^2+8}{4}
\]
nonisomorphic indecomposable MCM modules.

\smallskip
We conclude this section by noting that a two-dimensional local Krull domain $(R,\m)$ having completion isomorphic to a ring in Table \ref{hypers} may not be factorial. However,  Theorem \ref{dimension2classgroup} implies that $\mathcal C(\mcC(R))$ is a factor group of a  group given in Theorem \ref{2dimensional-classgroups}. In particular, $\mathcal C ( \mathfrak C (R))$ is a finite cyclic group such that every class contains a prime divisor, and thus the arithmetic of $\mathfrak C (R)$ is described in Theorem \ref{7.9}.

\bigskip
\section{The arithmetic of monoids of modules} \label{7}
\bigskip

In this section we study the arithmetic of the Krull monoids  that have been discussed in Sections \ref{5} and \ref{6}. Thus, using the transfer properties presented in Section \ref{3}, we describe the arithmetic of direct-sum decompositions of modules. Suppose that  $H$ is a Krull monoid having a divisor homomorphism $\varphi \colon H \to \mathcal F (\mathcal P)$ and let $G_{\mathcal P} \subset \mathcal C (\varphi)$ be the set of classes containing prime
divisors. The first subsection deals with quite general sets $G_{\mathcal P}$ and provides results on the finiteness or non-finiteness of various arithmetical parameters. The second subsection studies three specific sets $G_{\mathcal P}$, provides explicit results on arithmetical parameters, and establishes a characterization result (Theorems \ref{7.4}, \ref{7.8}, \ref{7.9}, and Corollary \ref{7.10}) . The third subsection completely determines the system of sets of lengths in case of small subsets $G_{\mathcal P}$. It shows that small subsets in torsion groups and in torsion-free groups can have the same systems of sets of lengths, and it reveals natural limits for arithmetical characterization results (Corollary \ref{7.15}).

\

\subsection{General sets $G_{\mathcal P}$ of  classes containing prime divisors.} \label{6a}
In this subsection we consider the algebraic and arithmetic structure of Krull monoids with respect to $ G_{\mathcal P}$. We will often assume that  $G_{\mathcal P} = - G_{\mathcal P}$, a property which has a strong influence on the arithmetic of $H$. Recall that $G_{\mathcal P} = - G_{\mathcal P}$ holds in many of the (finite and infinite representation type) module-theoretic contexts described in Sections \ref{5} and \ref{6}.
More generally, all configurations $(G, G_{\mathcal P})$ occur for certain monoids of modules (see Proposition \ref{realization}, \cite{He-Pr10a}, and \cite[Chapter 1]{Le-Wi12a}) and, by Claborn's Realization Theorem, all configurations $(G, G_{\mathcal P})$ occur for Dedekind domains (see  \cite[Theorem 3.7.8]{Ge-HK06a}). In addition, every abelian group can be realized as the class group of   a Dedekind domain which is a quadratic extension of a principal ideal domain, and in this case we have $G_{\mathcal P} = - G_{\mathcal P}$ (see \cite{Le72a}).

\medskip
\begin{proposition} \label{7.1}
Let $H$ be a Krull monoid, $\varphi \colon H \to  \mathcal F ({\mathcal P} )$
a  divisor homomorphism with class group $G = \mathcal C (\varphi)$, and let $G_{\mathcal P} \subset G$ denote the set of classes containing prime
divisors.
\begin{enumerate}[(1)]
\item If $G_{\mathcal P}$ is finite, then $\mathcal A (G_{\mathcal P})$ is finite and hence $\mathsf D (G_{\mathcal P}) < \infty$. If $G$ has finite total rank, then  $G_{\mathcal P}$ is finite if and only if $\mathcal A (G_{\mathcal P})$ is finite if and only if $\mathsf D (G_{\mathcal P}) < \infty$.

\smallskip
\item If $G_{\mathcal P} = -G_{\mathcal P}$, then $[G_{\mathcal P}] = G$.  Moreover, the map $\varphi \colon H \to D$ and the inclusion $\mathcal B (G_{\mathcal P}) \hookrightarrow \mathcal F (G_{\mathcal P})$ are both cofinal.

\smallskip
\item Suppose that $G$ is infinite cyclic, say $G = \langle e \rangle$, and that $\{-e, e \} \subset G_{\mathcal P}$. Then $\mathcal B (G_{\mathcal P}) \hookrightarrow \mathcal F (G_{\mathcal P})$ is a divisor theory if and only if there exist $k, l \in \N_{\ge 2}$ such that $-ke, le \in G_{\mathcal P}$.

\smallskip
\item Let $r, \alpha \in \mathbb N$ with $r+\alpha>2$. Let  $(e_1, \ldots, e_r) \in G_{\mathcal P}^r$ be independent and let $e_0 \in G_{\mathcal P}$ such that  $\alpha e_0 = e_1+ \cdots + e_r $, $\{-e_0, \ldots, -e_r\} \subset G_{\mathcal P}$,  and  $\langle  e_0, \ldots, e_r \rangle = G$.
    \begin{enumerate}
    \item The map $\varphi \colon H \to \mathcal F ({\mathcal P})$ and the inclusion $\mathcal B (G_{\mathcal P}) \hookrightarrow \mathcal F (G_{\mathcal P})$ are divisor theories with class group isomorphic to $G$.

    \smallskip
    \item If $0 \notin G_{\mathcal P}$, then $\mathcal B (G_{\mathcal P})$ is not a direct product of nontrivial submonoids.
    \end{enumerate}
\end{enumerate}
\end{proposition}

\begin{proof}
(1) follows from \cite[Theorem 3.4.2]{Ge-HK06a}.

\smallskip
If $G_{\mathcal P} = -G_{\mathcal P}$, then $[G_{\mathcal P}] = \langle G_{\mathcal P} \rangle = G$. By Lemma \ref{4.4}, (2) follows once we verify that $\varphi$ is cofinal. If $p \in {\mathcal P}$, then there is a $q \in {\mathcal P}$ with $q \in - [p ]$, whence there is an $a \in H$ with $\varphi (a) = pq$, and so $\varphi$ is cofinal.

\smallskip
If $\{-ke \, \colon \, k \in \mathbb N\} \cap G_{\mathcal P} = \{-e\}$ or $\{ke \, \colon \, k \in \mathbb N\} \cap G_{\mathcal P} = \{e\}$, then $\mathcal B (G_{\mathcal P})$ is factorial. Since $\mathcal F (G_{\mathcal P}) \ne  B (G_{\mathcal P})$, the inclusion $\mathcal B (G_{\mathcal P}) \hookrightarrow  F (G_{\mathcal P})$ is not a divisor theory. Conversely, suppose that there exist $k, l \in \N_{\ge 2}$ such that $-ke, le \in  G_{\mathcal P}$. Let $m \in \N$. If $me \in G_{\mathcal P}$, then $m e = \gcd \big( (me)(-e)^m, (me)^k(-ke)^m \big)$, and if $-me \in G_{\mathcal P}$, then $-me = \gcd \big( (-me)e^m, (-me)^l (le)^m \big)$. Thus every element of $G_{\mathcal P}$ is  a greatest common divisor of a finite set of elements from $\mathcal B (G_{\mathcal P})$ and hence $\mathcal B (G_{\mathcal P}) \hookrightarrow  F (G_{\mathcal P})$ is  a divisor theory.

\smallskip
We now suppose that $G$ and $G_{\mathcal P}$ are as in (4). To prove (a) it is sufficient to show that $\mathcal B (G_{\mathcal P}) \hookrightarrow \mathcal F (G_{\mathcal P})$ is a divisor theory. By \cite[Proposition 2.5.6]{Ge-HK06a} we need only verify that $\langle G_{\mathcal P} \rangle = [G_{\mathcal P} \setminus \{g\}]$ for every $g \in G_{\mathcal P}$. Let $g \in G$. We will show that $[G_{\mathcal P} \setminus \{g\}] = [G_{\mathcal P}] = [e_0, \ldots, e_r, -e_0, \ldots, -e_r]$. If $g \notin \{e_0, \ldots, e_r, -e_0, \ldots, -e_r\}$, the assertion is clear. By symmetry it suffices to consider the case where $g \in \{e_0, \ldots, e_r\}$. If $g = e_i$ for some $i \in [1,r]$, then $e_i = \alpha e_0 + (-e_1) +   \cdots + (-e_{i-1}) + (-e_{i+1}) +  \cdots +  (-e_r) \in [ \{ \pm e_{\nu} \, \colon \nu \in [0,r] \setminus \{i\} ]$, and hence $[G_{\mathcal P} \setminus \{g\}] \supset [ \{ \pm e_{\nu} \, \colon \nu \in [0,r] \setminus \{i\} \}] = [ \pm e_0, \ldots , \pm e_r] = G$.
If $g = e_0$, then $e_0 = e_1 + \cdots + e_r + (\alpha-1)(-e_0) \in [G_{\mathcal P} \setminus \{e_0\}]$, and hence $[G_{\mathcal P} \setminus \{g\}] = [G_{\mathcal P}] = G$.

To prove (b) we use Lemma \ref{4.2} and suppose that $0 \notin G_{\mathcal P}$ with $G_{\mathcal P} = G_1 \uplus G_2$ such that $\mathcal A (G_{\mathcal P}) = \mathcal A (G_1) \uplus \mathcal A (G_2)$. We must show that either $G_1$ or $G_2$ is empty.
Suppose that $V = e_1 \cdot \ldots \cdot e_r (-e_0)^{\alpha} \in \mathcal A (G_1)$. Since $(-e_0)e_0, \ldots , (-e_r)e_r \in \mathcal A (G_{\mathcal P})$, it follows that $\{\pm e_0, \ldots, \pm e_r \} \subset G_1$. Let $g \in G_{\mathcal P}$. Since $G_{\mathcal P} \subset G = [\pm e_0, \ldots, \pm e_r]$, there exists $U \in \mathcal A (G_{\mathcal P})$ such that $g \in \supp (U) \subset \{g, \pm e_0, \ldots, \pm e_r \}$, and hence $g \in G_1$. Thus $G_1 = G_{\mathcal P}$ and $G_2 = \emptyset$.
\end{proof}

 For our characterization results, we need to recall the concept of an absolutely irreducible element, a classical notion in algebraic number theory. An element $u$ in an atomic monoid $H$ is called {\it absolutely irreducible} if $u \in \mathcal A (H)$ and $|\mathsf Z (u^n)| = 1$ for all $n \in \mathbb N$; equivalently, the divisor-closed submonoid of $H$ generated by $u$ is factorial.  Suppose that $H \hookrightarrow \mathcal F (\mathcal P)$ is a divisor theory with class group $G$ and that  $u = p_1^{k_1} \cdot \ldots \cdot p_m^{k_m}$ where $m, k_1, \ldots, k_m \in \N$ and where $p_1, \ldots, p_m \in \mathcal P$  are pairwise distinct. Then $u$ is absolutely irreducible.
if and only if $(k_1, \ldots, k_m)$ is a minimal element of the set
\[
\Gamma = \{(s_1, \ldots, s_m) \in \N^l_0 \, \colon \, p^{s_1}_1
\cdot\ldots\cdot p^{s_m}_m \in H\} \setminus \{\boldsymbol{0}\}
\]
with respect to the usual product ordering, and the torsion-free rank of $\langle
[p_1], \ldots, [p_m]\rangle$ in $G$ is $m-1$ (see \cite[Proposition 7.1.4]{Ge-HK06a}). In particular, if $[p_1] \in G$ has finite order, then $p_1^{\ord([p_1])}$ is absolutely irreducible, and if $[p_1] \in G$ has infinite order, then $p_1q_1$ is absolutely irreducible for all $q_1 \in \mathcal P \cap (-[p_1])$.

\medskip
\begin{proposition} \label{7.2}
Let $H$ be a Krull monoid, $\varphi \colon H \to \mathcal F ({\mathcal P})$ a cofinal divisor homomorphism with   class group $G$, and let $G_{\mathcal P} \subset G$ denote the set of classes containing prime divisors.
\begin{enumerate}[(1)]
\item Suppose that  $G_{\mathcal P}$ is infinite.
      \begin{enumerate}
      \item If $G_{\mathcal P}$ has an infinite subset $G_0$ such that   $G_0  \cup  (-G_0) \subset G_{\mathcal P}$ and $\langle G_0 \rangle$ has finite total rank, then  $\mathcal U_k (H)$ is infinite for each $k \ge 2$. Moreover, $\mathsf D (G_{\mathcal P}) = \rho_k (H) =   \omega (H) = \mathsf t (H) = \infty$.

      \smallskip
      \item If there exists $e \in G$ such that both $G_{\mathcal P} \cap \{ke \, \colon \, k \in \mathbb N\}$ and $G_{\mathcal P} \cap \{-ke \, \colon \, k \in \mathbb N\}$ are infinite, then $\Delta (H)$ is infinite and $\mathsf c (H) = \mathsf c_{\mon} (H) = \infty$.

      \smallskip
      \item     If $G_{\mathcal P}$ contains an infinite group, then every finite subset $L \subset \mathbb N_{\ge 2}$ occurs as a set of lengths.
      \end{enumerate}

\smallskip
\item Suppose that  $G_{\mathcal P}$ is finite  and that $H$ is not factorial.
      \begin{enumerate}
      \item The set $\Delta (H)$ is finite and there is a constant $M_1 \in \mathbb N$ such that every set of lengths is an {\rm AAMP}
            with difference $d \in \Delta^* (H)$ and bound $M_1$.

      \item  There is a constant $M_2 \in \mathbb N$ such that, for every $k \ge 2$, the set $\mathcal U_k (H)$ is an {\rm AAMP} with period $\{0, \min \Delta (H)\}$ and bound $M_2$.

      \smallskip
      \item $\mathsf c (H) \le \omega (H)  \le \mathsf t (H) \le 1 + \frac{\mathsf D (G_{\mathcal P})(\mathsf D (G_{\mathcal P})-1)}{2}$ \ and
      \[
       \mathsf c_{\mon} (H) < \frac{|G_{\mathcal P}^{\bullet}|+2}{2} \big( (2|G_{\mathcal P}^{\bullet}|+2)(|G_{\mathcal P}^{\bullet}|+2)(\mathsf D
      (G_{\mathcal P})+1) \big)^{|G_{\mathcal P}^{\bullet}|+1}  \,.
      \]

      \smallskip
      \item Suppose that $G_{\mathcal P} = -G_{\mathcal P}$. Then $\omega (H) = \mathsf D (G_{\mathcal P})$, $\rho (H) = \mathsf D (G_{\mathcal P})/2$, $\rho_{2k} (H) = k \mathsf D (G_{\mathcal P}) < \infty$,  and $\lambda_{k \mathsf D (G_{\mathcal P})+j} (H) = 2k+j$ for all $k \in \mathbb N$ and $j \in [0,1]$. If $G$ is torsion-free, then $\mathsf D (G_{\mathcal P})$ is the maximal number $s$ of absolutely irreducible atoms $u_1, \ldots, u_s$ such that $2 \in \mathsf L (u_1 \cdot \ldots \cdot u_s)$.

      \smallskip
      \item If, in particular, $G_{\mathcal P}=-G_{\mathcal P}$ and $\mathsf D (G_{\mathcal P})=2$, then $\mathcal U_k (H) = \{k\}$ for all $k \in \mathbb N$ and $\mathsf c_{\mon}(H) = \mathsf c (H) =  \omega (H) = \mathsf t (H) = 2$.
      \end{enumerate}
\end{enumerate}
\end{proposition}

\begin{proof}
Throughout the proof we implicitly assume the results of Lemma \ref{2.2} and of Lemma \ref{4.4}. In particular,  we have $\rho (H) \le \omega (H)$ and $\mathsf c (H) \le \omega (H) \le \mathsf D (G_{\mathcal P})$.

\smallskip
For (1), suppose that $G_{\mathcal P}$ is infinite. We first prove (a). Theorem 3.4.2 in \cite{Ge-HK06a} implies that $\mathcal A (G_0)$ and $\mathsf D (G_0)$ are infinite. Thus, for every $k \in \mathbb N$, there is $U_k \in \mathcal A (G_0)$ with $|U_k| \ge k$ and hence $\mathsf L \big( U_k (-U_k) \big) \supset \{2, |U_k| \}$. This implies that $\mathcal U_2 (G_0)$ is infinite and thus $\mathcal U_k (G_0)$ is infinite for all $k \ge 2$. Therefore  $\rho_k (H) = \rho_k (G_{\mathcal P}) = \infty$ for all $k \ge 2$ and, since $\rho (H) \le \omega (H) \le \mathsf t (H)$, each of these invariants is infinite.

(b) follows from \cite[Theorem 4.2]{Ge-Gr-Sc-Sc10}.

(c) is a realization result is due to Kainrath. See \cite{Ka99a} or \cite[Theorem 7.4.1]{Ge-HK06a}.

\smallskip
Now, in order to prove  (2), we suppose that $G_{\mathcal P}$ is finite and that $H$ is not factorial. Then $\mathsf D (G_{\mathcal P})>1$ and $2 \le \mathsf c (H) \le \omega (H) \le \mathsf t (H)$. By Proposition \ref{7.1}, $\mathcal B (G_{\mathcal P})$ is finitely generated and $\mathsf D (G_{\mathcal P}) < \infty$. The respective upper bounds given in (c) for $\mathsf c_{\mon} (H)$ and $\mathsf t (H)$ can be found in \cite[Theorem 3.4]{Ge-Yu13a} and \cite[Theorem 3.4.10]{Ge-HK06a}.

We now consider (a). Since $2+\sup \Delta (H) \le \mathsf c (H) < \infty$, (c) implies that $\Delta (H)$ is finite. Since $\mathcal B (G_{\mathcal P})$ is finitely generated, the assertion on the structure of sets of lengths follows from \cite[Theorem 4.4.11]{Ge-HK06a}.

Since $\mathsf t (H) < \infty$ and $\Delta (H)$ is finite,   \cite[Theorems 3.5 and 4.2]{Ga-Ge09b} imply the assertion in (b) on the structure of the unions of sets of lengths.

 In order to prove (d), we suppose  that $G_{\mathcal P} = -G_{\mathcal P}$. The statements about $\rho_{2k} (H)$, $\rho (H)$, and $\lambda_{k \mathsf D (G_{\mathcal P})+j}(H)$ follow from Lemma \ref{4.3}, and it remains to show that $\omega (H) = \mathsf D (G_{\mathcal P})$. We have $\omega (H) \le \mathsf D (G_{\mathcal P}) < \infty$. If $\mathsf D (G_{\mathcal P})=2$, then $\omega (H) = \mathsf D (G_{\mathcal P})$. Suppose that $\mathsf D (G_{\mathcal P}) \ge 3$. If  $V = g_1 \cdot \ldots \cdot g_l \in \mathcal A (G_{\mathcal P})$ with $|V| = l = \mathsf D (G_{\mathcal P})$ and $U_i = (-g_i)g_i$ for all $i \in [1,l]$, then $V \t U_1 \cdot \ldots \cdot U_l$ but yet $V$ divides no proper subproduct of $U_1 \cdot \ldots \cdot U_l$. Thus $\mathsf D (G_{\mathcal P}) \le \omega (G_{\mathcal P}) \le \omega (H)$.

Let $t$ denote the maximal number of absolutely irreducible atoms with the required property. Since $\rho (H) = \mathsf D (G_{\mathcal P})/2$, it follows that $t \le \mathsf D (G_{\mathcal P})$.  Let $V= g_1 \cdot \ldots \cdot g_l \in \mathcal A (G_{\mathcal P})$ with $|V| = l = \mathsf D (G_{\mathcal P})$. For $i \in [1, l]$ choose an element $p_i \in {\mathcal P} \cap g_i$ and an element $q_i \in {\mathcal P} \cap (-g_i)$. Since $G$ is torsion-free, the element $u_i = p_iq_i \in H$ is absolutely irreducible for each $i \in [1,l]$ and, by construction, we have $2 \in \mathsf L (u_1 \cdot \ldots \cdot u_l)$.

The statement in (e) follows immediately from (c) and (d).
\end{proof}

\smallskip
Let all notation be as in Proposition \ref{7.2}. We note that if $G_{\mathcal P}$ is infinite but without a subset $G_0$ as in (1a), then none of the conclusions of (1a) need hold. A careful analysis of the case where $G$ is an infinite cyclic groups is handled in \cite{Ge-Gr-Sc-Sc10}. We also note that the description of the structure of sets of lengths given in  (2a) is best possible (see \cite{Sc09a}).

\smallskip
By Lemma \ref{4.4}, many arithmetical phenomena of a Krull monoid $H$ are determined by the tuple $(G, G_{\mathcal P})$. We now provide a first result indicating that conversely arithmetical phenomena give us back information on the class group. Indeed, our next corollary  characterizes arithmetically whether the class group of a Krull monoid is torsion-free or not. To do so we must study the arithmetical behavior of elements similar to absolutely irreducible elements. Note that such a result cannot be accomplished via sets of lengths alone (see Propositions \ref{7.12} and \ref{7.13} and (1c) of Proposition \ref{7.2}; in fact, there is an open conjecture  that every abelian group is the class group of a half-factorial Krull monoid \cite{Ge-Go03}).

\medskip
\begin{proposition} \label{7.3}
Let $H$ be a Krull monoid with class group $G$. Then $G$ has an element of infinite order if and only if there exists an irreducible element $u \in H$ having the following two arithmetical properties.
\begin{enumerate}[(a)]
\item Whenever there are $v \in H \setminus H^{\times}$ and $m \in \N$ with $v \t u^m$, then $u \t v^n$ for some $n \in \N$.

\item There exist $l \ge 2$ and $a_1, \ldots, a_l \in H$ such that $u \t a_1 \cdot \ldots \cdot a_l$ but yet $u \nmid a_{\nu}^{-1} (a_1 \cdot \ldots \cdot a_l)^N$ for each $\nu \in [1, l]$ and for every $N \in \mathbb N$.
\end{enumerate}
\end{proposition}

\begin{proof}
We may assume that $H$ is reduced. Consider a divisor theory $H \hookrightarrow \mathcal F (\mathcal P)$ and denote by $G_{\mathcal P} \subset G$ the set of classes containing prime divisors.

First suppose that $G$ is a torsion group and let $u \in \mathcal A (H)$ have Property (a). Then $u = p_1^{k_1} \cdot \ldots \cdot p_m^{k_m}$ for some $m, k_1, \ldots, k_m \in \N$ and pairwise distinct elements $p_1, \ldots, p_m \in \mathcal P$. Then (a) implies that $k=1$ and hence $u$ is absolutely irreducible. Thus Property (b) cannot hold for any $l \ge 2$.

Conversely, suppose that $G$ is not a torsion group. Since $[G_{\mathcal P}] = G$ there exists a $p \in \mathcal P$ such that $[p] \in G$ has infinite order, and there is an element $u' \in \mathcal A (H)$ with $p \t u'$.
Suppose that $u' = p_1 \cdot \ldots \cdot p_n \cdot q_1 \cdot \ldots \cdot q_r$, where $p=p_1, p_2, \ldots, p_n, q_1, \ldots, q_r \in \mathcal P$, $[p_1], \ldots, [p_n]$ have infinite order, and $[q_1], \ldots, [q_r]$ have finite order each of which divides some integer $N$. Then $(q_1 \cdot \ldots \cdot q_r)^N \in H$, whence $(p_1 \cdot \ldots \cdot p_n)^N \in H$. After a possible reordering there is an
 atom $u = p_1^{k_1} \cdot \ldots \cdot p_m^{k_m} \in \mathcal A (H)$ dividing a power of $(p_1 \cdot \ldots \cdot p_n)^N$
such that there is no atom $v \in \mathcal A (H)$ with $\supp_{\mathcal P} (v) \subsetneq \{p_1, \ldots, p_m\}$. Thus $u$ satisfies Property (a). Since $H \hookrightarrow \mathcal F (\mathcal P)$ is a divisor theory, there exist $b_1, \ldots, b_s \in H$ such that $p_1^{k_1} \cdot \ldots \cdot p_{m-1}^{k_{m-1}} = \gcd (b_1, \ldots, b_s)$. Hence there is an $i \in [1,s]$, say $i=1$, such that $p_m \nmid b_1$. Similarly, there are $c_1, \ldots, c_t \in H$ such that $p_m^{k_m} = \gcd (c_1, \ldots, c_t)$. Without loss of generality, there exists  $i \in [1, m-1]$ such that $p_i \nmid c_1$. Therefore $u \mid b_1c_1$, but yet $u \nmid b_1^N$ and $u \nmid c_1^N$ for any $N \in \mathbb N$, and so Property (b) is satisfied.
\end{proof}

Propositions \ref{7.1}, \ref{7.2}, and \ref{7.3} provide abstract finiteness and non-finiteness results. To obtain more precise information on the arithmetical invariants, we require specific information on $G_{\mathcal P}$. In the next subsection we will use such specific information  to give more concrete results.

\ \\

\subsection{Specific  sets $G_{\mathcal P}$ of classes containing prime divisors  and  arithmetical characterizations} \label{6b}

We now provide an in-depth study of the arithmetic of three classes of  Krull monoids studied in Sections \ref{5} and \ref{6}.
Theorem \ref{7.4} describes the arithmetic of the monoids discussed in
 Examples \ref{DivisorTheory2}, \ref{E:isomorphicG_0},   \ref{FRTmonoid_generalization} and in Theorem \ref{DivisorTheory2}.
Its arithmetic is  simple enough that we can  more or less give a complete description.

\medskip
\begin{theorem} \label{7.4}
Let $H$ be a Krull monoid with class group $G$ and suppose that
\[
G_{\mathcal P} = \{e_0,   \ldots, e_r, -e_0, \ldots, -e_r\} \subset G
\]
is the set of classes containing prime divisors, where $r, \alpha \in \N$ with $r+\alpha > 2$ and $(e_1, \ldots, e_r)$ is an independent family of elements each having infinite order such that $e_1+\cdots+e_r = \alpha e_0$. Then:
\begin{enumerate}[(1)]
\item $\mathcal A (G_{\mathcal P}) = \{V, -V, U_{\nu} \, \colon \nu \in [0,r] \}$, where $V = (-e_0)^{\alpha} e_1 \cdot \ldots \cdot e_r$ and $U_{\nu} = (-e_{\nu})e_{\nu}$ for all $\nu \in [0,r]$. In particular, $\mathsf D (G_{\mathcal P}) = r+\alpha$.

\smallskip
\item Suppose that
      \[
      S = \prod_{i=0}^r e_i^{k_i} (-e_i)^{l_i} \in \mathcal F (G_{\mathcal P}) \,,
      \]
      where $k_0,l_0, \ldots, k_r,l_r \in \mathbb N_0$. Then $S \in \mathcal B (G_{\mathcal P})$ if and only if $l_i = \alpha^{-1}(k_0-l_0)+k_i$ for all $i \in [1,r]$. If $S \in \mathcal B (G_{\mathcal P})$ with $k_0 \ge l_0$ and $k^* = \min \{k_1, \ldots, k_r \}$, then
      \[
      \mathsf Z (S) = \Big\{ V^{\nu} (-V)^{\alpha^{-1}(k_0-l_0)+\nu} U_0^{l_0- \alpha \nu} \prod_{i=1}^r U_i^{k_i-\nu} \, \colon \nu \in [0, \min \{\alpha^{-1} l_0, k^*\}] \, \Big\}
      \]
      and
      \[
      \mathsf L (S) = \{\alpha^{-1}( k_0-l_0) + l_0 + k_1 + \cdots + k_r -(r+\alpha-2)\nu \, \colon \nu \in [0, \min \{\alpha^{-1}l_0, k^*\}] \, \} \,.
      \]

\smallskip
\item The system of sets of lengths of $H$ can be described as follows.

\smallskip
\begin{enumerate}[(a)]
\item $\Delta (H) = \{r + \alpha - 2 \}$.
\smallskip
\item $\rho (H) = \mathsf D (G_{\mathcal P})/2$.
\smallskip
\item For each $k \in \mathbb N$, the set $\mathcal U_k (H)$ is an  arithmetical progression with difference $r+\alpha-2$.

\item For each $k \in \N$ and each $j \in [0,1]$, $\rho_{2k+j}(H) = k \mathsf D (G_{\mathcal P}) + j$.

\smallskip
\item For each $l \in \mathbb N_0$, $\lambda_{l \mathsf D (G_{\mathcal P}) +j}(H) =  2l+j$ whenever $j \in [0, \mathsf D (G_{\mathcal P}) -1]$ and $l \mathsf D (G_{\mathcal P}) +j \ge 1$.
\smallskip
\item Finally,
\[
\mathcal L (H) = \Big\{ m + \{2k^* + (r+\alpha-2)\lambda \, \colon \lambda \in [0, k^*] \} \, \colon m, k^* \in \mathbb N_0 \Big\} \,.
\]
\end{enumerate}
\smallskip
\item $\mathsf c (H) = \mathsf c_{\mon} (H)  =  \omega (H) = \mathsf t (H) = \mathsf D (G_{\mathcal P}) = r+\alpha$.
\end{enumerate}
\end{theorem}

\begin{proof}
By Lemma \ref{4.4}, all assertions on lengths of factorizations and on catenary degrees can
be proved working in $\mathcal B (G_{\mathcal P})$ instead of $H$.

\smallskip
Obviously, $\{ U_{\nu} \, \colon \nu \in [0,r] \} \subset \mathcal A (G_{\mathcal P})$ and to prove (1) it remains to verify that if $W \in \mathcal A (G_{\mathcal P})$ with $W\not=U_{\nu}$, then $W=V$. Note that $e_0 \in \langle e_1, \ldots, e_r \rangle$ but that $\langle e_0 \rangle \cap \langle e_i \, \colon i \in I \rangle = \{0\}$ for any proper subset $I \subsetneq [1, r]$. Thus, if
\[
W = \prod_{i \in I} e_i^{k_i} (-e_0)^{k_0} \in \mathcal A (G_{\mathcal P}) \setminus \{ U_{\nu} \, \colon \nu \in [0,r] \} \,,
\]
where $\emptyset \ne I \subset [1,r]$ and $k_0, k_i \in \mathbb N$ for all $i \in I$, then $k_0e_0 = \sum_{i \in I} k_i e_i \in \langle e_i \, \colon i \in I \rangle$ and hence $I = [1, r]$.
Assume to the contrary that there is $i \in [1, r]$ such that $k_i > 1$. Since $V \in \mathcal B (G_{\mathcal P})$ and $W \in \mathcal A (G_{\mathcal P})$, it follows that $k_0 \in [1, \alpha-1]$. Then
\[
0 \ne (k_1-1)e_1 + \cdots + (k_r-1)e_r = (k_0- \alpha)e_0 \in [e_1, \ldots, e_r] \cap [-e_1 , \ldots , -e_r] = \{0\} \,,
\]
a contradiction. Thus $k_1 = \cdots = k_r = 1$ and we obtain that $k_0 = \alpha $, whence $W = V \in \mathcal A (G_{\mathcal P})$.

\smallskip
To prove (2), suppose that $S \in \mathcal B (G_{\mathcal P})$ and  that $l_0 \ge k_0$. Then
\[
S' = (-e_0)^{l_0-k_0} \prod_{i=1}^r e_i^{k_i} (-e_i)^{l_i} \in \mathcal B (G_{\mathcal P}) \,,
\]
whence $l_0 - k_0 = \alpha m_0 \in \alpha \N_0$, $S'' = \prod_{i=1}^r e_i^{k_i-m_0} (-e_i)^{l_i} \in \mathcal B (G_{\mathcal P})$, and
\[
l_i = k_i - m_0 = k_i - \frac{l_0-k_0}{\alpha} = \frac{k_0 - l_0}{\alpha} + k_i \quad \text{for each} \quad i \in [1,r] \,.
\]
The same holds true if $l_0 \le k_0$. Conversely,  if $l_1, \ldots, l_r$ satisfy the asserted equations, then obviously $\sigma(S)=0$.

Suppose that $S \in \mathcal B (G_{\mathcal P})$ and that $k_0 \ge l_0$. Then
\[
\begin{aligned}
S & =  e_0^{k_0} (-e_0)^{l_0} \prod_{i=1}^r e_i^{k_i}(-e_i)^{k_i+ \alpha^{-1}(k_0-l_0)} \\
 & = \big( (-e_0)e_0\big)^{l_0} \big( {e_0}^{\alpha}(-e_1) \cdot \ldots \cdot (-e_r) \big)^{\alpha^{-1}(k_0-l_0)} \prod_{i=1}^r \big( (-e_i)e_i \big)^{k_i} \\
 & = \big( (-e_0)e_0\big)^{l_0- \alpha \nu} \big( {e_0}^{\alpha}(-e_1) \cdot \ldots \cdot (-e_r) \big)^{\alpha^{-1}(k_0-l_0) + \nu} \big( (-e_0)^{\alpha}e_1 \cdot \ldots \cdot e_r \big)^{\nu} \prod_{i=1}^r \big( (-e_i)e_i \big)^{k_i-\nu} \\
  & = U_0^{l_0 - \alpha \nu} (-V)^{\alpha^{-1}(k_0-l_0)+\nu} V^{\nu} \prod_{i=1}^r U_i^{k_i-\nu}
\end{aligned}
\]
for each $\nu \in [0, \min \{\alpha^{-1}l_0, k^*\}]$. Therefore $\mathsf Z (S)$ and hence $\mathsf L (S)$ have the given forms.

\smallskip
We now consider the statements of (3). The assertion on $\Delta (G_{\mathcal P})$ follows immediately from (2). Since $\Delta (G_{\mathcal P}) = \{r+\alpha-2\}$, all sets $\mathcal U_k (G_{\mathcal P})$ are arithmetical progressions with difference $r+\alpha-2$.  The assertion on each $\rho_{2k}(G_{\mathcal P})$ and each $\rho (G_{\mathcal P})$ follow from Proposition \ref{7.2}.

In order to determine
$\mathcal L (G_{\mathcal P})$, let $S \in \mathcal B (G_{\mathcal P})$ be given with all parameters as in (2). First suppose that $l_0 \ge \alpha k^*$. Then
\[
\begin{aligned}
\mathsf L (S) & = (\alpha^{-1}(k_0-l_0)+l_0+k_1+\cdots +k_r)-(r+\alpha)k^* + \{(r+\alpha)k^*-(r+\alpha-2)\nu \, \colon \nu \in [0, k^*]\} \\
 & = (\alpha^{-1}(k_0-l_0)+l_0+k_1+\cdots +k_r)-(r+\alpha)k^* + \{2k^*+(r+\alpha-2)\lambda \, \colon \lambda \in [0, k^*]\} \,.
\end{aligned}
\]
Thus $\mathsf L (S)$ has the form
\[
\mathsf L (S) = m + \{2k^*+(r+\alpha-2)\lambda \, \colon \lambda \in [0, k^*]\}
\]
 for some $m, k^* \in \mathbb N_0$. Conversely, for every choice of $m, k^* \in \mathbb N_0$, there is an $S \in \mathcal B (G_{\mathcal P})$ such that $\mathsf L (S)$ has the given form.

Now suppose that $l_0 \le \alpha k^*-1$ and set $m_0 = \lfloor \frac{l_0}{\alpha} \rfloor$. Then
\[
\begin{aligned}
\mathsf L (S) & = (\alpha^{-1}(k_0-l_0)+l_0+k_1+\cdots +k_r)-(r+\alpha)m_0 + \{(r+\alpha)m_0 -(r+\alpha-2)\nu \, \colon \nu \in [0, m_0]\} \\
 & = (\alpha^{-1}(k_0-l_0)+l_0+k_1+\cdots +k_r)-(r+\alpha)m_0 + \{2m_0 +(r+\alpha-2)\lambda \, \colon \lambda \in [0, m_0]\} \,,
\end{aligned}
\]
and hence $\mathsf L (S)$ has the form
\[
\mathsf L (S) = m + \{2m_0+(r+\alpha-2)\lambda \, \colon \lambda \in [0, m_0]\}
\]
for some $m \in \mathbb N$ and $m_0 \in \mathbb N_0$.

Next we  verify that, for every $k \in \mathbb N$, $\rho_{2k+1}(G_{\mathcal P}) \le k \mathsf D (G_{\mathcal P}) + 1$. By \cite[Proposition 1.4.2]{Ge-HK06a}, for all $k \in \N$,  $\rho_k (G_{\mathcal P}) = \sup \{ \sup L \, \colon L \in \mathcal L (G_{\mathcal P}), k = \min L \}$.
Thus we may choose  $k \in \mathbb N$ and $L \in \mathcal L (G_{\mathcal P})$ with $\min L = 2k+1$. Then, there exist $l, m, k^* \in \mathbb N_0$ such that
\[
L = m + \{2k^*+(r+\alpha-2)\lambda \, \colon \lambda \in [0, k^*]\}
\]
with $m=2l+1$ and $2k+1 = \min L = 2(k^*+l)+1$. Now
\[
\max L = m + (r+\alpha)k^* = 2l+1+(r+\alpha)(k-l) = (r+\alpha)k+1-(r+\alpha-2)l \le k \mathsf D (G_{\mathcal P})+1 \,,
\]
and thus $\rho_{2k+1}(G_{\mathcal P}) \le k \mathsf D (G_{\mathcal P})+1$.

It remains to verify the assertions on the $\lambda_{l \mathsf D (G_{\mathcal P})+j} (G_{\mathcal P})$. Let $l \in \mathbb N_0$ and $j \in [0, \mathsf D (G_{\mathcal P})-1]$.
Then Lemma  \ref{4.3} implies $\lambda_{l \mathsf D (G_p)+j} (G_{\mathcal P}) \le 2 l  + j$, and that equality holds if $j \in [0,1]$. It remains to verify that $\lambda_{l \mathsf D (G_p)+j} (G_{\mathcal P}) \ge 2 l  + j$ when $j \in [2,  \mathsf D (G_{\mathcal P})-1]$. Let  $L \in \mathcal L (G_{\mathcal P})$ with $ l \mathsf D (G_{\mathcal P})+j \in L$. Then there exist $m, k^* \in \mathbb N_0$ such that
\[
L = m + \{2k^*+(r+\alpha-2)\lambda \, \colon \lambda \in [0, k^*]\} = m + k^* \mathsf D (G_{\mathcal P}) - \{(\mathsf D (G_{\mathcal P})-2) \nu \, \colon \nu \in [0, k^*] \} \,.
\]
Suppose   $l \mathsf D (G_{\mathcal P})+j = \max L - \nu(\mathsf D (G_{\mathcal P})-2) = m + k^* \mathsf D (G_{\mathcal P}) - \nu(\mathsf D (G_{\mathcal P})-2)$ for some $\nu \in [0, k^*]$. Then $j \equiv m + 2\nu \mod \mathsf D (G_{\mathcal P})$ and hence $m+2\nu \ge j$. This implies
\[
(k^*-\nu)\mathsf D (G_{\mathcal P}) + j \le m + k^* \mathsf D (G_{\mathcal P}) - \nu ( \mathsf D (G_{\mathcal P}) -2) = l \mathsf D (G_{\mathcal P})+j \,,
\]
and hence $l \ge k^* - \nu$. Therefore we obtain
\[
\min L = l \mathsf D (G_{\mathcal P})+j - (k^*-\nu)(\mathsf D (G_{\mathcal P})-2)  =  (l-k^*+\nu) \mathsf D (G_{\mathcal P}) + j + 2(k^*-\nu) \ge 2l + j \,,
\]
and thus $\lambda_{l \mathsf D (G_p)+j} (G_{\mathcal P}) \ge 2 l  + j$.

\smallskip
Finally we consider the catenary degrees of $H$ and prove the statements given in (4). Using Proposition \ref{7.2} we infer
\[
\mathsf D (G_{\mathcal P}) = r+\alpha = 2+ \max \Delta (G_{\mathcal P}) \le \mathsf c (G_{\mathcal P}) = \mathsf c (H) \le \omega (H) \le \mathsf t (H) \,.
\]
Since $\mathsf c (G_{\mathcal P}) \le \mathsf c_{\mon} (G_{\mathcal P})$, it remains to show that $\mathsf c_{\mon} (G_{\mathcal P}) \le \mathsf D (G_{\mathcal P})$ and that $\mathsf t (H) \le \mathsf D (G_{\mathcal P})$.

We proceed in two steps. First we verify that  $\mathsf c_{\mon} (G_{\mathcal P}) = \max \{\mathsf c_{\eq} (G_{\mathcal P}), \mathsf c_{\adj} (G_{\mathcal P})\} \le r+\alpha $.
Since
\[
\mathcal A (\sim_{\mathcal B(G_{\mathcal P}), \eq}) = \{ (-V, -V), (V,V), (U_{\nu}, U_{\nu}), (-U_{\nu}, -U_{\nu}) \, \colon \nu \in [0, r]\} \,,
\]
it follows that $\mathsf c_{\eq} (G_{\mathcal P}) = 0$. If $A_{r+\alpha-2} = \{ x \in \mathsf Z (G_{\mathcal P}) \, \colon |x|-(r+\alpha - 2) \in \mathsf L ( \pi(x))\}$,
then $\text{\rm Min} (A_{r+\alpha-2}) = \{U_0^{\alpha} U_1 \cdot \ldots \cdot U_r\}$, and hence  $\mathsf c_{\adj}(G_{\mathcal P}) \le r+\alpha$ by Lemma \ref{3.1}.

In order to show that  $\mathsf t (H) \le \mathsf D (G_{\mathcal P})$,  we must verify the following assertion (see \cite[Theorem 3.6]{Ge-Ha08a}).

\begin{enumerate}[{\bf (A)}]
\item Let $j \in \mathbb N$ and $w, w_1, \ldots, w_j \in \mathcal A (H)$ such that $w$ divides the product $w_1 \cdot \ldots \cdot w_j$ but yet $w$ divides no proper subproduct of $w_1 \cdot \ldots \cdot w_j$. Then $\min \mathsf L (w^{-1} w_1 \cdot \ldots \cdot w_j) \le \mathsf D (G_{\mathcal P})-1$.
\end{enumerate}

\noindent {\it Proof of \,{\bf (A)}}.\, We use the transfer homomorphism $\boldsymbol \beta \colon H \to \mathcal B (G_{\mathcal P})$ as defined in Lemma \ref{4.4}. Set $W = \boldsymbol \beta (w)$ and $W_i = \boldsymbol \beta (w_i)$ for each $i \in [1,j]$. Then $j \le |W|$ and $W, W_1, \ldots, W_j \in \mathcal A (G_{\mathcal P})$. Clearly
\[
\min \mathsf L (w^{-1} w_1 \cdot \ldots \cdot w_j) \le \max \mathsf L (W^{-1} W_1 \cdot \ldots \cdot W_j) \le \frac{|W_1 \cdot \ldots \cdot W_j| - |W|}{2} \,.
\]
Thus, if $|W| = 2$, then
\[
\min \mathsf L (w^{-1} w_1 \cdot \ldots \cdot w_j) \le \frac{|W_1|+|W_2|-|W|}{2} \le \mathsf D (G_{\mathcal P})-1 \,.
\]
It remains to consider the case $W \in \{-V, V\}$, and by symmetry we may suppose that $W = V$.
If $|W_1|=\cdots=|W_j|=2$, then $j = |V|$ and $w^{-1}w_1 \cdot \ldots \cdot w_j \in \mathcal A (H)$. Suppose there is $\nu \in [1,j]$, say $\nu=1$, such that $\boldsymbol \beta (w_1) \in \{-V, V\}$. Since $w$ does not divide a subproduct of $w_1 \cdot \ldots \cdot w_j$ and $\gcd (V, -V) = 1$, it follows that $\boldsymbol \beta (w_1) = V$. Then $\mathsf L (w^{-1} w_1 \cdot \ldots \cdot w_j) = \mathsf L (W^{-1} W_1 \cdot \ldots \cdot W_j) = \mathsf L (W_2 \cdot \ldots \cdot W_j)$ and hence
\[
\min \mathsf L (w^{-1} w_1 \cdot \ldots \cdot w_j) \le j-1 \le |V|-1= \mathsf D (G_{\mathcal P})-1 \,. \qedhere
\]
\end{proof}

\smallskip
The next corollary again reveals  that certain arithmetical phenomena characterize certain algebraic properties of   the class group.

\medskip
\begin{corollary} \label{7.5}
 Let $H$ be a Krull monoid  as in Theorem \ref{7.4} with class group $G$ and set $G_{\mathcal P}$ of classes containing prime divisors. Then $r+1$ is the minimum of all $s \in \mathbb N$ having the following property{\rm \,:}
\begin{itemize}
\item[{\bf (P)}] There are  absolutely irreducible elements $w_1, \ldots, w_s \in \mathcal A (H)$ such that  $2, \mathsf D (G_{\mathcal P}) \in \mathsf L (w_1^{k_1} \cdot \ldots \cdot w_s^{k_s})$ for some $(k_1, \ldots, k_s) \in \mathbb N_0^s$.
\end{itemize}
\end{corollary}

\begin{proof}
First we verify that $r+1$ satisfies property {\bf (P)}. For $i \in [0,s]$, let $p_i \in \mathcal P \cap e_i$ and $q_i \in \mathcal P \cap (-e_i)$ and set $w_i = p_iq_i$. Then $w_0, \ldots, w_s$ are  absolutely irreducible elements and, by Theorem \ref{7.4}, it follows that $2, \mathsf D (G_{\mathcal P}) \in \mathsf L (w_0^{\alpha} w_1 \cdot \ldots \cdot w_r)$.

Conversely, let $s \in \mathbb N$, and let $w_1, \ldots, w_s$ and $k_1, \ldots, k_s$ be as above. For $i \in [1,s]$, we set $W_i = \boldsymbol \beta (w_i)$. Since $\rho (G_{\mathcal P}) = \mathsf D (G_{\mathcal P})/2$ and $2, \mathsf D (G_{\mathcal P}) \in \mathsf L (W_1^{k_1} \cdot \ldots \cdot W_s^{k_s})$, it follows that $\sum_{i=1}^s k_i|W_i| = \mathsf D (G_{\mathcal P})$, $|W_1| = \cdots = |W_s| = 2$, $W_i = (-g_i)g_i$ for $i \in [1,s]$, and that $S = g_1^{k_1} \cdot \ldots \cdot g_s^{k_s} \in \mathcal A (G_{\mathcal P})$. Now Theorem \ref{7.4} implies  $S = (-e_0)^{\alpha}e_1 \cdot \ldots \cdot e_r$, whence $\{W_1, \ldots, W_s\} = \{(-e_0)e_0, \ldots, (-e_r)e_r\}$. Thus $|\{w_1, \ldots, w_s\}| \ge |\{W_1, \ldots, W_s\}| = r+1$ and so $r+1$ is minimal with property {\bf (P)}.
\end{proof}

\medskip
We now begin collecting information in order to study  the arithmetic of the Krull monoid presented in Example \ref{cube-and-its-negative}.
In spite of the simple geometric structure of $G_{\mathcal P}$ (the set consists of the  vertices of the unit cube and their negatives), the arithmetic of this Krull monoid is highly complex. We get only very limited information. Nevertheless, this will  be sufficient to give an arithmetical characterization.

\medskip
\begin{lemma} \label{7.6}
Let $G$ be an abelian group and let $(e_n)_{n \ge 1}$ be a family of independent elements each having infinite order. For $r \in \mathbb N$, set
\[
G_r^+ = \{ a_1 e_1 + \cdots + a_r e_r \, \colon a_1, \ldots, a_r \in [0,1]\}, \ G_r^- = - G_r^+, \quad \text{and} \quad G_r = G_r^+ \cup G_r^- \,.
\]
\begin{enumerate}[(1)]
\item Let $s \in  [2, r]$, $f_0 = e_1+ \cdots + e_s$, and $f_i = f_0-e_i$ for all $i \in [1,s]$. Then $(f_1, \ldots, f_s)$ is independent, $f_1+ \cdots + f_s = (s-1)f_0$, and $\Delta (\{f_0, \ldots, f_s, -f_0, \ldots, -f_s\}) = \{2s-3\}$.

\smallskip
\item Let $s \in  [3, r]$, $f_0 = e_1+ \cdots + e_s$, $f_i = f_0-e_i$ for each $i \in [1,s-1]$, and set $f_s' = -e_s$. Then $(f_1, \ldots, f_{s-1}, f_s')$ is independent, $f_1+ \cdots + f_{s-1}+f_s' = (s-2)f_0$, and \\ $\Delta (\{f_0, \ldots, f_{s-1}, f_s', -f_0, \ldots, -f_{s-1}, -f_s'\}) = \{2s-4\}$.

\smallskip
\item If $s \le [1, r-1]$, then $\mathsf D (G_r) \ge \mathsf D (G_s) + \mathsf D (G_{r-s}) - 1$.  In particular,   $\mathsf D (G_1) = 2$ and $\mathsf D (G_r) > \mathsf D (G_{r-1})$ for $r \ge 2$.
\end{enumerate}
\end{lemma}

\begin{proof}
Since $(e_1, \ldots, e_s)$ is a basis, there is a matrix $A_s$ with $(f_1, \ldots, f_s) = (e_1, \ldots, e_s) A_s$. Since $\det (A_s) \ne 0$, it follows that $(f_1, \ldots, f_s)$ is independent. By definition, we have $f_1+ \cdots + f_s = (s-1)f_0$. The assertion on the set of distances then follows from Theorem \ref{7.4} and we have proved (1).

\smallskip
We now consider (2). Note that $f_s' = f_s-f_0$. Using (1) we infer that
\[
0 = (f_1 + \cdots + f_s) - (s-1)f_0 = (f_1+ \cdots + f_{s-1}) + (f_s-f_0) - (s-2)f_0
\]
and hence $f_1+ \cdots + f_{s-1}+f_s' = (s-2)f_0$. Since $(f_1, \ldots, f_{s-1}, - e_s) = (e_1, \ldots, e_s) B_s$ for some matrix $B_s$ with $\det (B_s) = (-1)^{2s} \det (A_{s-1}) \ne 0$, it follows that $(f_1, \ldots, f_{s-1}, f_s')$ is independent.
The assertion on the set of distances follows from Theorem \ref{7.4}.

\smallskip
It is clear that $\mathsf D(G_1)=2$ and that $\mathsf D(G_r)>\mathsf D(G_{r-1})$ whenever $r\geq 2$. To prove the remaining statements of (3), suppose that $s \in [1, r-1]$. After a change of notation, we may suppose that $G_{r-s} \subset \langle e_{s+1}, \ldots, e_r \rangle$ such that $\langle G_r \rangle = \langle G_s \rangle \oplus \langle G_{r-s} \rangle$. If $U = a_1 \cdot \ldots \cdot a_k \in \mathcal A (G_s)$ with $k = \mathsf D (G_s)$ and $V = b_1 \cdot \ldots \cdot b_l \in \mathcal A (G_{r-s})$ with $l = \mathsf D (G_{r-s})$, then
$W = (a_1+b_1) \cdot a_2 \cdot \ldots \cdot a_k  b_2 \cdot \ldots \cdot b_l \in \mathcal A (G_r)$, and hence $\mathsf D (G_r) \ge |W| = k+l-1 = \mathsf D (G_s) + \mathsf D (G_{r-s})-1$.
\end{proof}

In Theorem \ref{7.8} we restrict to class groups of rank $r \ge 3$ because when $r \le 2$ we are in the setting of Theorem \ref{7.4} where we have precise information about arithmetical invariants. For $r \in \mathbb N_0$, we denote by $\mathsf F_r$ the $r$th {\it Fibonacci number}. That is, $\mathsf F_0 = 0$, $\mathsf F_1 = 1$, and $\mathsf F_r = \mathsf F_{r-1}+\mathsf F_{r-2}$ for all $r \ge 2$.

\medskip
\begin{theorem} \label{7.8}
Let $H$ be a Krull monoid with free abelian class group $G$ of rank $r \ge 3$ and let $G_{\mathcal P} \subset G$ denote the set of classes containing prime divisors. Suppose that there is a basis $(e_1, \ldots, e_r)$ of $G$ such that $G_{\mathcal P}^{\bullet} = G_{\mathcal P}^+ \cup G_{\mathcal P}^-$, where
$G_{\mathcal P}^+ = \{ \epsilon_1 e_1 + \cdots + \epsilon_r e_r \, \colon \epsilon_1, \ldots, \epsilon_r \in [0,1]\}$ and  $G_{\mathcal P}^- = - G_{\mathcal P}^+$.

\begin{enumerate}[(1)]
\item $\mathsf F_{r+2} \le \mathsf D (G_{\mathcal P})$.

\smallskip
\item $\mathsf c (H) \le \omega (H) = \mathsf D (G_{\mathcal P})$, $\rho (H) = \mathsf D (G_{\mathcal P})/2$, and $\rho_{2k} (H) = k \mathsf D (G_{\mathcal P})$ for each $k \in \mathbb N$.

\smallskip
\item $[1, 2r-3]  \subset \Delta^* (H) \subset \Delta (H) \subset [1, \mathsf c (H)-2]$.
\end{enumerate}
\end{theorem}

\begin{proof}
See \cite{Ba-Ge-Gr-Sm14} for the proof of (1).
(2) follows from Proposition \ref{7.2}.

Note that for every $s \in [2,r]$ we have $2s-3 \in \Delta^* (H)$ and,  by Lemma \ref{7.6}, for all $s \in [3,r]$ we have $2s-4 \in \Delta^* (H)$.
This implies that the interval $[1, 2r-3]$ is contained in $\Delta^* (H)$, giving (3).
\end{proof}

\medskip
The third class of Krull monoids studied in this subsection are Krull monoids with finite cyclic class group having prime divisors in each class. Thus Theorem \ref{7.9} describes the arithmetic of the monoids constructed   in Theorem \ref{2dimensional-classgroups}. Holomorphy rings in  global fields are Krull monoids with finite class group and prime divisors in all classes. For this reason this class of Krull monoids has received a great deal of attention.

\medskip
\begin{theorem} \label{7.9}
Let $H$ be a Krull monoid with finite cyclic class group $G$ of order $|G|=n \ge 3$, and suppose that every class contains a prime divisor. Then:
\begin{enumerate}[(1)]
\item $\mathsf c (H) = \omega (H) = \mathsf D (G) =  n$ and $\Delta (H) = [1, n-2]$.

\smallskip
\item For every $k \in \mathbb N$ the set $\mathcal U_k (H)$ is a finite interval, whence $\mathcal U_k (H) = [\lambda_k (H), \rho_k (H)]$. Moreover,  for all $l \in \mathbb N_0$ with $ln+j \ge 1$,
    \[
\rho_{2k+j}(H) = kn + j \quad \text{for} \quad j \in [0,1] \quad \text{and} \quad
\lambda_{ln+j}(H) = \begin{cases} 2l+j \quad  & \text{for} \quad j \in [0,1] \\
                  2l+2 \quad & \text{for} \quad j \in [2, n-1].
                  \end{cases}
\]

\item
\[
\max \Delta^* (H) = n-2 \quad \text{and} \quad \max \bigl(
\Delta^* (H) \setminus \{n-2\} \bigr) = \left\lfloor \frac{n}{2}
\right\rfloor - 1 .
\]
\end{enumerate}
\end{theorem}

\begin{proof}
The proof of (1) can be found in \cite[Theorem 6.7.1]{Ge-HK06a} and the proof of (3) can be found in \cite[Theorem 6.8.12]{Ge-HK06a}.
For (2) see \cite[Corollary 5.3.2]{Ge09a}.
\end{proof}

\medskip
Much recent research is devoted to the arithmetic of Krull monoids discussed in Theorem \ref{7.9}. We briefly address some open questions. Let $H$ be  as above and suppose that $n \ge 5$. The precise values of $\mathsf t (H)$ and of $\mathsf c_{\mon} (H)$ are unknown. It is easy to check that $\mathsf D (G) = n < \mathsf t (H)$ (in contrast to what we have in Theorem \ref{7.4}). For recent results on lower and upper bounds of the tame degree, see \cite{Ga-Ge-Sc14a}. We remark that there is a standing conjecture that the monotone catenary degree is that $n = \mathsf c (H) = \mathsf c_{\mon} (H)$ (this coincides what we have in Theorem  \ref{7.4}; see \cite{Ge-Yu13a}). For recent progress on $\Delta^* (H)$ we refer to \cite{Pl-Sc13a}.

\medskip
Having at least a partial description of the arithmetic of the three monoids described in Theorems \ref{7.4}, \ref{7.8}, and \ref{7.9}, we now work to show that except for in a small number of exceptions, these monoids have vastly different arithmetic. After some preliminary work this distinction is made clear in Corollary \ref{7.10}.

\medskip
\begin{lemma} \label{7.7}
Let $G$ be an abelian group with finite total rank and let $G_0 \subset G$ be a subset with $G_0 = -G_0$. Suppose that $\mathcal L (G_0) = \mathcal L (C_n)$ for some $n \ge 5$. Then there exists an absolutely irreducible element $U \in \mathcal A (G_0)$ with $|U| = \mathsf D (G_0)$.
\end{lemma}

\begin{proof}
First observe that $\mathsf D (G_0) = \rho_2 (G_0) = \rho_2 (C_n) = \mathsf D(C_n) = n$ and, by \cite[Theorem 3.4.2]{Ge-HK06a},  $\mathcal A (G_0)$ is finite, say $\mathcal A (G_0) = \{ U_1, -U_1, \ldots, U_q, - U_q\}$. If $g \in C_n$ with $\ord (g) = n$, then for all $k \in \N$ we have

\[
L_k = \{ 2k + \nu (n-2) \, \colon \, \nu \in [0,k] \} = \mathsf L \big( g^{nk} (-g)^{nk} \big)   \in \mathcal L (C_n) = \mathcal L (G_0) \,.
\]
Since $\rho (L_k) = \rho (G_0) = \mathsf D (G_0)/2$, there exists, for every $k \in \N$, a tuple $(k_1, \ldots, k_q) \in \N_0^{(q)}$ such that $k_1+ \cdots + k_q = k$ and
\[
L_k = \mathsf L \big( (-U_1)^{k_1}U_1^{k_1} \cdot \ldots \cdot (-U_q)^{k_q} U_q^{k_q} \big) \,.
\]
Therefore there exists  $\lambda \in [1,q]$ such that $\mathsf L \big( (-U_{\lambda})^k U_{\lambda}^k \big) = L_k$ for every $k \in \N$. Set $U = U_{\lambda}$ and note that
 for every $V \in \mathcal A (G_0)$ with $V \t (-U)^k U^k$ for some $k \in \N$, it follows that  $|V| \in \{2, n\}$. After changing notation if necessary, we may suppose that there is no $V \in \mathcal A (G_0)$ such that $|V| = n$, $\supp (V) \subsetneq \supp (U)$, and $V \t U^k (-U)^k$ for some $k \in \mathbb N$.

In order to show that $U$ is absolutely irreducible, it remains to verify that the torsion-free rank of $\langle \supp (U) \rangle$ is $|\supp (U)|-1$.
Assume to the contrary that there exist $t \in [2, |\supp (U)|-1]$ and $g_1, \ldots, g_t \in \supp (U)$ which are linearly dependent. Then there are $s \in [1,t]$, $m_1, \ldots, m_s \in \mathbb N$, and $m_{s+1}, \ldots, m_t \in - \mathbb N$ such that
\[
m_1 g_1 + \cdots + m_s g_s + (-m_{s+1})(-g_{s+1}) + \cdots + (-m_t)(-g_t) = 0 \,.
\]
Then
\[
V = g_1^{m_1} \cdot \ldots \cdot g_s^{m_s} (-g_{s+1})^{-m_{s+1}} \cdot \ldots \cdot (-g_t)^{-m_t} \in \mathcal B (G_0) \,.
\]
Without restriction we may suppose that the above equation is minimal and that $V \in \mathcal A (G_0)$. Since $V \t U^k (-U)^k$ for some $k \in \mathbb N$ and $|V| > 2$, we obtain a contradiction to the minimality of $\supp (U)$.
\end{proof}

\medskip
The following corollary highlights that the observed arithmetical phenomena in our case studies --- Theorems \ref{7.4}, \ref{7.8}, and \ref{7.9} --- are characteristic for the respective Krull monoids. In particular, this illustrates that the structure of direct-sum decompositions over the one-dimensional Noetherian local rings with finite representation type studied in Section \ref{5} can be quite different from the structure of direct-sum decompositions over the two-dimensional Noetherian local Krull domains with finite representation type studied in Section \ref{6}. As characterizing tools we use the system of sets of lengths along with the behavior of absolutely irreducible elements.

\medskip
\begin{corollary} \label{7.10}
For $i \in [1,3]$, let $H_i$ and $H_i'$ be  Krull monoids with class groups $G_i$ and $G_i'$. Further  suppose that
\begin{itemize}
\item $G_1$ and $G_1'$ are finitely generated and torsion-free of rank $r_1$ and $r_1'$ with sets of classes containing prime divisors as in Theorem \ref{7.4}
      $($with parameters $\alpha, \alpha' \in \N$ such that $\alpha + r_1 \ge \alpha'+r_1' > 2)$.
\item $G_2$ and $G_2'$ are finitely generated and torsion-free of rank $r_2 \ge r_2' \ge 3$ with sets of classes containing prime divisors as in Theorem \ref{7.8}.
\item $G_3$ and $G_3'$ are finite cyclic of order $|G_3|  \ge |G_3'| \ge 5$ such that every class contains a prime divisor.
\end{itemize}
Then:
\begin{enumerate}[(1)]
\item $\mathcal L (H_1) = \mathcal L (H_1')$ if and only if $r_1+\alpha=r_1'+\alpha'$. If this holds, then the arithmetic behavior of the absolutely irreducible elements of $H_1$ and $H_1'$ coincide in the sense of Corollary \ref{7.5} if and only if $r_1=r_1'$.

\item $\mathcal L (H_2) = \mathcal L (H_2')$ if and only if $r_2=r_2'$.

\item $\mathcal L (H_3) = \mathcal L (H_3')$ if and only if $|G_3|=|G_3'|$.

\item $\mathcal L (H_1) \ne \mathcal L (H_2)$ and $\mathcal L (H_1) \ne \mathcal L (H_3)$.

\item For $i \in [2,3]$, let $s_i$ denote the maximal number of absolutely irreducible elements  $u_1, \ldots, u_{s_i} \in H_i  $ such that $2 \in \mathsf L (u_1 \cdot \ldots \cdot u_{s_i})$. Then either $\mathcal L (H_2) \ne \mathcal L (H_3)$ or $s_2 \ne s_3$.
\end{enumerate}
\end{corollary}

\begin{proof}
The if and only if statement in (1) follows immediately from Theorem \ref{7.4}. Suppose that $\mathcal L (H_1) = \mathcal L (H_1')$. Then the assertion in (1) on the arithmetic behavior of absolutely irreducible elements follows from Corollary \ref{7.5}.

\smallskip
To prove (2), first note that one implication is clear, both for $H_2$ and $H_3$. Suppose that $\mathcal L (H_2) = \mathcal L (H_2')$, and let $G_{\mathcal P} \subset G_2$ and $G_{\mathcal P}' \subset G_2'$ denote the set of classes containing prime divisors. Theorem \ref{7.8} implies that
\[
\mathsf D (G_{\mathcal P}) = \rho_2 (H) = \rho_2 (H') = \mathsf D (G_{\mathcal P}') \,,
\]
and thus Lemma \ref{7.6} implies $r_2=r_2'$. Now consider (3). If $\mathcal L (H_3) = \mathcal L (H_3')$, then Theorem \ref{7.9} implies that
\[
|G_3|-2 = \max \Delta (H_3) = \max \Delta (H_3') = |G_3'| - 2 \,.
\]

\smallskip
For (4), note that $\mathcal L (H_1)$ is distinct from both $\mathcal L (H_2)$ and $\mathcal L (H_3)$ since $|\Delta (H_1)| = 1$, $|\Delta (H_2)| > 1$, and $|\Delta (H_3)| > 1$.

\smallskip
For (5) we assume that $\mathcal L (H_2) = \mathcal L (H_3)$ and let $G_{\mathcal P} \subset G_2$ denote the set of classes containing prime divisors. Theorems \ref{7.8} and \ref{7.9} imply that
\[
\mathsf D (G_{\mathcal P}) = \rho_2 (H_2) = \rho_2 (H_3) = |G_3| \,.
\]
By  Proposition \ref{7.3} we obtain that $\mathsf D (G_{\mathcal P}) = s_2$. Now assume to the contrary that $s_2 = s_3$.  If $|G_3|=n$, then there are absolutely irreducible elements $u_1, \ldots, u_n$ and atoms $v_1, v_2 \in \mathcal A (H_3)$ such that $v_1v_2 = u_1 \cdot \ldots \cdot u_n$. Without restriction, we  suppose  $H_3$ is reduced and we consider a divisor theory $H \hookrightarrow \mathcal F ({\mathcal P})$. Since a minimal zero-sum sequence of length $n$ over $G_3$ consists of one element of order $n$ repeated $n$ times, the factorization of the atoms $v_1, v_2, u_1, \ldots, u_n $ in $\mathcal F ({\mathcal P})$ must have the following form: $v_1 = p_1 \cdot \ldots \cdot p_n$, $v_2 = q_1 \cdot \ldots \cdot q_n$, and $u_i = p_iq_i$ for all $i \in [1, n]$, where $p_1, \ldots, p_n, q_1, \ldots, q_n \in {\mathcal P}$, $[p_1] = \cdots = [p_n] \in G_3$, and $[q_1] = \cdots = [q_n] = [-p_1]$. But \cite[Proposition 7.1.5]{Ge-HK06a} implies that the elements $u_1, \ldots, u_n$ are not absolutely irreducible, a contradiction.
\end{proof}

\medskip
\begin{remark} \label{7.11}
Let $H_2$ and $H_3$ be as in Corollary \ref{7.10}. We set $n = |G_3|$, $r = r_2$, and let $G_{\mathcal P, r} \subset G_2$ denote the set of classes containing prime divisors. Assume that $\mathcal L (H_2) = \mathcal L (H_3)$. Then
\[
\mathsf F_{r+2} \le \mathsf D (G_{\mathcal P, r}) = \rho_2 (H_2) = \rho_2 (H_3) = n \,.
\]
That is, the orders of the cyclic groups for which $\mathcal L (H_2) = \mathcal L (H_3)$ grow faster than the sequence of Fibonacci numbers. We conjecture that $\mathcal L (H_2)$ and $\mathcal L (H_3)$ are always distinct but have not further investigated this (rather delicate combinatorial) problem which would require a more detailed investigation of $\mathsf D (G_{\mathcal P, r})$.

Now suppose that $H$ is a Krull monoid with class group $G$ such that every class contains a prime divisor. If $\mathcal L (H) = \mathcal L (H_3)$, then following Theorem \ref{7.9}, one can show that $G$ is isomorphic to the finite cyclic group $G_3$ (see \cite[Corollary 5.3.3]{Ge09a}). Therefore sets of lengths characterize Krull monoids with finite cyclic class group having the property that every class contains a prime divisor.
\end{remark}

\quad

\subsection{Small sets $G_{\mathcal P}$ of classes containing prime divisors and   limits of arithmetical characterizations} \label{6c}

In this final subsection we study the arithmetic of Krull monoids having small sets of classes containing prime divisors. This study pertains to the monoids of Theorem \ref{DivisorTheory2}, Example
\ref{4.19},  Example \ref{E:isomorphicG_0}, and Theorem \ref{2dimensional-classgroups}. The most striking phenomenon here is that these systems of sets of lengths are additively closed (see Proposition \ref{7.14}). As a consequence, if $\mathcal L (H)$ is such a system and $H'$ is a monoid with $\mathcal L (H') \subset \mathcal L (H)$, then $\mathcal L (H \times H') = \mathcal L (H)$ (see  Example \ref{splitexample-1}, Example \ref{splitexample-2}, and Corollary \ref{7.15}).
These phenomena are in strong contrast to the results in the previous subsection, and they show up natural limits for obtaining arithmetical characterization results. Recall that, for $l \in \N_0$ and $d \in \N$, $P_l (d) = \{0, d, \ldots, ld\}$.

\medskip
\begin{proposition} \label{7.12}
Let $H$ be a Krull monoid with infinite cyclic class group $G$ and suppose that
\[
G_{\mathcal P} = \{-2e, -e, 0, e, 2e\} \subset G = \langle e \rangle
\]
is the set of classes containing prime divisors.
Then there is a transfer homomorphism $\theta \colon H \to \mathcal B (C_3)$, and hence
      \[
      \mathcal L (H) = \mathcal L (C_3)  = \mathcal L (C_2 \oplus C_2) = \bigl\{ y
      + 2k + P_k (1) \, \colon \, y,\, k \in \N_0 \bigr\} \,.
      \]
Moreover, $\mathcal L (H)$ coincides with the system of sets of lengths of the Krull monoid studied in Theorem \ref{7.4} with parameters $r=2$ and $\alpha = 1$.
\end{proposition}

\begin{proof}
By Lemma \ref{4.4} there is a transfer homomorphism $\boldsymbol \beta \colon H \to \mathcal B (G_{\mathcal P})$.
Since the composition of two transfer homomorphisms is a transfer homomorphism, it is sufficient to show that there is a transfer homomorphism $\theta' \colon \mathcal B (G_{\mathcal P}) \to \mathcal B (C_3)$. Write $C_3 = \{0, g, -g\}$. Since $\mathcal B (G_{\mathcal P}) = \mathcal F ( \{0\}) \times \mathcal B (G_{\mathcal P}^{\bullet})$ and $\mathcal B (C_3) = \mathcal F (\{0\}) \times \mathcal B ( \{-g, g\})$, it suffices to show that there is a transfer homomorphism $\theta \colon \mathcal B (G_{\mathcal P}^{\bullet}) \to \mathcal B (\{-g, g\})$. In this case,  $\mathcal L (H) = \mathcal L (G_{\mathcal P}) = \mathcal L (C_3)$. Moreover,  $\mathcal L (C_3) = \mathcal L (C_2 \oplus C_2)$ has the form given in \cite[Theorem 7.3.2]{Ge-HK06a} and this coincides with the system of sets of lengths in Theorem \ref{7.4}, provided $(r, \alpha) = (2,1)$.

Note that $\mathcal A (G_{\mathcal P}^{\bullet}) = \{V, -V, U_1, U_2\}$, where $V = e^2 (-2e)$, $U_1 = (-e)e$, and $U_2 = (-2e)(2e)$, and
$\mathcal A (\{-g,g\}) = \{ \overline V, - \overline V, \overline U \}$, where $\overline V = g^3$ and $\overline U = (-g)g$.
Then there is a monoid epimorphism $\widetilde \theta \colon \mathcal F (G_{\mathcal P}^{\bullet}) \to \mathcal F ( \{-g, g\})$ satisfying $\widetilde \theta (e) = \widetilde \theta (-2e) = g$ and $\widetilde \theta (-e) = \widetilde \theta (2e) = -g$. If
\[
A = e^{k_1}(-e)^{k_1'}(2e)^{k_2}(-2e)^{k_2'} \in \mathcal F (G_{\mathcal P}^{\bullet}) \quad \text{with} \quad k_1, k_1', k_2, k_2' \in \mathbb N_0 \,,
\]
then $A \in \mathcal B (G_{\mathcal P}^{\bullet})$ if and only if $k_1-k_1'+2(k_2-k_2')=0$. If this holds, then $k_1+k_2'-(k_1'+k_2) \equiv 0 \mod 3$ and hence
\[
\widetilde \theta (A) = g^{k_1+k_2'}(-g)^{k_1'+k_2} \in \mathcal B ( \{-g, g\}) \,.
\]
Thus $\theta = \widetilde \theta \mid_{\mathcal B (G_{\mathcal P}^{\bullet})}  \colon \mathcal B (G_{\mathcal P}^{\bullet}) \to \mathcal B ( \{-g, g\})$ is a monoid epimorphism satisfying  $\theta (V) = \overline V$, $\theta (-V) = - \overline V$,  $\theta (U_1) = \theta (U_2) = \overline U$, and $\theta^{-1} (1) = \{1\}  = \mathcal B (G_{\mathcal P}^{\bullet})^{\times}$.

Thus in order to show that $\theta$ is a transfer homomorphism, it remains to verify Property {\bf (T2)}. Let $A \in \mathcal B (G_{\mathcal P}^{\bullet})$ be as above and suppose that
\[
\theta (A) = \widetilde B \widetilde C
\]
with $\widetilde B, \widetilde C \in \mathcal B ( \{-g, g\})$ and $\widetilde B = g^m (-g)^{m'}$ such that $m \in [0, k_1+k_2']$, $m' \in [0, k_1'+k_2]$ and $m \equiv m' \mod 3$. Our goal is to find $B, C \in \mathcal B (G_{\mathcal P}^{\bullet})$ such that $A = B C$, $\theta (B) = \widetilde B$, and $\theta (C) = \widetilde C$. Clearly it is sufficient to find $B \in \mathcal B (G_{\mathcal P}^{\bullet})$ with $B \t A$ and $\theta (B) = \widetilde B$, that is, to find parameters
\[
m_1 \in [0, k_1], \ m_1' \in [0, k_1'], \ m_2 \in [0, k_2], \ \ \text{and} \ \ m_2' \in [0, k_2']
\]
such that
\[
m_1+m_2' = m, \ m_1' + m_2 = m' , \quad \text{and} \quad m_1-m_1' + 2(m_2-m_2') = 0 \,. \bf{\tag{$\bf C1$}}
\]
To do so we proceed by induction on $|\widetilde B|$. If $|\widetilde B| = |A|$, then $k_1+k_1'+k_2+k_2' = |A| = |\widetilde B| = m+m'$  and hence $m = k_1+k_1'$ and $m' = k_1'+k_2$. Thus we set
\[
m_1 = k_1, \ m_1' = k_1', \ m_2 = k_2 , \ \text{and} \ m_2' = k_2' \,,
\]
and the assertion is satisfied with $B=A$. Suppose now that the quadruple $(m_1, m_1', m_2, m_2')$ satisfies {\bf (C1)} with respect to the pair $(m, m')$. Dividing $\widetilde B$ by an atom of $\mathcal B ( \{-g, g\})$ (if possible) shows that we must verify that there are solutions to {\bf (C1)} with respect to each of the pairs
\[
(m-1, m'-1),  (m-3,m'), \ \text{and} \ (m, m'-3) \quad \text{ in} \quad \mathbb N_0^{(2)} \,.
\]
One checks respectively that at least one of  the following quadruples satisfy {\bf (C1)}.
\begin{itemize}
\item $(m_1-1, m_1'-1, m_2, m_2')$ or $(m_1, m_1', m_2-1, m_2'-1)$

\item $(m_1-2, m_1', m_2, m_2'-1)$ or $(m_1-3, m_1'-1, m_2+1, m_2')$

\item $(m_1, m_1'-2, m_2-1, m_2')$ or $(m_1-1, m_1'-3, m_2, m_2'+1)$
\end{itemize}
Now the assertion follows by the induction hypothesis.
\end{proof}

\medskip
\begin{proposition} \label{7.13}
Let $H$ be a Krull monoid with free abelian class group $G$ of rank $2$. Let $(e_1, e_2)$ be a basis of $G$  and suppose that
\[
G_{\mathcal P} = \{0, e_1,  e_2, 2e_2, e_1+2e_2, - e_1,  -e_2,  -2e_2,  -e_1-2e_2 \}
\]
is the set of classes containing prime divisors.
Then there is a transfer homomorphism $\theta \colon H \to \mathcal B (C_4)$ and hence
\[
\mathcal L (H) = \mathcal L (C_4)   =  \bigl\{ y + k+1 + P_k (1) \, \colon\, y,
      \,k \in \N_0 \bigr\} \,\cup\,  \bigl\{ y + 2k + P_k (2) \, \colon
      \, y,\, k \in \N_0 \bigr\}  \supset \mathcal L (C_3) \,.
\]
\end{proposition}

\begin{proof}
As in Proposition \ref{7.12}, it suffices to show that there is a transfer homomorphism  $\theta \colon \mathcal B (G_{\mathcal P}^{\bullet}) \to \mathcal B (C_4^{\bullet})$. Then $\mathcal L (H) = \mathcal L (G_{\mathcal P}) = \mathcal L (C_4)$ and  $\mathcal L (C_4)$ has the form given in \cite[Theorem 7.3.2]{Ge-HK06a}. Proposition \ref{7.12} shows that $\mathcal L (C_3) \subset \mathcal L (C_4)$.

We note that
$\mathcal A (G_{\mathcal P}^{\bullet}) = \{W, -W, V_1, -V_1, V_2, -V_2, U_1, U_2, U_3, U_4  \}$, where
\[
\begin{aligned}
W & = e_1 e_2 e_2 (-e_1-2e_2), \ V_1 = e_1 (2e_2) (-e_1-2e_2), \ V_2 = e_2 e_2 (-2e_2), \\
U_1 & = (-e_1)e_1, \ U_2 = (-e_2)e_2), \ U_3 = (-e_1-2e_2)(e_1+2e_2), \ \text{and} \ U_4 = (-2e_2)(2e_2) \,.
\end{aligned}
\]
We set $C_4 = \{0, g, 2g, -g\}$ and observe that $\mathcal A (C_4^{\bullet}) = \{\overline W, - \overline W, \overline V, - \overline V, \overline{U_1}, \overline{U_2} \}$, where
\[
\overline W = g^4, \ \overline V = g^2 (2g), \ \overline{U_1} = (-g)g, \quad  \text{and} \quad \overline{U_2} = (2g)(2g) \,.
\]
There is a monoid epimorphism $\widetilde \theta \colon \mathcal F (G_{\mathcal P}^{\bullet}) \to \mathcal F (C_4^{\bullet})$ satisfying
\[
\widetilde \theta (e_1) = \widetilde \theta (e_2) = \widetilde \theta (-e_1-2e_2) =  g, \
\widetilde \theta (-e_1) = \widetilde \theta (-e_2) = \widetilde \theta (e_1+2e_2) =  -g, \ \text{and} \
\widetilde \theta (2e_2) = \widetilde \theta (-2e_2) = 2g \,.
\]
If
\[
A = e_1^{k_1}(-e_1)^{k_1'}e_2^{k_2} (-e_2)^{k_2'}(2e_2)^{k_3}(-2e_2)^{k_3'}(e_1+2e_2)^{k_4}(-e_1-2e_2)^{k_4'} \in \mathcal F (G_{\mathcal P}^{\bullet}) \,,
\]
with $k_1, k_1', \ldots, k_4, k_4' \in \mathbb N_0$, then $A \in \mathcal B (G_{\mathcal P}^{\bullet})$ if and only if
\[
k_1-k_1' + k_4-k_4' = 0 \quad \text{and} \quad k_2-k_2' + 2k_3-2k_3' + 2k_4-2k_4' = 0 \,.
\]
If this holds, then
\[
k_1-k_1' + k_2-k_2' - (k_4-k_4') + 2k_3+2k_3' \equiv 0 \mod 4
\]
and hence
\[
\widetilde \theta (A) = g^{k_1+k_2+k_4'}(-g)^{k_1'+k_2'+k_4}(2g)^{k_3+k_3'} \in \mathcal B (C_4^{\bullet}) \,.
\]
Thus $\theta = \widetilde \theta \mid_{\mathcal B (G_{\mathcal P}^{\bullet})} \colon \mathcal B (G_{\mathcal P}^{\bullet})  \to \mathcal B (C_4^{\bullet})$ is a monoid epimorphism satisfying  $\theta (W) = \overline W$, $\theta (-W) = - \overline W$, $\theta (V_1) = \theta (V_2) = \overline V$, $\theta (-V_1) = \theta (-V_2) = - \overline V$,  $\theta (U_1) = \theta (U_2) = \theta (U_3 ) = \overline{ U_1}$,  $\theta (U_4) = \overline{U_2}$, and $\theta^{-1} (1) = \{1\}  = \mathcal B (G_{\mathcal P}^{\bullet})^{\times}$.

Thus in order to show that $\theta$ is a transfer homomorphism, it remains to verify Property {\bf (T2)}. Let $A \in \mathcal B (G_{\mathcal P}^{\bullet})$ be as above and suppose that
\[
\theta (A) = \widetilde B \widetilde C
\]
with $\widetilde B, \widetilde C \in \mathcal B ( C_4^{\bullet})$ and $\widetilde B = g^m (-g)^{m'}(2g)^{m''}$ such that
\[
m \in [0, k_1+k_2+k_4'], \ m' \in [0, k_1'+k_2'+k_4], \ m'' \in [0, k_3+k_3'],  \ \text{and} \ m-m'+2m'' \equiv 0 \mod 4 \,.
\]
Our goal is to find $B, C \in \mathcal B (G_{\mathcal P}^{\bullet})$ such that $A = B C$, $\theta (B) = \widetilde B$, and $\theta (C) = \widetilde C$. It will suffice to find  $B \in \mathcal B (G_{\mathcal P}^{\bullet})$ with $B \t A$ and $\theta (B) = \widetilde B$.
Thus we must find parameters
\[
m_{\nu} \in [0, k_{\nu}] \ \text{and} \ m_{\nu}' \in [0, k_{\nu}'] \quad \text{for} \quad \nu \in [1,4]
\]
such that
\[
\begin{aligned}
m_1+m_2+m_4' & = m, \ m_1'+m_2'+m_4 = m', \ m_3+m_3' = m'', \\
m_1-m_1' + m_4-m_4' & = 0 \quad \text{and} \quad m_2-m_2'+2m_3-2m_3'+2m_4-2m_4' = 0 \,. \qquad \qquad {\bf (C2)}
\end{aligned}
\]
We proceed by induction on $|\widetilde B| = m+m'+m''$. If $|\widetilde B| = |A|$, then we set $m_{\nu} = k_{\nu}$ and $m_{\nu}' = k_{\nu}'$ for all $\nu \in [1,4]$, and the assertion is satisfied with $B =A$. Suppose now that the octuplet $(m_1, m_1', \ldots, m_4, m_4')$ satisfies {\bf(C2)} with respect to the triple $(m, m', m'')$.
Dividing $\widetilde B$ by an element of $\mathcal A ( C_4^{\bullet})$ (if possible) shows that we must verify  that there are solutions to {\bf (C2)} with respect to each of the triples
\[
(m-1, m'-1, m''),  (m-2,m', m''-1),
(m, m'-2, m''-1), (m,m', m''-2), (m-4, m', m''), \ \text{and} \ (m, m'-4, m'')
\]
provided that they lie in $\mathbb N_0^{(8)}$.
As in proof of the previous proposition, one finds the  required solutions  and hence the assertion follows by the induction hypothesis.
\end{proof}

\medskip
Let $\mathcal L$ be a family of subsets of $\mathbb Z$. We say that $\mathcal L$ is {\it additively closed} if the sumset $L+L' \in \mathcal L$ for all $L, L' \in \mathcal L$.

\medskip
\begin{proposition} \label{7.14}
Let $G$ be a finite cyclic group. Then $\mathcal L (G)$ is additively closed if and only if $|G| \le 4$.
\end{proposition}

\begin{proof}
We suppose that $|G| = n$ and distinguish four cases.

First assume that $n \le 2$. Since $\mathcal B (G)$ is factorial, it follows that $\mathcal L (G) = \{ \{m\} \, \colon m \in \mathbb N_0 \}$ which is obviously additively closed.

\smallskip
Next assume that $n=3$. By Proposition \ref{7.12} we have $\mathcal L (C_3)   = \bigl\{ y
      + 2k + P_k (1) \, \colon \, y,\, k \in \N_0 \bigr\}$. If $y_1, y_2, k_1, k_2 \in \mathbb N_0$. Then
\[
\big(y_1 + 2k_1 + P_{k_1} (1)    \big) +    \big(y_2 + 2k_2 + P_{k_2} (1)    \big) = (y_1+y_2) + 2(k_1+k_2) + P_{k_1+k_2}(1) \in \mathcal L (C_3) \,,
\]
and hence $\mathcal L (C_3)$ is additively closed.

\smallskip
Now assume that $n=4$. By Proposition \ref{7.13} we have
\[
\mathcal L (C_4)   =  \bigl\{ y + k+1 + P_k (1) \, \colon\, y,
      \,k \in \N_0 \bigr\} \,\cup\,  \bigl\{ y + 2k + P_k (2) \, \colon
      \, y,\, k \in \N_0 \bigr\}  \,.
\]
Clearly, the sumset of two sets of the first form is of the first form again, and the sumset of two sets of the second form again the second form.
Thus it remains to consider the sumset $L_1+L_2$ where $L_1$ has the first form, $L_2$ has the second form, and both $L_1$ and $L_2$ have more than one element. If $y_1, y_2 \in \mathbb N_0$ and $ k_1, k_2 \in \mathbb N$, then
\[
\big(y_1 + k_1+1 + P_{k_1} (1)    \big) +    \big(y_2 + 2k_2 + P_{k_2} (2)    \big) = (y_1+y_2)+(k_1+2k_2)+1+P_{k_1+2k_2}(1) \in \mathcal L (C_4) \,.
\]

\smallskip
Finally, assume that $n \ge 5$ and assume to the contrary that $\mathcal L (G)$ is additively closed.
Let $d \in [1, n-2]$. Then $\{2, d+2\} \in \mathcal L (G)$ by \cite[Theorem 6.6.2]{Ge-HK06a}, and hence the $k$-fold sumset
\[
\{2, d+2\} + \cdots + \{2, d+2\} =
2k +  P_k (d)
\]
lies in $\mathcal L (G)$ for all $k \in \mathbb N$. Then \cite[Corollary 4.3.16]{Ge-HK06a} implies that $n-3$ divides some $d \in \Delta^* (G)$. By Theorem \ref{7.9} we have
\[
\max \Delta^* (G) = n-2 \quad \text{and} \quad \max \bigl(
\Delta^* (G) \setminus \{n-2\} \bigr) = \left\lfloor \frac{n}{2}
\right\rfloor - 1 \,,
\]
a contradiction to $n \ge 5$.
\end{proof}

\begin{corollary} \label{7.15}~

\begin{enumerate}[(1)]
\item Let $H$ be an atomic monoid such that $\mathcal L (H)$ is additively closed, and let $H'$ be an atomic monoid with $\mathcal L (H') \subset \mathcal L (H)$. Then $\mathcal L (H \times H') = \mathcal L (H)$.

\smallskip
\item Let $H$ be an atomic monoid with $\mathcal L (H) = \mathcal L (C_n)$ for $n \in [3,4]$. For $k \in \mathbb N$  and $i \in [1,k]$, let $H_i$ be an atomic monoid with $\mathcal L (H_i) \subset \mathcal L (C_n)$. Then $\mathcal L (H \times H_1 \times \cdots \times H_k) = \mathcal L (C_n)$.
\end{enumerate}
\end{corollary}

\begin{proof}
Since $\mathcal L (H \times H') = \{L + L' \, \colon L \in \mathcal L (H), L' \in \mathcal L (H') \}$, (1) follows.

\smallskip
For (2), we set $H' = H_1 \times \cdots \times H_k$. Since $\mathcal L (C_n)$ is additively closed by Proposition \ref{7.14}, it follows that  $\mathcal L (H') = \{ L_1 + \cdots + L_k \, \colon L_i \in \mathcal L (H_i), i \in [1,k] \} \subset \mathcal L (C_n)$. Finally 1. implies that $\mathcal L (H \times H') = \mathcal L (H)$.
\end{proof}

\bigskip
We conclude this manuscript by suggesting a rich program for further study. Any progress in these directions will lead to a better understanding of direct-sum decompositions of classes of modules where each module has a semilocal endomorphism ring.
Moreover, this program could stimulate new studies in combinatorial Factorization Theory where much of the  focus has been  on Krull monoids having finite class group.

\smallskip
\noindent
{\bf Program for Further Study.}

\smallskip
\noindent
{\bf A. Module-Theoretic Aspect.} Let $R$ be a ring and let $\mathcal C$ be a class of right $R$-modules which is closed under finite direct sums, direct summands, and isomorphisms, and such that the endomorphism ring $\End_R (M)$ is semilocal for each module $M$ in $\mathcal C$ (such classes of modules  are presented in a systematical way in \cite{Fa04a}). Then $\mathcal V ( \mathcal C)$, the monoid of isomorphism classes of modules in $\mathcal C$ is a reduced Krull monoid with class group $G$ and set  $G_{\mathcal P} \subset G$ of classes containing prime divisors.

\smallskip
 Since the long-term goal --- to determine the characteristic of $\mathcal V(\mathcal C)$ --- is out of reach in most cases, the focus of study should be  on those properties of $G_{\mathcal P}$ which have most crucial influence on the arithmetic of direct-sum decompositions. In particular,
 \begin{itemize}
 \item Is $G_{\mathcal P}$ finite or infinite ?
 \item Is $G_{\mathcal P}$  well-structured in the sense of Proposition \ref{7.2} ?
 \end{itemize}

\medskip
\noindent{\bf B. Arithmetical Aspect of Direct-Sum Decompositions.} Let $H$ be a Krull monoid with finitely generated class group $G$ and let $G_{\mathcal P} \subset G$ denote the set of classes containing prime divisors.

\medskip
\noindent
{\bf 1.} {\it Finiteness conditions.}
 \begin{enumerate}
\item[(a)] Characterize the finiteness of arithmetical invariants (as introduced in Section \ref{3}) and the validity of structural finiteness results
       (as given in Proposition \ref{7.2}, 2.(a) and 2.(b)).

       \noindent
       For infinite cyclic groups much work in this direction can  be found  done in \cite{Ge-Gr-Sc-Sc10}.

\item[(b)] If $G_{\mathcal P}$ contains an infinite group, then every finite subset $L \subset \N_{\ge 2}$ occurs as a set of lengths in $H$ (see Proposition \ref{7.2}) and hence $\Delta (H) = \N$, and $\mathcal U_k (H) = \N_{\ge 2}$ for all $k \ge 2$. Weaken the assumption on $G_{\mathcal P}$ to obtain similar results.

    \noindent
    A weak condition on $G_{\mathcal P}$ enforcing that $\Delta (H) = \N$ can be found in \cite{Ha02c}.
\end{enumerate}

\medskip
\noindent
{\bf 2.} {\it Upper bounds and precise formulas.} Suppose that $G$ is torsion-free, say $G_{\mathcal P} \subset G = \Z^{(q)} \subset (\R^{(q)}, |\cdot|)$, where $|\cdot | \colon \R^{(q)} \to \R_{\ge 0}$ is a Euclidean norm.

\begin{enumerate}
\item[(a)] If $G_{\mathcal P} \subset \{ x \in \R \, \colon \, |x| \le M\}$ for some $M \in \N$, then derive upper bounds for the arithmetical invariants in terms of $M$.

\item[(b)] If $G_{\mathcal P}$ has a simple geometric structure (e.g., the set of vertices in a cube; see Examples \ref{cube-and-its-negative} and \ref{bigcube}), derive precise formulas for the arithmetical invariants, starting with the Davenport constant.

    \noindent
    A first result in this direction can be found in \cite{Ba-Ge-Gr-Sm14}.

\item[(c)]  Determine the extent to which  the arithmetic of a Krull monoid with $G_{\mathcal P}$ as in (b) is characteristic for $G_{\mathcal P}$. In particular, determine how this compares with the arithmetic of a   Krull monoid $H'$ where $G_{\mathcal P}'$ has the same  geometric structure as $G_{\mathcal P}$ with different parameters and how this compares with the arithmetic of  a Krull monoid having finite class group and prime divisors in all classes.
\end{enumerate}

\providecommand{\bysame}{\leavevmode\hbox to3em{\hrulefill}\thinspace}
\providecommand{\MR}{\relax\ifhmode\unskip\space\fi MR }
\providecommand{\MRhref}[2]{%
  \href{http://www.ams.org/mathscinet-getitem?mr=#1}{#2}
}
\providecommand{\href}[2]{#2}

\end{document}